\documentclass[11pt,a4paper]{article}
\usepackage[T1]{fontenc}
\usepackage{lmodern}
\usepackage[margin=1in]{geometry}
\usepackage{amsmath,amssymb,amsthm,mathtools,bm,mathrsfs}
\usepackage{slashed}
\usepackage{microtype,booktabs,array,longtable,needspace}
\usepackage{pdflscape}
\usepackage[hidelinks]{hyperref}
\numberwithin{equation}{section}
\newtheorem{theorem}{Theorem}[section]
\newtheorem{proposition}[theorem]{Proposition}
\newtheorem{lemma}[theorem]{Lemma}
\newtheorem{corollary}[theorem]{Corollary}
\newtheorem{maintheorem}{Theorem}

\theoremstyle{definition}

\theoremstyle{remark}
\newtheorem{remark}[theorem]{Remark}
\newtheorem{example}[theorem]{Example}
\newcommand{\dd}{\mathrm d}
\newcommand{\eps}{\epsilon}
\newcommand{\R}{\mathbb R}
\newcommand{\C}{\mathbb C}
\newcommand{\id}{\operatorname{id}}

\newcommand{\tr}{\operatorname{tr}}
\newcommand{\End}{\operatorname{End}}
\newcommand{\ad}{\operatorname{ad}}

\newcommand{\CS}{\operatorname{CS}}
\newcommand{\sg}{\mathfrak s}
\newcommand{\Bsp}{\beta_\Delta}
\newcommand{\red}[1]{{{#1}^{[0]}}}

\newcommand{\cS}{\mathcal S}
\newcommand{\cT}{\mathcal T}

\newcommand{\cL}{\mathcal L}

\newcommand{\ph}{\mathrm{ph}}
\newcommand{\LC}{\mathrm{LC}}

\DeclareMathOperator{\vol}{vol}
\DeclareMathOperator{\Spin}{Spin}
\DeclareMathOperator{\SO}{SO}

\allowdisplaybreaks[1]

\usepackage{etoolbox}
\makeatletter
\patchcmd{\LT@output}{\vss}{\vfil}{}{}
\patchcmd{\LT@output}{\vss}{\vfil}{}{}
\makeatother
\hypersetup{pdftitle={Cyclic $L_3$ Models and Holomorphic Reduction in Heterotic $G_2$ Deformation Theory},pdfauthor={Bram Brongers}}
\title{Cyclic $L_3$ Models and Holomorphic Reduction in Heterotic $G_2$ Deformation Theory}
\author{Bram Brongers\\
\href{mailto:brambrongers98@gmail.com}{\texttt{brambrongers98@gmail.com}}}
\date{}
\begin{document}
\maketitle
\begin{abstract}
We construct a finite cyclic $L_3$ presentation 
(that is, a cyclic $L_\infty$ algebra with brackets vanishing above arity three)
of the variational deformation theory defined by the heterotic \(G_2\) superpotential 
near a torsion-free standard embedding, to first
order in $\alpha'$. Auxiliary fields for the inverse coframe, spinor
derivative, flux and induced tangent connection give a shifted-cotangent
Hamiltonian of degree at most four. Eliminating the auxiliary fields
recovers the physical functional. On the flat seven-torus, this functional
has a nonzero quintic Taylor coefficient in local physical spinor coordinates.

For $Y=S^1\times X$, with $X$ a Calabi--Yau threefold, we restrict to
the product locus with horizontal gauge and gerbe data. The functional
reduces to
\[
 \frac L4\operatorname{Im}\int_X(H+i\,d\omega)\wedge\Omega,
\]
where $L$ is the circle length and $\Omega=e^{-2\Phi}\Psi$ is the
weighted complex volume. The induced $SU(3)$ structure selects a
Hermitian tangent Courant subalgebroid. Transverse isotropic lifts in
its complexification encode the holomorphic connection and string
equations, together with the induced tangent connection, weighted
canonical line and gerbe hierarchy. Their Taylor brackets form an
$L_3$ algebra whose string part agrees with the known holomorphic
heterotic algebra.

We then show that this holomorphic theory is obtained by characteristic
Hamiltonian reduction of a finite contractible stabilization of the product
variational theory, in every
cochain degree. While the quotient captures the holomorphic equations,
conformal balance, gauge primitivity and leading-flux admissibility 
are additional conditions on a real representative.
\end{abstract}
\clearpage
\setcounter{tocdepth}{1}
\begingroup\small
\tableofcontents
\endgroup
\clearpage
\section{Introduction and main results}
\label{sec:introduction}

The heterotic $G_2$ superpotential couples the geometry of a $G_2$
structure to gauge connections and flux \cite{G2Superpotential}.
Its string Courant formulation gives a variational deformation theory
near the standard embedding \cite{BrongersPaperI}. On a Calabi--Yau
threefold, finite deformations of the holomorphic heterotic system are
governed by an $L_3$ algebra \cite{FiniteDeformations}. This raises the question of how
the seven-dimensional variational theory is related to the holomorphic
deformation theory obtained after reduction on $S^1\times X$. We answer
this question in two stages. First, we construct a finite cyclic $L_3$ presentation
of the seven-dimensional theory. We then show that, on $S^1\times X$,
the Calabi--Yau holomorphic theory arises from it by characteristic
Hamiltonian reduction.

The seven-dimensional functional depends on an inverse metric, a
normalized spinor and the Hull connection induced by the metric and
flux. These expressions need not be polynomial in physical
coordinates. We introduce independent auxiliary fields and impose
their defining relations at first differential order. An axial
torsion identity makes the flux constraint cubic, so the
shifted-cotangent Hamiltonian is quartic. Its Taylor brackets at
the standard embedding therefore form a cyclic $L_3$ algebra.
The auxiliary equations have a local formal inverse, and their
Hamiltonian theory is a finite contractible stabilization of the
physical variational theory.

On the product locus in $S^1\times X$, circle reduction gives a
six-dimensional string Courant algebroid with both tangent and
gauge adjoint bundles. A connection preserving the induced
Hermitian tangent structure then defines a Courant subalgebroid.
Transverse isotropic lifts in its complexification encode the
holomorphic connection and string equations. Their deformation
algebra also includes the tangent partial connection induced by the
complex distribution, the weighted canonical line and the gerbe
gauge hierarchy. The final comparison identifies this algebra as a
characteristic quotient of the stabilized product variational
theory.

\subsection{Setting and main results}
Let $Y$ be a closed spin seven-manifold with a fixed torsion-free
standard embedding. We use the coefficient convention of
\cite[Section~2.4]{BrongersPaperI}:
\[
 \mathbb D=\mathbb R[\epsilon]/(\epsilon^2),\qquad
 \epsilon=\alpha',\qquad
 \mathbb D_0=\mathbb D/(\epsilon).
\]
Geometric fields are $\mathbb D$-valued, while gauge fields are
$\mathbb D_0$-valued,\footnote{At this order, the gauge connection is
kept only at leading order in $\alpha'$. An $\epsilon$-order correction
to the gauge connection contributes to the $\alpha'$-weighted anomaly
terms only at order $\epsilon^2$, while the gauge instanton equation is
imposed on its leading coefficient.} and their variational duals have
coefficients in $\epsilon\mathbb D$.

For the product reduction, let $Y=S^1\times X$, where $X$ is a closed
Calabi--Yau threefold. We choose the circle spin structure for which the
background parallel spinor is invariant, fix the standard-embedding
bundle identifications, and use a gerbe presentation pulled back from
$X$. Whenever a nondegenerate Courant pairing is required, we first work
at nonzero real anomaly coefficient $z$ and then specialize the resulting
regular identities to $z=\epsilon$.

\begin{maintheorem}[Finite cyclic presentation]
\label{thm:main-composite}\label{thm:intro-finite}
There is a finite first-order Hamiltonian resolution
\[
(\mathcal M,\omega,S,Q)
\]
of the heterotic $G_2$ variational theory such that
\[
\deg S\leq 4,
\qquad
S=S_2+S_3+S_4,
\qquad
Q=Q_1+Q_2+Q_3
\]
at the standard embedding. Hence its Taylor brackets satisfy
\[
\ell_n=0\qquad (n\geq 4)
\]
and define a cyclic $L_3$ algebra.

Moreover, the resolution is locally Hamiltonian $Q$-isomorphic
to the physical-coordinate variational theory times a finite contractible
Hamiltonian factor $\mathcal K$.
On the flat seven-torus, the physical functional obtained after eliminating
the auxiliary fields has a nonzero quintic Taylor coefficient at the
standard embedding in the local physical spinor coordinates of
Section~\ref{subsec:quintic-coefficient}.
\end{maintheorem}

The resolution and degree bound are constructed in
Sections~\ref{sec:finite-resolution}--\ref{sec:cyclic-l3};
Theorem~\ref{thm:parent-equivalence} proves Hamiltonian
stabilization.

Let $\eta$ be the unit circle one-form, $L=\int_{S^1}\eta$,
$(\omega,\Psi)$ the induced $SU(3)$ structure and
$\Omega=e^{-2\Phi}\Psi$. The product locus has unit radius and
horizontal gauge and gerbe data.

\begin{maintheorem}[Product and holomorphic reduction]
\label{thm:intro-product}
For $Y=S^1\times X$ and the product locus, isotropic circle reduction gives
\[
 \mathbb E_X
 \simeq
 TX\oplus\ad P_T\oplus\ad P_G\oplus T^*X,
\]
and
\[
 W_{\mathrm{Courant},G_2}\big|_{\rm product}
 =
 \frac L4\operatorname{Im}
 \int_X(H+i\,\dd\omega)\wedge\Omega.
\]

If $P_H\subset P_T$ is the Hermitian tangent reduction and $\Gamma$
is a compatible connection, then
\[
 \mathbb E_H
 =
 TX\oplus\ad P_H\oplus\ad P_G\oplus T^*X
 \subset \mathbb E_X
\]
is a Courant subalgebroid. The transverse isotropic-lifting problem in
$\mathbb E_{H,\mathbb C}$, with tangent connection
\[
 \kappa_e=K_{\Gamma_0}(\mu),\qquad
 \kappa=\kappa_e^{[0]},
\]
weighted canonical line and gerbe hierarchy, is governed by a
homological vector field
\[
 Q_N=Q_1+Q_2+Q_3,\qquad Q_N^2=0.
\]
Hence its Taylor brackets satisfy
\[
 \ell_n=0\qquad(n\geq4).
\]
Before imposing the tangent graph and adjoining the canonical-line
component, the string operations agree with those of
\cite{FiniteDeformations}, after the coefficient and trace normalizations
of Section~\ref{subsec:unary-structure}:
\[
 \ell_1=-\ell^{\rm FD}_1,\qquad
 \ell_2=\ell^{\rm FD}_2,\qquad
 \ell_3=\ell^{\rm FD}_3.
\]

For the coefficient-restored theory at nonzero real anomaly coefficient,
every Maurer--Cartan solution defines
an involutive transverse lift $L\subset\mathbb E_{H,\mathbb C}$ and
therefore a holomorphic string algebroid (Corollary~\ref{cor:hol-string-algebroid})
\[
 \mathcal Q_L=L^\perp/L,
 \qquad
 \dd\Omega=0.
\]
\end{maintheorem}

Section~\ref{sec:product-reduction} proves the product reduction
and identifies the physical equations. On the product locus,
volume closure, gauge type and the torsion equation take the form
$\dd\Omega=0$, $F_A^{0,2}=0$ and $H=-\dd^c\omega$.
Balance, primitivity and $H^{[0]}=0$ are additional physical
conditions (Proposition~\ref{prop:product-bps-dictionary}).
This separation motivates the holomorphic construction in
Section~\ref{sec:holomorphic-lifting}.
Section~\ref{subsec:unary-structure} identifies its tangent complex
and gives a constant-mode example with nonzero ternary bracket.

The physical product functional is written using the Hull connection, whereas
the holomorphic construction uses a Hermitian tangent representative. Passing
between these representatives changes both the connection and the gerbe
coordinate and produces a first-order change in the functional, which is
removed by a finite-jet correction. This gives compatible representatives in
which the full graded Hamiltonian theories can be compared.

The final step is a characteristic reduction. If
\(i:C\hookrightarrow(M,\omega)\) is a coisotropic inclusion, a projection
\(\pi:C\to M_{\rm red}\) is a characteristic quotient when
\[
\ker D\pi=\ker i^*\omega.
\]

\begin{maintheorem}[Hamiltonian realization]
\label{thm:intro-hamiltonian}
After finite contractible Hamiltonian stabilization, the product
variational theory has a local real formal Hamiltonian presentation
$(\widehat{\mathcal M},\widehat\omega,\widehat S)$ and a preserved
coisotropic subspace $i:C\hookrightarrow \widehat{\mathcal M}$ with a
characteristic projection $\pi:C\to\mathcal M_N$ such that
\[
 i^*\widehat\omega=\pi^*\omega_N,\qquad
 i^*\widehat S=\pi^*I_N,\qquad
 \ker i^*\widehat\omega=\ker D\pi.
\]
Here $(\omega_N,I_N)$ is the real variational Hamiltonian structure
obtained by applying $\frac L4\operatorname{Im}$ to the complex symplectic
form and Hamiltonian of the holomorphic lifting theory. Its characteristic quotient carries
$Q_N$ in every cochain degree.
\end{maintheorem}

This is proved in Section~\ref{sec:hamiltonian-reduction} as
Theorem~\ref{thm:hamiltonian-realization}. The preceding coordinate
maps include the Hull--Hermitian representative change, a cotangent
polarization, a Moser transformation and a canonical transformation
from tangent transgression. The characteristic quotient need not
preserve cohomology: the real product and holomorphic theories have
different constant scalar reducibility sheaves.

\subsection{Relation to previous work}
The superpotential of de la Ossa--Larfors--Magill--Svanes is the
variational starting point \cite{G2Superpotential}. The infinitesimal
heterotic $G_2$ complex and its restriction to Strominger--Hull
geometry are studied in \cite{DLSHeterotic,DLSRestrictions}; the
physical moduli, metric and dimensional reduction are treated in
\cite{MSS,MSSMetric}. The second-variation complex in
\cite{QuantumG2} and the $G$-structure BPS complexes in \cite{BPS}
concern related linearized theories. 
The physical and variational constructions used here are established in \cite{BrongersPaperI}. 

The Courant Lie-$2$ symmetry is due to Roytenberg--Weinstein
\cite{CourantLie2}. We use generalized Dirac operators
\cite{GFDirac} and isotropic Courant reduction \cite{BCGReduction}.
The auxiliary construction is related to parent formulations
\cite{GrigorievParent,GrigorievRudinsky}; its quartic bound depends
on the axial torsion identity and the mixed coefficient convention.

Ashmore--de la Ossa--Minasian--Strickland-Constable--Svanes constructed
the finite heterotic $SU(3)$ deformation algebra
\cite{FiniteDeformations}. We derive its string brackets from
transverse isotropic lifts, impose the induced tangent graph,
adjoin the canonical-line module and prove the Hamiltonian
realization from seven dimensions.
Garc\'ia-Fern\'andez--Rubio--Tipler constructed holomorphic string
algebroids and studied their gauge theory
\cite{StringAlgebroids,GaugeStringAlgebroids}. These algebroids
arise here from involutive lifts with the fixed bundle and gerbe
data. The physical comparison uses the leading-flux restriction of
\cite[Section~2.5]{BrongersPaperI}.

\section[Heterotic G2 background and conventions]{Heterotic $G_2$ background and conventions}\label{sec:courant-g2}\label{sec:geometry}
We recall the background data and the functional from
\cite[Section~2]{BrongersPaperI} that enter the auxiliary resolution.

Let $Y$ be a closed spin seven-manifold, and let $P_T$ and $P_G$ be
principal $SO(7)$-bundles. Write
\[
E_T=P_T\times_{SO(7)}\mathbb R^7,
\qquad
E_G=P_G\times_{SO(7)}\mathbb R^7,
\]
and fix isometric identifications \(E_T \simeq E \simeq E_G \)
at the standard embedding, together with the spin lift of $P_T$.

We also fix a differential trivialization of the anomaly class
determined by $(P_T,P_G)$. On a good cover $\{U_i\}$, choose a
Čech--Deligne representative
\[
\check B=(B_i,\Lambda_{ij},g_{ijk}),
\qquad
B_i\in\Omega^2(U_i),\quad
\Lambda_{ij}\in\Omega^1(U_{ij}),\quad
g_{ijk}:U_{ijk}\to U(1).
\]
The symbol $B$ denotes the local two-form coordinate $B_i$ in a chosen
trivialization. Changes of $(B_i,\Lambda_{ij},g_{ijk})$ are the gerbe
gauge transformations. Fix also the string Courant extension below,
whose anomaly polynomial uses the invariant form
\[
 P(u,v)=-\tr_7(uv).
\]
We work formally near the torsion-free standard embedding
$E\simeq(TY,g_0)$ and $A_0=\Theta_0$, the Levi--Civita connection. For
$z\in\mathbb R^\times$, put $\mathfrak k=\ad P_T\oplus\ad P_G$ and
\[
 \mathfrak c_z((u_T,u_G),(v_T,v_G))
 =\frac z4\bigl(P(u_T,v_T)-P(u_G,v_G)\bigr).
\]
The string Courant extension is
\begin{equation}\label{eq:ii-courant-extension}
0\longrightarrow T^*Y\xrightarrow{j}\mathbb E
 \longrightarrow\operatorname{At}(P_T\times_Y P_G)\longrightarrow0.
\end{equation}
In a connection splitting its anchor, pairing and anomaly identity are
\begin{equation}\label{eq:ii-courant-pairing}
\begin{split}
a_{\mathbb E}(X+u+\eta)&=X,\\
\langle X+u+\eta,Y+v+\zeta\rangle
 &=\tfrac12\bigl(\zeta(X)+\eta(Y)\bigr)+\mathfrak c_z(u,v),\\
dH&=\mathfrak c_z(F_C,F_C),\qquad C=(C_T,A).
\end{split}
\end{equation}
Here $\operatorname{At}$ is the Atiyah algebroid, $F_C$ the curvature and
$H$ the splitting three-form \cite[(2.3)--(2.4)]{BrongersPaperI}.

An admissible metric $V_+\subset\mathbb E$ determines $g$, an isotropic
splitting $s$, $C$ and $H$ by
\cite[Definition~2.1 and Proposition~2.2]{BrongersPaperI}. A spin-compatible
isometry $e:(TY,g)\to E$ identifies the tangent bundles. We impose the
Hull condition \cite[(2.7)]{BrongersPaperI}:
\begin{equation}\label{eq:ii-hull-compatibility}
e^{-1}\nabla_X^{C_T}(eY)=\nabla_X^gY-\tfrac12g^{-1}H(X,Y,\cdot).
\end{equation}
For a compatible unit spinor line $[\chi]\subset S(V_+)$ and positive
density $\varrho$, set
\begin{equation}\label{eq:ii-density}
\varrho=e^{-2\Phi}\vol_g,\qquad\zeta_{\rm ph}=e^{-\Phi}\chi.
\end{equation}
Write $\varphi$ for the positive form and $\psi=*_{g}\varphi$. Physical
splitting coordinates are $(\varphi,B,\Phi,A)$, with $B$ in the fixed
differential-gerbe presentation.

The first-order coefficients are \cite[Section~2.4]{BrongersPaperI}
\begin{equation}\label{eq:ii-mixed-coefficients}
\begin{gathered}
\mathbb D=\mathbb R[\epsilon]/(\epsilon^2),\qquad\epsilon=\alpha',\qquad
\mathbb D_0=\mathbb D/(\epsilon),\\
\jmath_\epsilon:\mathbb D_0\longrightarrow\epsilon\mathbb D,
\qquad[x]\longmapsto\epsilon x.
\end{gathered}
\end{equation}
Geometric variables are $\mathbb D$-valued and leading gauge variables
$\mathbb D_0$-valued. Every argument of an explicit $\epsilon$ is first
reduced modulo $\epsilon$, denoted $q^{[0]}$. Whenever a
nondegenerate Courant pairing is required, we establish the identity for
$z\ne0$ and then specialize polynomially to $z=\epsilon$ with the mixed
coefficient modules.

We use the canonical Courant Lie-$2$ action and relative Green--Schwarz
hierarchy \cite[Appendix~D.1--D.2]{BrongersPaperI}, with left graded
derivations and the pairing \eqref{eq:ii-courant-pairing}.
Local formal maps, variational duals and shifted-cotangent constructions
use \cite[Section~3 and Appendix~C]{BrongersPaperI}.

Let $D_g$ be the spin Dirac operator, $c_g$ Clifford multiplication and
$\beta_\Delta$ the symmetric real spinor pairing, with the signs of
\cite[Section~2.2--2.3 and Appendix~A]{BrongersPaperI}. The canonical
positive generalized Dirac operator determined by $(V_+,\operatorname{div}_{\varrho})$ is
\begin{equation}\label{eq:ii-generalized-dirac}
\slashed D^+=D_g+\tfrac14c_g(H)-c_g(d\Phi).
\end{equation}
On the Hull-compatible configuration space its expectation value is
\begin{equation}\label{eq:ii-dirac-functional}
\begin{split}
W_{\mathrm{Courant},G_2}
 &=\int_Y\varrho\,\beta_\Delta(\chi,\slashed D^+\chi)=W_{\rm ext}\\
 &=\frac14\int_Ye^{-2\Phi}
   \left(H\wedge\psi-\frac12d\varphi\wedge\varphi\right)
\end{split}
\end{equation}
by \cite[Propositions~4.1--4.2]{BrongersPaperI}. The corresponding physical
variational theory is constructed in \cite[Section~4]{BrongersPaperI}.

\section{A finite first-order resolution}\label{sec:finite-resolution}\label{sec:first-order}
The physical functional depends nonpolynomially on the inverse coframe,
spinor normalization and induced Hull connection. We replace these
expressions by independent variables and impose their defining relations.
Using the axial torsion three-form in the flux constraint keeps the
multiplier term of degree at most four.

\subsection{Fields and constraints}

Let $\Delta$ be the real spinor bundle associated to the fixed bundle $E$.
Write $\Bsp$ for the transported symmetric spinor pairing.
The field $e$ determines
\[
 g_e=e^*\langle-,-\rangle_E.
\]
We write $\vol_e$ for its volume form and $\Theta_{\LC}(e)$ for its
Levi--Civita connection transported to $E$ by $e$.

The variables $(e,\zeta,B,\nu,A)$ represent the physical coframe, spinor,
gerbe, density and gauge data. The variables $(U,\Xi,\Theta,h)$ are auxiliary
and represent the densitized inverse coframe, spinor derivative, induced
tangent connection and flux. Their coefficient modules are listed in
Table~\ref{tab:first-order-fields}. We use the formal open locus where
$\nu^{[0]}$ is positive and $\zeta^{[0]}$ is nowhere zero.
\begin{table}[ht]
\centering
\begin{tabular}{@{}lp{10.5cm}@{}}
\toprule
Field & Space\\
\midrule
$e$
& $\{e\in\Omega^1(E)\otimes\mathbb D:
   e^{[0]}:TY\to E\text{ is a fibrewise isomorphism}\}$
\\[1mm]
$\zeta$ & $\Gamma(\Delta)\otimes\mathbb D$\\[1mm]
$B$ & $\Omega^2(Y)\otimes\mathbb D$\\[1mm]
$\nu$ & $\Omega^7(Y)\otimes\mathbb D$\\[1mm]
$A$ & $A_0+\Omega^1(\mathfrak{so}(E))\otimes\mathbb D_0$\\[1mm]
\addlinespace
$U$ & $\Omega^6(E^*)\otimes\mathbb D$\\[1mm]
$\Xi$ & $\Omega^1(\Delta)\otimes\mathbb D$\\[1mm]
$\Theta$ & $\Theta_0+\Omega^1(\mathfrak{so}(E))\otimes\mathbb D$\\[1mm]
$h$ & $\Gamma(\Lambda^3E^*)\otimes\mathbb D$\\
\bottomrule
\end{tabular}
\caption{Independent fields of the first-order presentation.}
\label{tab:first-order-fields}
\end{table}

Since $A$ and $\Theta^{[0]}$ are connections on the same $SO(E)$-bundle,
their difference is
\[
 a:=A-\Theta^{[0]}
 \in\Omega^1(\mathfrak{so}(E))\otimes\mathbb D_0.
\]

For $h\in\Gamma(\Lambda^3E^*)$, define $K_e(h)\in\Omega^2(E)$ by
\begin{equation}\label{eq:Ke-definition}
 \bigl\langle K_e(h)(X,Y),s\bigr\rangle_E=h(s,eX,eY)
\end{equation}
for $X,Y\in TY$ and $s\in E$. We use the fibre metric to pair vector-bundle
valued forms by exterior multiplication and fibre contraction. On exterior
powers of $E^*$ we use the induced metric pairing with the standard $1/p!$
normalization in degree $p$.

For a connection $\theta$ and adjoint-valued one-form $a$, the relative
transgression is
\begin{equation}\label{eq:ii-transgression-convention}
 \cT_3(\theta,a)=2P(a,R_\theta)+P(a,D_\theta a)
                         +\tfrac13P(a,[a,a]),
\end{equation}
where $R_\theta$ is the curvature, as in \cite[(2.14)]{BrongersPaperI}.

Define the four constraints by
\begin{equation}\label{eq:first-order-constraints}
\begin{aligned}
 C_U&:=e\wedge U-\nu\,\id_E
 &&\in\Omega^7(\End E)\otimes\mathbb D,\\
 C_\Xi&:=\Xi-D_\Theta\zeta
 &&\in\Omega^1(\Delta)\otimes\mathbb D,\\
 T&:=D_\Theta e+K_e(h)
 &&\in\Omega^2(E)\otimes\mathbb D,\\
 C_H&:=-\frac13\langle e\wedge D_\Theta e\rangle_E-dB
       +\frac{\epsilon}{4}\cT_3(\Theta^{[0]},a)
 &&\in\Omega^3(Y)\otimes\mathbb D.
\end{aligned}
\end{equation}
Here $D_\Theta$ is the induced covariant exterior derivative, and we write
$C=(C_U,C_\Xi,T,C_H)$. The equations $C_U=C_\Xi=0$ identify $U$ with the
densitized inverse coframe and $\Xi$ with the spinor derivative. The equation
$T=0$ determines the tangent connection with torsion $-e^*h$, while $C_H=0$
identifies its axial three-form with the relative Chern--Simons flux.

The corresponding constraint multipliers are
\[
\begin{aligned}
 m_U&\in\Omega^0((\End E)^*)\otimes\mathbb D,
 &m_\Xi&\in\Omega^6(\Delta^*)\otimes\mathbb D,\\
 m_T&\in\Omega^5(E^*)\otimes\mathbb D,
 &m_H&\in\Omega^4(Y)\otimes\mathbb D.
\end{aligned}
\]

The three-form represented by $h$ and $e$ is
\begin{equation}\label{eq:Hh-definition}
 H_h:=e^*h\in\Omega^3(Y)\otimes\mathbb D.
\end{equation}
For the contorsion map
$C_g:\Omega^3(Y)\to\Omega^1(\mathfrak{so}(TY,g))$, we use
\cite[(2.15)]{BrongersPaperI}:
\begin{equation}\label{eq:contorsion}
 g\bigl(Y,C_g(H)_X Z\bigr)=\tfrac12H(X,Y,Z).
\end{equation}
Let $C_e(H)$ be $C_{g_e}(H)$ transported from $TY$ to $E$ through $e$:
\[
 C_e(H)_X=e\circ C_{g_e}(H)_X\circ e^{-1}.
\]

\begin{lemma}[Axial contraction]
\label{lem:axial-contraction}
For the independent fields,
\begin{equation}\label{eq:axial-identity}
 -\frac13\langle e\wedge D_\Theta e\rangle_E
 =H_h-\frac13\langle e\wedge T\rangle_E.
\end{equation}
Consequently, the equations $T=0$ and $C_H=0$ are equivalent to
\begin{equation}\label{eq:determined-graph-small}
\begin{aligned}
 \Theta^u&=\Theta_{\LC}(e)+C_e(dB),\\
 H&=dB-\frac{\epsilon}{4}
 \cT_3\bigl((\Theta^u)^{[0]},A-(\Theta^u)^{[0]}\bigr),\\
 \Theta&=\Theta_{\LC}(e)+C_e(H),
 &H_h&=H.
\end{aligned}
\end{equation}
\end{lemma}

\begin{proof}
By \eqref{eq:Ke-definition},
\begin{equation}\label{eq:axial-Ke}
 \langle e\wedge K_e(h)\rangle_E=3e^*h=3H_h.
\end{equation}
Since $T=D_\Theta e+K_e(h)$, this gives
\[
 \langle e\wedge T\rangle_E
 =\langle e\wedge D_\Theta e\rangle_E+3H_h,
\]
which is \eqref{eq:axial-identity}.

For fixed $e$, define
\[
 L_e:\Omega^1(\mathfrak{so}(E))\longrightarrow\Omega^2(E),
 \qquad L_e(k)=k\wedge e,
\]
where
\[
 (k\wedge e)(X,Y)=k_X(eY)-k_Y(eX).
\]
This is a pointwise isomorphism; its explicit inverse is recorded in
Lemma~\ref{lem:constraint-inverse}. The Levi--Civita connection satisfies
$D_{\Theta_{\LC}(e)}e=0$. If
\[
 \widehat H:=(e^{-1})^*H\in\Gamma(\Lambda^3E^*),
 \qquad e^*\widehat H=H,
\]
then the contorsion convention \eqref{eq:contorsion} gives
\[
 L_e(C_e(H))=-K_e(\widehat H).
\]
Hence $T=0$ is equivalent to
\[
 \Theta=\Theta_{\LC}(e)+C_e(H_h).
\]

When $T=0$, equation \eqref{eq:axial-identity} reduces the flux constraint to
\[
 H_h=dB-\frac{\epsilon}{4}
 \cT_3(\Theta^{[0]},A-\Theta^{[0]}).
\]
Reduction modulo $\epsilon$ gives $H_h^{[0]}=dB^{[0]}$, and therefore
\[
 \Theta^{[0]}
 =\bigl(\Theta_{\LC}(e)+C_e(dB)\bigr)^{[0]}=(\Theta^u)^{[0]}.
\]
Substitution in the transgression term gives
\eqref{eq:determined-graph-small}.
\end{proof}

For fixed $(e,\zeta,B,\nu,A)$, the constraint map is locally formally
invertible in $(U,\Xi,\Theta,h)$ by Lemma~\ref{lem:constraint-inverse}.

\subsection{The polynomial functional and scale symmetry}

Since $\vol_e$ is nowhere vanishing, there is a unique function $\rho$ such
that $\nu=\rho\,\vol_e$. On $C_U=0$,
\begin{equation}\label{eq:densitized-frame}
 U(s)=\iota_{e^{-1}s}\nu
 \qquad(s\in\Gamma(E)).
\end{equation}
Thus $U$ is the densitized inverse coframe determined by $(e,\nu)$.

The fields
\[
 (\nu,U,\zeta,\Xi)
 \in\Omega^7\oplus\Omega^6(E^*)\oplus\Gamma(\Delta)\oplus\Omega^1(\Delta)
\]
carry a local rescaling symmetry. Let $u\in C^\infty(Y,\mathbb D)^\times$
have positive leading coefficient. The formal functions $u^{1/2}$,
$u^{-1/2}$ and $\log u$ are taken on the branches through their background
values. The rescaling action is defined by
\begin{equation}\label{eq:finite-scale}
\begin{aligned}
 \nu'&=u\nu,&U'&=uU,\\
 \zeta'&=u^{-1/2}\zeta,
 &\Xi'&=u^{-1/2}\left(\Xi-\frac12d\log u\,\zeta\right).
\end{aligned}
\end{equation}
All other base fields are fixed. The induced action on the constraint
multipliers is included in the symmetry differential below.

The combinations
\begin{equation}\label{eq:physical-scale-invariants}
 \zeta_{\ph}=\rho^{1/2}\zeta,
 \qquad
 \varrho=\Bsp(\zeta,\zeta)\nu,
 \qquad
 \chi=\frac{\zeta}{\sqrt{\Bsp(\zeta,\zeta)}}
\end{equation}
are invariant under \eqref{eq:finite-scale}, since
$\rho'=u\rho$ and $\Bsp(\zeta',\zeta')=u^{-1}\Bsp(\zeta,\zeta)$. The square
roots use the branches fixed above.

We use polynomial degree in the jet coordinates centred at the background,
with connection variables represented by their displacements from the
background connections. The axial form of the flux constraint in
\eqref{eq:first-order-constraints} is chosen so that the resulting
functional has bounded degree. Indeed, imposing instead
\[
 H_h=dB-\frac{\epsilon}{4}\cT_3
\]
would introduce the term \(H_h=e^*h\), which is quartic in \((e,h)\);
pairing this equation with a multiplier would therefore give a quintic
term. By contrast,
\[
 C_H
 =
 -\frac13\langle e\wedge D_\Theta e\rangle_E
 -dB+\frac{\epsilon}{4}\cT_3
\]
has degree at most three.

Let
\[
 c_E:\Lambda^\bullet E^*\longrightarrow \End(\Delta)
\]
denote Clifford multiplication on the fixed spinor bundle, with the
conventions of \cite[Section~2.3]{BrongersPaperI}. For
\(\zeta\in\Gamma(\Delta)\otimes\mathbb D\), define
\[
 \Phi_\zeta\in\Gamma(\Lambda^3E^*)\otimes\mathbb D,
 \qquad
 \Phi_\zeta(s_1,s_2,s_3)
 =
 \Bsp\!\left(
   \zeta,
   c_E(s_1^\flat\wedge s_2^\flat\wedge s_3^\flat)\zeta
 \right),
\]
and for
\(\Xi\in\Omega^1(\Delta)\otimes\mathbb D\), define
\[
 J_\zeta(\Xi)\in\Omega^1(E^*)\otimes\mathbb D
\]
by
\[
 \bigl\langle J_\zeta(\Xi)(X),s\bigr\rangle
 =
 \Bsp\!\left(\zeta,c_E(s^\flat)\Xi_X\right).
\]
The resolved functional is
\begin{equation}\label{eq:first-order-source}
\begin{split}
 W_{\rm aux}=\int_Y\bigl[
   &\langle J_\zeta(\Xi)\wedge U\rangle_E
    +\langle h,\Phi_\zeta\rangle_E\,\nu\\
   &+\langle m_U,C_U\rangle
    +\langle m_\Xi,C_\Xi\rangle
    +\langle m_T,T\rangle
    +m_H\wedge C_H
 \bigr].
\end{split}
\end{equation}
Each term has polynomial degree at most four. Hence \(\deg W_{\rm aux}\leq 4\).

\begin{proposition}[Agreement with the geometric functional]\label{prop:source-identity}
On the constraint graph,
\begin{equation}\label{eq:source-on-graph}
 W_{\rm aux}\big|_{C=0}=\int_Y
 \left[\Bsp(\zeta_\ph,D_g\zeta_\ph)
       +\tfrac14\Bsp(\zeta_\ph,H\mathbin{\cdot}_g\zeta_\ph)\right]\vol_g.
\end{equation}
\end{proposition}
\begin{proof}
Set $g=g_e$ and let $\nabla^\Theta$ be the induced spinor connection. On the
graph, $U(s)=\iota_{e^{-1}s}\nu$ and $\Xi=D_\Theta\zeta$. Hence
\[
 \langle J_\zeta(D_\Theta\zeta)\wedge U\rangle_E
 =\Bsp(\zeta,\mathscr D_\Theta\zeta)\nu,
 \qquad \mathscr D_\Theta=\mathrm{cl}_g\circ\nabla^\Theta,
\]
using $\alpha\wedge\iota_v\nu=\alpha(v)\nu$. By
\eqref{eq:determined-graph-small},
$\Theta=\Theta_{\LC}(e)+C_e(H)$ and $H_h=H$. The spinorial contorsion identity
\cite[Section~2.3 and Appendix~A.3]{BrongersPaperI} gives
$\mathscr D_\Theta=D_g-3c_g(H)/4$, while
\[
 \langle h,\Phi_\zeta\rangle_E\nu
 =\Bsp(\zeta,c_g(H)\zeta)\nu.
\]
The sum is
$[\Bsp(\zeta,D_g\zeta)+\Bsp(\zeta,c_g(H)\zeta)/4]\nu$. Substituting
$\nu=\rho\vol_g$ and $\zeta_\ph=\rho^{1/2}\zeta$ gives
\eqref{eq:source-on-graph}. The derivative of $\rho^{1/2}$ has zero pairing
by \cite[(A.2)]{BrongersPaperI}. The generalized Dirac identity
\cite[Propositions~4.1--4.2]{BrongersPaperI} identifies this expression with
$W_{\rm ext}$.
\end{proof}

\subsection{Symmetry differential}\label{sec:auxiliary-symmetry}
Let $\mathscr B$ be the graded formal base obtained by adjoining the
symmetry coordinates in the following table to the fields and constraint
multipliers above:
\begin{center}
\begin{tabular}{lllc}
\toprule
Ghost&Bundle&Symmetry&Ghost number\\
\midrule
$\xi$&$TY\otimes\mathbb D$&diffeomorphisms&1\\
$r$&$\mathfrak{so}(E)\otimes\mathbb D$&frame rotations&1\\
$\lambda$&$\mathfrak{so}(E)\otimes\mathbb D_0$&gauge transformations&1\\
$\Lambda$&$T^*Y\otimes\mathbb D$&gerbe transformations&1\\
$\sigma$&$\mathbb D$&scale transformations&1\\
$c$&$\mathbb D$&gerbe reducibility&2\\
\bottomrule
\end{tabular}
\end{center}
For the gerbe symmetry, $\delta_\Lambda B=\dd\Lambda$, so $\Lambda$ and
$\Lambda+\dd c$ generate the same transformation; $c$ is the corresponding
degree-two reducibility parameter. The frame parameter $r$ is independent in
the auxiliary presentation. On the physical metric-frame slice it is replaced
by the metric-dependent rotation of
\cite[(D.3), (D.7)]{BrongersPaperI}.

Let
\[
 f\in\operatorname{Der}^1\mathscr O(\mathscr B)
\]
be the odd derivation defined as follows. On the tensor fields,
\[
\begin{aligned}
 fe&=-\mathcal L_\xi e+r\cdot e,\\
 fU&=-\mathcal L_\xi U-r^{\mathsf t}\!\cdot U+\sigma U,\\
 f\nu&=-\mathcal L_\xi\nu+\sigma\nu,\\
 f\zeta&=-\mathcal L_\xi\zeta+s(r)\zeta-\frac12\sigma\zeta,\\
 fh&=-\mathcal L_\xi h+\rho_{\Lambda^3E^*}(r)h.
\end{aligned}
\]
The spinor derivative and the two connections transform by
\[
\begin{aligned}
 f\Xi&=-\mathcal L_\xi\Xi+s(r)\Xi
       -\frac12\sigma\Xi-\frac12\,d\sigma\,\zeta,\\
 f\Theta&=-\mathcal L_\xi\Theta-D_\Theta r,\\
 fA&=-\mathcal L_\xi A-D_A\bigl(r^{[0]}+\lambda\bigr).
\end{aligned}
\]
Put
\[
 a=A-\Theta^{[0]},\qquad
 b_\lambda=
 \frac12P(\lambda,F_A)
 -\frac14P(a,D_A\lambda).
\]
The gerbe coordinates satisfy
\[
\begin{aligned}
 fB&=-\mathcal L_\xi B-d\Lambda-\epsilon b_\lambda,\\
 f\Lambda&=dc-\mathcal L_\xi\Lambda
 -\frac{\epsilon}{4}P(\lambda,D_A\lambda),\\
 fc&=-\xi(c)+\frac{\epsilon}{24}P(\lambda,[\lambda,\lambda]).
\end{aligned}
\]
Finally,
\[
\begin{aligned}
 f\xi&=-\frac12[\xi,\xi],\\
 fr&=\frac12[r,r]-\xi(r),\\
 f\lambda&=\frac12[\lambda,\lambda]-\xi(\lambda)
            +[\lambda,r^{[0]}],\\
 f\sigma&=-\xi(\sigma).
\end{aligned}
\]
The derivation acts from the left and commutes with \(d\).

The four constraints transform covariantly:
\[
 fC_i
 =
 -\mathcal L_\xi C_i+\rho_i(r)C_i+w_i\sigma C_i,
 \qquad
 (w_U,w_\Xi,w_T,w_H)
 =
 \left(1,-\frac12,0,0\right),
\]
where \(C_i=(C_U,C_\Xi,T,C_H)\) and \(\rho_i\) is the induced
\(SO(E)\)-representation. We define the multiplier transformations by
the dual action,
\[
 fm_i
 =
 -\mathcal L_\xi m_i-\rho_i(r)^{\mathsf t}m_i-w_i\sigma m_i.
\]

\begin{proposition}[Closure and invariance]
\label{prop:gauge-theorem}
The derivation \(f\) satisfies
\[
 f^2=0,\qquad fW_{\rm aux}=0,
\]
where the second equality is understood modulo horizontal divergences.
\end{proposition}

\begin{proof}
The transformations of
\[
(\Theta,A,B;\xi,r,\lambda,\Lambda,c)
\]
are the relative-coordinate form of the Courant Lie-\(2\) and
Green--Schwarz differential. Its closure, including the two higher
descent identities for \(\Lambda\) and \(c\), is proved in
\cite[Appendix~D.1--D.2]{BrongersPaperI}.

The remaining fields carry the indicated tensor representations and
scale weights. The inhomogeneous term in \(f\Xi\) gives
\[
 fC_\Xi
 =
 -\mathcal L_\xi C_\Xi+s(r)C_\Xi-\frac12\sigma C_\Xi,
\]
and the other three constraints satisfy the corresponding covariance
relations displayed above. Hence the multiplier terms in
\(W_{\rm aux}\) are invariant modulo Lie derivatives.

It remains to consider the first two terms of
\eqref{eq:first-order-source}. Their homogeneous scale weights vanish,
and frame invariance follows from equivariance of the spinor and metric
pairings. The only inhomogeneous scale contribution is
\[
 -\frac12
 \bigl\langle J_\zeta(d\sigma\,\zeta)\wedge U\bigr\rangle_E,
\]
which vanishes because
\[
 \Bsp\!\left(\zeta,c_E(s^\flat)\zeta\right)=0
 \qquad (s\in E).
\]
Diffeomorphisms act on the integrand by a Lie derivative of a top form,
hence by a horizontal divergence. Therefore \(fW_{\rm aux}=0\).
\end{proof}

\section[The quartic Hamiltonian and cyclic L3 structure]{The quartic Hamiltonian and cyclic $L_3$ structure}\label{sec:cyclic-l3}\label{sec:cyclic}\label{sec:deformation-algebra}

\subsection{The quartic Hamiltonian}
The graded base $\mathscr B$ of Section~\ref{sec:auxiliary-symmetry} has
shifted cotangent space
\[
 \mathcal M_B=T^*[-1]\mathscr B,\qquad
 \omega_B=\delta\vartheta_B,
\]
where $\vartheta_B$ is the shifted density-valued cotangent potential. Set
\begin{equation}\label{eq:small-master}
 \cS_B=W_{\rm aux}+\iota_f\vartheta_B,
\end{equation}
and let $Q_B$ be its Hamiltonian vector field, defined by
\[
 \iota_{Q_B}\omega_B=-\delta\cS_B.
\]

\begin{proposition}[Hamiltonian structure of the first-order theory]\label{prop:small-master}
The functional \eqref{eq:small-master} satisfies
\begin{equation}\label{eq:small-cme}
 \{\cS_B,\cS_B\}=0,\qquad Q_B^2=0,\qquad
 \delta\omega_B=0,\qquad\cL_{Q_B}\omega_B=0.
\end{equation}
\end{proposition}
\begin{proof}
Proposition~\ref{prop:gauge-theorem} gives $f^2=0$ and $fW_{\rm aux}=0$.
The cotangent master-equation result \cite[Lemma~3.8]{BrongersPaperI}
therefore gives $\{\cS_B,\cS_B\}=0$ and $Q_B^2=0$. Since
$\omega_B=\delta\vartheta_B$, it is closed and preserved by its Hamiltonian
vector field.
\end{proof}
The restriction of $Q_B$ to $\mathscr B$ is $f$, and its cotangent component is
the variational cotangent lift, including the Euler--Lagrange derivative of
$W_{\rm aux}$.

Throughout this section, the subscript $\ast$ denotes evaluation at the
standard embedding. Write $e_\ast$ for the torsion-free coframe,
$\chi_\ast$ for its parallel unit spinor and $\Phi_\ast$ for the constant
dilaton. Choose the flat gerbe background $B_\ast$, with $\dd B_\ast=0$, and
set
\[
\begin{aligned}
 \Theta_\ast&=\Theta_{\LC}(e_\ast),&
 A_\ast&=\red{\Theta_\ast},&
 H_\ast&=0,\\
 \nu_\ast&=\vol_{e_\ast},&
 \zeta_\ast&=e^{-\Phi_\ast}\chi_\ast.
\end{aligned}
\]
Separate the fields as $q=(e,\zeta,B,\nu,A)$ and
$y=(U,\Xi,\Theta,h)$, and let $y_\ast$ be the constraint-graph value at
$q_\ast$. In particular,
\[
 U_\ast(s)=\iota_{e_\ast^{-1}s}\nu_\ast,\qquad
 h_\ast=0,\qquad
 \Xi_\ast=D_{\Theta_\ast}\zeta_\ast=0.
\]
Let $W_{\rm geom}$ denote the first two terms of \eqref{eq:first-order-source},
so that
\[
 W_{\rm aux}=W_{\rm geom}+\langle m,C\rangle,
\]
and put
\[
 \mathcal J_\ast:=(D_yC)_{(q_\ast,y_\ast)}.
\]
The stationary multiplier is
\begin{equation}\label{eq:stationary-multiplier}
 m_\ast=-(\mathcal J_\ast^\dagger)^{-1}(D_yW_{\rm geom})_\ast.
\end{equation}
The inverse exists by Lemma~\ref{lem:constraint-inverse}.

The constraint graph $y=y(q)$ satisfies
\[
 C(q,y(q))=0,\qquad
 W_{\rm aux}(q,y(q),m)=W_{\rm ext}(q).
\]
Hence
\[
 (D_mW_{\rm aux})_\ast=0,\qquad
 (D_yW_{\rm aux})_\ast
 =(D_yW_{\rm geom})_\ast+\mathcal J_\ast^\dagger m_\ast=0.
\]
The chain rule gives
\[
 (D_qW_{\rm ext})_\ast=(D_qW_{\rm aux})_\ast.
\]
The physical first variation vanishes at the standard embedding
\cite[Section~4.5, (4.7)]{BrongersPaperI}, so
\begin{equation}\label{eq:ext-stationary-point}
 (D_qW_{\rm ext})_\ast=0.
\end{equation}
Thus every Euler--Lagrange derivative of $W_{\rm aux}$ vanishes there. Let
$b_\ast\in\mathcal M_B$ have field and multiplier coordinates
$(q_\ast,y_\ast,m_\ast)$, with all ghosts and shifted cotangent variables
zero. Then
\begin{equation}\label{eq:parent-stationary-point}
 Q_B(b_\ast)=0,\qquad \cS_B(b_\ast)=0.
\end{equation}
We take the formal neighbourhood of $b_\ast$.

Assign polynomial degree one to every non-background graded coordinate on
$\mathcal M_B$, including ghosts and shifted cotangent variables. Every
component of $f$ has polynomial degree at most three by
Appendix~\ref{app:gauge-hierarchy} and the inherited hierarchy
\cite[(D.10)--(D.11)]{BrongersPaperI}. Since $\vartheta_B$ is linear in a
shifted cotangent variable and a base variation,
\[
 \deg(\iota_f\vartheta_B)\le4,\qquad \deg\cS_B\le4.
\]
Translation to $b_\ast$ does not increase degree, and $\omega_B$ has constant
coefficients in the independent field and cotangent coordinates. Therefore
\begin{equation}\label{eq:strict-l3-expansion}
 \cS_B=S_2+S_3+S_4,\qquad Q_B=Q_1+Q_2+Q_3.
\end{equation}

\begin{theorem}[Finite cyclic presentation]\label{thm:main-b}
The first-order fields, constraint multipliers, gauge and reducibility
hierarchy, and their shifted density-valued variational duals are locally
modelled on sections of a graded coefficient bundle
\[
 \mathcal E_B
 =(E_{\mathrm{fr}}\otimes_{\R}\mathbb D)
 \oplus(E_0\otimes_{\R}\mathbb D_0),
\]
where $E_{\mathrm{fr}}$ and $E_0$ are finite-rank graded real bundles. The
pairing is induced by the canonical shifted cotangent pairing on
$\mathcal M_B$, and
\[
 \deg\cS_B\le4,\qquad \{\cS_B,\cS_B\}=0.
\]
At $b_\ast$, the Taylor operations form an uncurved cyclic $L_3$ algebra:
\[
 \ell_n=0\qquad(n\ge4).
\]
\end{theorem}
\begin{proof}
The fields, multipliers and symmetry coordinates form a finite list of
finite-rank bundles. Their shifted density-valued duals use
$\epsilon\mathbb D\cong\mathbb D_0$ on the residue component. Proposition~\ref{prop:small-master}
gives the master equation. By \eqref{eq:strict-l3-expansion}, Hamiltonian
differentiation gives a vector field of degree at most three, and
\eqref{eq:parent-stationary-point} removes its constant term. Its Taylor
coefficients therefore satisfy the uncurved $L_\infty$ identities and vanish
above arity three. Cyclicity follows from the constant shifted symplectic
pairing by \cite[Appendix~C.4, Proposition~C.2]{BrongersPaperI}.
\end{proof}

The inherited scalar transformation contains
$fc=-\xi(c)+\eps P(\lambda,[\lambda,\lambda])/24$
\cite[(D.11)]{BrongersPaperI}. Its cubic term gives
\begin{equation}\label{eq:cartan-ternary-small}
 \ell_3(\lambda_1,\lambda_2,\lambda_3)
 =-\tfrac\eps4P(\lambda_1,[\lambda_2,\lambda_3]),
\end{equation}
with the suspension signs verified in Appendix~\ref{app:cyclic-conventions}.

\paragraph{Reduction to physical coordinates.}\label{sec:higher-interactions}
Auxiliary elimination recovers the physical variational Hamiltonian of
\cite[Section~4]{BrongersPaperI},
\[
 \cS_{\rm red}=W_{\rm ext}+\iota_{f_{\rm red}}\vartheta_{\rm red}.
\]
The elimination formulas contain inverse coframes, spinor normalization and
the induced connection, and are therefore nonpolynomial in physical
coordinates. The quartic degree bound for $\cS_B$ consequently need not
survive elimination. Appendix~\ref{app:interaction-coefficients} records the
flat-torus coefficient calculation used below.

\subsection{A nonzero quintic coefficient}
\label{subsec:quintic-coefficient}
Let $Y=T^7$ carry the flat metric $g_0$ and a parallel unit spinor $\chi_0$.
Set $B=0$, $A=\Theta_0$ and $\varrho=w\vol_0$ with constant $w>0$.
Write $c_0$ for Clifford multiplication at the fixed background metric
$g_0$. For a sufficiently small vector field $v$, define
\[
 \chi_v=\frac{\chi_0+c_0(v^\flat)\chi_0}{\sqrt{1+|v|_0^2}}.
\]
The spinor $c_0(v^\flat)\chi_0$ lies in the tangent space to the unit-spinor
sphere at $\chi_0$, and $v\mapsto\chi_v$ is a local coordinate chart near
$\chi_0$. Let $y^0$ have period $2\pi$ and let
$E_0=\partial_{y^0},E_1$ be parallel orthonormal vector fields. Set
\[
 v=\cos y^0\,E_0+(1+\sin y^0)\,E_1.
\]
The physical spinor family below is the straight coordinate line $tv$
in this chart:
\begin{equation}\label{eq:normalized-spinor-example}
 \chi(t)=\chi_{tv}=\frac{\chi_0+t\,c_0(v^\flat)\chi_0}{\sqrt{1+t^2|v|_0^2}},
 \qquad\zeta_\ph(t)=\sqrt w\,\chi(t),
\end{equation}
This spinor determines a positive form for the fixed metric. Put
\[
 \mathcal Q(v)=\Bsp\bigl(c_0(v^\flat)\chi_0,
               \mathscr D_0(c_0(v^\flat)\chi_0)\bigr),
\]
with $\mathscr D_0$ the background Dirac operator. Then
\begin{equation}\label{eq:rational-spinor-source}
 W(t)=w\int_Y\frac{-t\operatorname{div}_0v+t^2\mathcal Q(v)}
                         {1+t^2|v|_0^2}\,\vol_0.
\end{equation}
The normalization calculation is given in Appendix~\ref{app:interaction-coefficients}.
For this $v$, $\mathcal Q(v)=0$,
$\operatorname{div}_0v=-\sin y^0$ and $|v|_0^2=2+2\sin y^0$.
With $W_n=[t^n]W(t)$ and $V_6$ the volume in the other six directions,
\begin{equation}\label{eq:quintic-example}
 W_3=-2\pi wV_6,\qquad W_5=8\pi wV_6\ne0.
\end{equation}
Thus the physical functional has a nonzero quintic Taylor coefficient
in these local physical spinor coordinates.

\section{Product reduction and physical equations}
\label{sec:product-reduction}\label{sec:product}

We now specialize to
\[
Y=S^1\times X,
\]
with \(X\) a Calabi--Yau threefold, and restrict the seven-dimensional
theory to the product locus described below. Circle reduction then produces
a string Courant algebroid and variational functional over \(X\). On the
same locus, the heterotic supersymmetry and instanton equations split
naturally into holomorphic equations and conditions depending on the
Hermitian metric and real representative. This decomposition provides the
geometric motivation for the holomorphic lifting problem constructed in
Section~\ref{sec:holomorphic-lifting}.

\subsection{Product data and Courant reduction}
Let $X$ be a closed Calabi--Yau threefold with reference forms
$(\omega_0,\Psi_0)$ and constant reference dilaton $\Phi_0$, and set
$Y=S^1\times X$. Let $\partial_\theta$ generate the circle action and let
$\eta$ be the invariant closed one-form with
$\eta(\partial_\theta)=1$. Set
\[
 L=\int_{S^1}\eta>0,\qquad
 \vol_Y=\eta\wedge\vol_X,\qquad
 \frac i8\Psi_0\wedge\overline\Psi_0=\frac{\omega_0^3}{6}.
\]
Choose the spin structure on $S^1$ for which the background parallel spinor
on $Y$ is $S^1$-invariant. Fix the standard-embedding bundle
identifications so that the reference tangent and gauge connections have no
circle component. We also fix a differential-gerbe presentation on $X$ and
pull it back to $Y$.

For an invariant metric, let $R$ be the circle radius and $v$ its
horizontal mixing one-form. The unit circle coframe is $e^7$, where
\begin{equation}\label{eq:product-radius}
 R=\sqrt{g(\partial_\theta,\partial_\theta)},\qquad
 v=R^{-2}g(\partial_\theta,-)-\eta,
 \qquad e^7=R(\eta+v).
\end{equation}
The spinor determines an $SU(3)$ structure $(\omega,\Psi)$.
Write $\rho_X=\operatorname{Re}\Psi$ and
$\widehat\rho_X=\operatorname{Im}\Psi$. Our phase convention is
\begin{equation}\label{eq:product-forms}
 \varphi=e^7\wedge\omega+\rho_X,\qquad
 \psi=\tfrac12\omega^2-e^7\wedge\widehat\rho_X,
 \qquad \omega\wedge\Psi=0,
\end{equation}
with $i\Psi\wedge\overline\Psi/8=\omega^3/6$. Define
\begin{equation}\label{eq:product-weighted-volume}
 w=e^{-2\Phi},\qquad \Omega=w\Psi,
 \qquad \Omega_0=e^{-2\Phi_0}\Psi_0.
\end{equation}

Decompose the relative gauge coordinate and gerbe two-form as
\[
 a=a_1+\eta a_0,\qquad B=b_2+\eta\wedge b_1.
\]
The product locus is
\begin{equation}\label{eq:product-base-locus}
 R=1,\qquad v=0,\qquad b_1=0,\qquad a_0=0.
\end{equation}
Thus $g=\eta^2+g_X$, the gerbe representative is horizontal, and the
relative gauge connection has no circle component.

Let $\mathbb E_Y$ be the string Courant algebroid in the product splitting.
The isotropic circle action is generated by $\partial_\theta+0+0$, and we
write
\[
 K_c=\mathbb R\partial_\theta.
\]

\begin{proposition}\label{prop:real-circle-reduction}
On the product locus \eqref{eq:product-base-locus},
\begin{equation}\label{eq:real-circle-reduction}
 \mathbb E_X=(K_c^\perp/K_c)/S^1
 \simeq TX\oplus\ad P_T\oplus\ad P_G\oplus T^*X.
\end{equation}
This is the full real transitive string Courant algebroid determined by the
horizontal connection and flux.
\end{proposition}
\begin{proof}
In the product splitting,
\[
 K_c^\perp/K_c
 \cong \pi_X^*\bigl(TX\oplus\ad P_T\oplus\ad P_G\oplus T^*X\bigr),
\]
where $\pi_X:Y\to X$ is the product projection. The horizontal
connection and flux make the Dorfman bracket preserve the displayed
summands on invariant sections. The
anchor, pairing and anomaly identity are the restrictions of those on
$\mathbb E_Y$. Hence the quotient satisfies the hypotheses of isotropic
Courant reduction \cite{BCGReduction} and has the stated form.
\end{proof}

Write $(\mathbb E_Y,V_+,[\chi],\varrho,e)$ for the geometric data of
Section~\ref{sec:courant-g2}. Let $a$ be the anchor and
$j:T^*Y\to\mathbb E_Y$ the cotangent inclusion. The positive lift
$v_+=(a|_{V_+})^{-1}$ defines
\[
 g(U,V)=\langle v_+U,v_+V\rangle,\qquad
 s(U)=v_+U-j(gU).
\]
On the product locus, $s(\partial_\theta)=\partial_\theta$ in the product
splitting, so the isotropic line is the same $K_c$ as in
the preceding quotient.

\begin{proposition}[Intrinsic product data]\label{prop:intrinsic-product-geometry}
Let $\pi_c:K_c^\perp\to K_c^\perp/K_c$ be the quotient map. The full circle
quotient carries the positive rank-six bundle
\begin{equation}\label{eq:intrinsic-product-positive}
 V_{+,X}=\pi_c(V_+\cap K_c^\perp)/S^1.
\end{equation}
Its anchor metric is the restriction of $g$ to the horizontal tangent
bundle. The spinor line and oriented unit circle determine the $SU(3)$
structure in \eqref{eq:product-forms}. The normalized quotient density and weighted complex volume are
\begin{equation}\label{eq:intrinsic-product-density}
 \varrho_X=\frac1L\int_{S^1}\varrho=w\vol_6,\qquad
 \Omega=w\Psi,\qquad
 \frac i8\Omega\wedge\overline\Omega=w^2\vol_6,
 \qquad w=e^{-2\Phi}.
\end{equation}
\end{proposition}
\begin{proof}
Put $p_+=v_+\partial_\theta$. Since
$s(\partial_\theta)=p_+-j\eta$, the Courant pairing gives
\[
 \langle p_+,p_+\rangle=1,\qquad
 \langle v_+U,s(\partial_\theta)\rangle
 =\tfrac12g(U,\partial_\theta),
\]
and hence
\[
 V_+\cap K_c^\perp=v_+((\partial_\theta)^\perp).
\]
This bundle meets $K_c$ trivially, and its quotient anchor is an isomorphism
onto $TX$ with pairing $g_6=g|_{(\partial_\theta)^\perp}$.

Using the spinor pairing $\beta_\Delta$ of
Section~\ref{sec:courant-g2}, define
\[
 \chi_\pm=\frac{\chi\mp i\gamma_{\partial_\theta}\chi}{\sqrt2}.
\]
Then
\[
 \gamma_{\partial_\theta}\chi_\pm=\pm i\chi_\pm,
 \qquad
 \beta_{\Delta,\mathbb C}(\chi_-,\chi_+)=1.
\]
The stabilizer of the resulting unit chiral spinor in $\Spin(6)$ is
$SU(3)$. The associated forms are
\[
 \omega=\iota_{\partial_\theta}\varphi,\qquad
 \rho_X=\varphi|_{(\partial_\theta)^\perp},\qquad
 \widehat\rho_X=-\iota_{\partial_\theta}\psi,
\]
and hence satisfy \eqref{eq:product-forms}. Finally,
\[
 \varrho=w\,\eta\wedge\vol_6,
\]
so integration over the circle gives
\eqref{eq:intrinsic-product-density}.
\end{proof}

\subsection{The reduced functional}
\begin{theorem}[Off-shell functional reduction]\label{thm:functional-reduction}
For every formal field in the product locus,
\begin{equation}\label{eq:product-functional}
W_{\mathrm{Courant},G_2}\big|_{\rm product}
 =\frac L4\operatorname{Im}\int_X(H+i\,\dd\omega)\wedge\Omega.
\end{equation}
\end{theorem}
\begin{proof}
The generalized Dirac identity \cite[Propositions~4.1--4.2]{BrongersPaperI}
gives \eqref{eq:ii-dirac-functional}. Since the product density is
$w\eta\wedge\vol_X$ and $H$ is horizontal,
\[
 H\wedge\psi=\eta\wedge H\wedge\widehat\rho_X.
\]
Differentiating $\omega\wedge\rho_X=0$ gives
\[
 \dd\varphi\wedge\varphi
 =\eta\wedge(\dd\rho_X\wedge\omega-\dd\omega\wedge\rho_X)
 =-2\eta\wedge\dd\omega\wedge\rho_X.
\]
After integration over the circle,
\[
 W_{\mathrm{Courant},G_2}\big|_{\rm product}
 =\frac L4\int_Xw\bigl(H\wedge\widehat\rho_X
                 +\dd\omega\wedge\rho_X\bigr),
\]
which is \eqref{eq:product-functional}. No field equations are used in this
reduction.
\end{proof}

\subsection{The physical equations}\label{sec:product-bps}
Write $\rho=\operatorname{Re}\Psi$,
$\widehat\rho=\operatorname{Im}\Psi$, $w=e^{-2\Phi}$ and
$\Omega=w\Psi$. The supersymmetry and instanton (BPS) residuals of
\cite[Section~2.5]{BrongersPaperI} are
\begin{equation}\label{eq:physical-bps-residuals}
 \begin{aligned}
 C&=\dd(w\psi),& \Sigma&=\dd\varphi\wedge\varphi,\\
 T&=H+*_7(\dd\varphi-2\dd\Phi\wedge\varphi),&
 I&=F_A\wedge\psi^{[0]}.
 \end{aligned}
\end{equation}
Here $C,\Sigma,T$ are $\mathbb D$-valued and $I$ is
$\mathbb D_0$-valued.
The physical problem also imposes $H^{[0]}=0$
\cite[Section~2]{BrongersPaperI}.
Let $J$ be defined by $\omega(X,Y)=g(JX,Y)$, and let $\mathcal J$
act on forms by applying $J$ to every argument. Thus
$\mathcal J\dd z^a=i\dd z^a$. When $J$ is integrable, set
$\dd^c=i(\bar\partial-\partial)$.

Throughout this subsection, we consider real formal deformations of the
fixed Calabi--Yau background satisfying the product conditions
\eqref{eq:product-base-locus}, with the invariant spin structure and the
fixed horizontal gerbe presentation. The geometric variables are
$\mathbb D$-valued, while
\[
 A\in\mathcal A(P_G)\otimes\mathbb D_0,
\]
and \(H,\Theta_H\) are determined by
\eqref{eq:determined-graph-small}. Type decompositions of gauge forms are
taken with respect to \(J^{[0]}\).

\begin{proposition}[BPS equations on the product locus]
\label{prop:product-bps-dictionary}
Under these assumptions,
\[
 C=\Sigma=T=I=0,\qquad H^{[0]}=0
\]
if and only if
\begin{equation}\label{eq:product-bps-six}
 \begin{aligned}
 d\Omega&=0,
 &d(w\omega^2)&=0,\\
 H&=i(\partial-\bar\partial)\omega=-d^c\omega,
 &F_A^{0,2}&=0,\\
 F_A\wedge(\omega^2)^{[0]}&=0,
 &d_Xb_2^{[0]}&=0.
 \end{aligned}
\end{equation}
\end{proposition}

The first three equations in \eqref{eq:product-bps-six} are
$\mathbb D$-valued; the gauge equations and the leading-flux equation are
$\mathbb D_0$-valued.
\begin{proof}
On the product locus, the determined flux \(H\) and the connections
\(\Theta_H,A\) are horizontal; see
Appendix~\ref{app:product-bps}. For a horizontal \(k\)-form \(\alpha\),
\[
 *_7\alpha=(-1)^k\eta\wedge *_6\alpha,
 \qquad
 *_7(\eta\wedge\alpha)=*_6\alpha.
\]
Set
\[
 U=d\omega-2d\Phi\wedge\omega,
 \qquad
 V=d\rho-2d\Phi\wedge\rho.
\]
Decomposing each residual into its horizontal and \(\eta\)-components gives
\begin{equation}\label{eq:product-bps-components}
 \begin{array}{c|l|l}
 &\text{horizontal}&\eta\text{-component}\\ \hline
 C&
 \displaystyle C_h=\frac12d(w\omega^2)&
 \displaystyle C_v=d(w\widehat\rho)\\[3pt]
 \Sigma&
 \Sigma_h=0&
 \displaystyle \Sigma_v=d\rho\wedge\omega-d\omega\wedge\rho\\[3pt]
 T&
 \displaystyle T_h=H-*_6U&
 \displaystyle T_v=*_6V\\[3pt]
 I&
 \displaystyle I_h=\frac12F_A\wedge(\omega^2)^{[0]}&
 \displaystyle I_v=-F_A\wedge\widehat\rho^{[0]} .
 \end{array}
\end{equation}
Since \(\Omega=w(\rho+i\widehat\rho)\), these formulas imply
\begin{equation}\label{eq:product-volume-origin}
 d\operatorname{Im}\Omega=C_v,
 \qquad
 d\operatorname{Re}\Omega=w*_6T_v.
\end{equation}
Moreover, differentiating \(\omega\wedge\rho=0\) gives
\[
 \Sigma
 =
 2\eta\wedge\omega\wedge *_6T_v.
\]
Thus \(C_v=T_v=0\) is equivalent to \(d\Omega=0\), and then
\(\Sigma=0\) automatically.

It remains to compare the horizontal equations \(C_h=T_h=0\) with
balance and the torsion equation. Wedge multiplication
\[
 A_\omega:\Omega^1(X)\longrightarrow\Omega^5(X),
 \qquad
 A_\omega(\lambda)=\lambda\wedge\omega^2,
\]
is a pointwise isomorphism. Define
\[
 b_\omega=A_\omega^{-1}(w^{-1}C_h),
 \qquad
 S_H=H-\mathcal Jd\omega.
\]
If
\[
 d\omega=\beta+\lambda\wedge\omega
\]
is the Lefschetz decomposition with $\beta$ primitive, the \(SU(3)\) Hodge identities of
Appendix~\ref{app:product-bps} give
\[
 *_6\beta=\mathcal J\beta,
 \qquad
 *_6(\lambda\wedge\omega)
 =-\mathcal J(\lambda\wedge\omega).
\]
Since
\[
 b_\omega=\lambda-d\Phi,
\]
we obtain
\begin{equation}\label{eq:product-bps-inverse}
 \begin{aligned}
 C_v&=d\operatorname{Im}\Omega,
 &
 T_v&=w^{-1}*_6d\operatorname{Re}\Omega,\\
 T_h&=S_H+2\mathcal J(b_\omega\wedge\omega),
 &
 \Sigma&=2\eta\wedge\omega\wedge *_6T_v.
 \end{aligned}
\end{equation}

Suppose first that
\[
 C=\Sigma=T=I=0,\qquad H^{[0]}=0.
\]
Equation \eqref{eq:product-volume-origin} gives \(d\Omega=0\), while
\(C_h=0\) gives
\[
 d(w\omega^2)=0.
\]
Because \(\Omega\) is nowhere vanishing and decomposable, its closure
implies that the almost-complex structure is integrable. Hence
\[
 \mathcal Jd\omega
 =
 i(\partial-\bar\partial)\omega.
\]
Balance gives \(b_\omega=0\), so \(T_h=0\) and
\eqref{eq:product-bps-inverse} yield
\[
 H=i(\partial-\bar\partial)\omega=-d^c\omega.
\]

For the gauge equations, wedge multiplication by
\(\widehat\rho^{[0]}\) vanishes on real \((1,1)\)-forms and is injective
on the real \((2,0)+(0,2)\) component. Since \(A\) is a real compact
connection,
\[
 I_v=0
 \quad\Longleftrightarrow\quad
 F_A^{0,2}=0.
\]
The remaining component is
\[
 I_h=0
 \quad\Longleftrightarrow\quad
 F_A\wedge(\omega^2)^{[0]}=0.
\]
Finally, on the product locus,
\[
 H_h^{[0]}=d_Xb_2^{[0]},
 \qquad
 H_v^{[0]}=0,
\]
so \(H^{[0]}=0\) gives \(d_Xb_2^{[0]}=0\). This proves
\eqref{eq:product-bps-six}.

Conversely, suppose that \eqref{eq:product-bps-six} holds. From
\(d\Omega=0\) and \eqref{eq:product-volume-origin},
\[
 C_v=T_v=0,
\]
and therefore \(\Sigma=0\). Balance gives \(C_h=0\) and
\(b_\omega=0\). The torsion equation says \(S_H=0\), so
\eqref{eq:product-bps-inverse} gives \(T_h=0\). Thus
\[
 C=\Sigma=T=0.
\]
The conditions \(F_A^{0,2}=0\) and
\(F_A\wedge(\omega^2)^{[0]}=0\), together with reality of \(A\), give
\(I_v=I_h=0\). Finally,
\(d_Xb_2^{[0]}=0\) and the product expression for the leading flux imply
\(H^{[0]}=0\). Hence
\[
 C=\Sigma=T=I=0,\qquad H^{[0]}=0,
\]
as required.
\end{proof}

\begin{corollary}[Leading Calabi--Yau geometry]\label{cor:leading-cy}
A physical solution under the hypotheses of
Proposition~\ref{prop:product-bps-dictionary} satisfies
\[
 \dd\omega^{[0]}=0,\qquad \dd\Psi^{[0]}=0,\qquad \dd\Phi^{[0]}=0.
\]
Its leading geometry is Calabi--Yau with constant dilaton. The
leading Hull connection is the Levi--Civita connection, and its
curvature is an $SU(3)$ instanton.
\end{corollary}
\begin{proof}
The equations $H^{[0]}=T_h^{[0]}=0$ give
$\dd\omega^{[0]}=2\dd\Phi^{[0]}\wedge\omega^{[0]}$.
Balance then gives
$\dd\Phi^{[0]}\wedge(\omega^{[0]})^2=0$. Invertibility of
$A_{\omega^{[0]}}$ implies $\dd\Phi^{[0]}=0$ and
$\dd\omega^{[0]}=0$. Weighted-volume closure gives
$\dd\Psi^{[0]}=0$. The holomorphic volume has constant norm for
the K\"ahler metric, so it is parallel and the holonomy lies in
$SU(3)$. The induced curvature is of type $(1,1)$ and primitive.
\end{proof}

\section[The holomorphic lifting L3 algebra]{The holomorphic lifting $L_3$ algebra}
\label{sec:holomorphic-lifting}
\subsection{Hermitian tangent subalgebroid}
\label{subsec:product-geometry}

The product Hull connection induces
\[
 \Theta_{H,X}
 =
 \nabla^{g_6}-\frac12g_6^{-1}H
\]
on \(TX\), but need not preserve the almost-complex structure \(J\).
Let \(P_H\subset P_T\) be the \(U(3)\)-reduction determined by
\((g_6,J)\), and let \(\Gamma\) be a connection on \(P_H\). Writing
\[
 t=\Theta_H-\Gamma,
\]
we use the corresponding Courant splitting with three-form
\[
 H_\Gamma
 =
 H_H-\frac z4\,\mathcal T_{3,P}(\Gamma,t).
\]
It satisfies
\[
 dH_\Gamma
 =
 \frac z4\left(
 P(R_\Gamma,R_\Gamma)-P(F_A,F_A)
 \right).
\]

\begin{proposition}[Hermitian tangent subalgebroid]
\label{prop:hermitian-carrier}
At nonzero anomaly coefficient \(z\), a splitting whose tangent
connection \(\Gamma\) preserves \(P_H\) restricts the real string Courant
algebroid to
\begin{equation}\label{eq:hermitian-carrier}
 \mathbb E_H
 =
 TX\oplus\ad P_H\oplus\ad P_G\oplus T^*X
 \subset\mathbb E_X.
\end{equation}
With respect to
\[
 E_{\mathbb C}
 \simeq
 T^{1,0}X\oplus T^{0,1}X\oplus\mathbb C\,\partial_\theta,
\]
the complexified tangent adjoint algebra is embedded by
\begin{equation}\label{eq:hermitian-carrier-trace}
 \iota:\mathfrak{gl}(3,\mathbb C)\hookrightarrow
       \mathfrak{so}(7,\mathbb C),
 \qquad
 \iota(M)=\operatorname{diag}(M,-M^{\mathsf T},0),
\end{equation}
and the restricted anomaly pairing is
\[
 c_z((M,a),(N,b))
 =
 \frac z4
 \left\{
 \operatorname{tr}_7(ab)-2\operatorname{tr}_3(MN)
 \right\}.
\]
\end{proposition}

\begin{proof}
Since \(\Gamma\) preserves \(P_H\),
\[
 D_\Gamma\Omega^\bullet(\ad P_H)
 \subset\Omega^{\bullet+1}(\ad P_H),
 \qquad
 R_\Gamma\in\Omega^2(\ad P_H).
\]
Together with
\(
[\mathfrak u(3),\mathfrak u(3)]\subset\mathfrak u(3)
\),
this shows that the covariant-derivative, curvature and adjoint-bracket
terms in the splitting Dorfman bracket preserve
\eqref{eq:hermitian-carrier}; the remaining terms take values in
\(T^*X\). Moreover,
\[
 \operatorname{tr}_7(\iota(M)\iota(N))
 =
 2\operatorname{tr}_3(MN),
\]
so the restricted pairing has the stated form and is nondegenerate.
The transgressed three-form \(H_\Gamma\) satisfies the corresponding
restricted anomaly identity.
\end{proof}

\begin{remark}
The passage from \(\mathbb E_X\) to \(\mathbb E_H\) is a subalgebroid
restriction, not a quotient by complementary tangent directions. Indeed,
the trace-orthogonal complement of
\(\mathfrak u(3)\subset\mathfrak{so}(7)\) is nonzero and proper, hence
cannot be an ideal because \(\mathfrak{so}(7)\) is simple.
\end{remark}

For the comparison with the physical Hull representative, we henceforth
take \(\Gamma\) to be the Ehresmann--Libermann connection of
\((g_6,J)\), characterized by
\[
 \nabla^\Gamma g_6=0,
 \qquad
 \nabla^\Gamma J=0,
 \qquad
 (T_\Gamma)^{1,1}=0.
\]
It depends on the first derivatives of the almost-Hermitian structure
and agrees with the Chern connection when \(J\) is integrable
\cite[Lemma~2.2]{CanonicalHermitian}.
Section~\ref{sec:representatives} compares it with \(\Theta_H\).

\subsection{Transverse isotropic lifts}

Let
\[
 D_\mu=T^{0,1}_\mu X\subset T_{\mathbb C}X
\]
for a deformation of the reference antiholomorphic distribution, and let
$\mathcal I_\mu\subset\Omega^\bullet(X;\mathbb C)$ be the exterior ideal
generated by $\operatorname{Ann}(D_\mu)$.
A transverse isotropic lift is a subbundle
\[
 \mathcal L\subset\mathbb E_{H,\mathbb C},
 \qquad
 a|_{\mathcal L}:\mathcal L\xrightarrow{\sim}D_\mu.
\]
Such a lift determines, in addition to \(D_\mu\), tangent and gauge
partial connections and a string two-form datum.

Fix the Calabi--Yau standard-embedding background and write
\[
 C_0=(\Gamma_0,A_0)
\]
for its tangent and gauge connections. Their curvatures have type
\((1,1)\), and their characteristic forms cancel. The background positive
lift is
\[
 \mathcal L_+=e^{i\omega_0}T^{0,1}X,
\]
which is identified with \(T^{0,1}X\) by the closed transform
\(e^{-i\omega_0}\).

Write
\[
 \rho=(\kappa,\alpha)
\]
for the tangent and gauge connection deformations and set
\[
 C'=C_0+\rho.
\]
For a complex two-form \(B\), define
\[
 G=dB+T_c(C_0,\rho),
\]
where
\[
 T_c(C_0,\rho)
 =
 2c(\rho,F_0)
 +c(\rho,D_{C_0}\rho)
 +\frac13c(\rho,[\rho,\rho]).
\]
The corresponding change of Courant splitting is
\begin{equation}\label{eq:hol-connection-map}
 \begin{split}
 \Phi_{\rho,B}:\mathbb E_{C',G}&\longrightarrow\mathbb E_{C_0,0},\\
 X+r+\xi&\longmapsto
 X+r-\rho(X)+\xi+\iota_XB
 +2c(\rho,r)-c(\rho(X),\rho).
 \end{split}
\end{equation}
Every transverse isotropic lift sufficiently near the background can be
written in the form
\begin{equation}\label{eq:intrinsic-hol-lift}
 \mathcal L=\Phi_{\rho,B}(D_\mu).
\end{equation}
Only the class of \(B\) modulo \(\mathcal I_\mu^2\) enters
\eqref{eq:intrinsic-hol-lift}. Relative to the background type
decomposition, this class is parametrized by
\[
 x\in\Omega^{0,1}(T^{*1,0}X),
 \qquad
 b\in\Omega^{0,2}(X).
\]
We normalize \(x\) so that its Courant cotangent component is \(2x\).
An explicit local representative of \(B\) in these coordinates is
recorded in Appendix~\ref{app:holomorphic-q}.

\begin{proposition}[Involutivity of the lift]
\label{prop:intrinsic-lift-involutivity}
The lift \(\mathcal L\) of \eqref{eq:intrinsic-hol-lift} is involutive
if and only if
\begin{equation}\label{eq:intrinsic-lift-equations}
 [D_\mu,D_\mu]\subset D_\mu,
 \qquad
 F_{C'}=0\pmod{\mathcal I_\mu},
 \qquad
 G=0\pmod{\mathcal I_\mu^2}.
\end{equation}
\end{proposition}

\begin{proof}
In the splitting \(\mathbb E_{C',G}\), sections \(U,V\) of \(D_\mu\)
satisfy
\[
 U\circ_GV
 =
 [U,V]-F_{C'}(U,V)+\iota_V\iota_UG.
\]
The vector component lies in \(D_\mu\) precisely when \(D_\mu\) is
involutive. The adjoint component vanishes precisely when
\(F_{C'}\) vanishes on \(\Lambda^2D_\mu\), equivalently
\(F_{C'}=0\pmod{\mathcal I_\mu}\). The cotangent component vanishes for
all \(U,V\in D_\mu\) precisely when
\(G=0\pmod{\mathcal I_\mu^2}\).
\end{proof}

At the background let
\[
 \mathcal Q_0
 =
 (T^{0,1}X)^\perp/T^{0,1}X
 \simeq
 T^{1,0}X
 \oplus\ad P_{H,\mathbb C}
 \oplus\ad P_{G,\mathbb C}
 \oplus T^{*1,0}X.
\]
For \(s\in\Gamma(\mathcal Q_0)\) and
\(U\in\Gamma(T^{0,1}X)\), choose
\(\widetilde s\in\Gamma((T^{0,1}X)^\perp)\) representing \(s\) and set
\[
 (\bar\partial_Qs)(U)
 =
 [U,\widetilde s]\bmod T^{0,1}X.
\]
Write
\[
 y=(x,\rho,\mu)
\]
and use the pairing
\begin{equation}\label{eq:hol-signed-pairing}
 \langle y,y'\rangle
 =
 \langle\mu,x'\rangle
 +\langle x,\mu'\rangle
 +c(\rho,\rho').
\end{equation}
Exterior factors in the pairing occur in the displayed order.
Let \([\ ,\ ]\) denote the skew holomorphic Courant bracket on the
background extension, extended to \((0,\bullet)\)-forms by the Koszul
rule. Here $\partial b\in\Omega^{0,\bullet}(T^{*1,0}X)$ has
components $\partial_a b$ in holomorphic coordinates.

Define
\begin{equation}\label{eq:hol-mc-residuals}
 \begin{aligned}
 R_y&=
 \bar\partial_Qy-\frac12[y,y]-\frac12\partial b,\\
 R_b&=
 \bar\partial b
 -\frac12\langle y,\partial b\rangle
 +\frac16\langle y,[y,y]\rangle.
 \end{aligned}
\end{equation}

\begin{proposition}[Maurer--Cartan form of involutivity]
\label{prop:hol-mc-involutivity}
The equations \eqref{eq:intrinsic-lift-equations} are equivalent to
\begin{equation}\label{eq:hol-involutive-equivalence}
 R_y=0,
 \qquad
 R_b=0.
\end{equation}
\end{proposition}

\begin{proof}
The vector and adjoint components of \(R_y\) are respectively the
integrability tensor of \(D_\mu\) and the curvature of the induced
partial connection. Their vanishing therefore gives the first two
equations in \eqref{eq:intrinsic-lift-equations}.

Once these components vanish, the image of \(G\) in
\(\Omega^3/\mathcal I_\mu^2\) has two remaining components: its
\(T^{*1,0}X\)-valued component and its scalar \((0,3)\)-component.
The linear coefficient is $-2R_x$. Once $R_x=0$, the scalar
coefficient is $R_b$, as computed in Appendix~\ref{app:holomorphic-q}.
Hence
\[
 G=0\pmod{\mathcal I_\mu^2}
 \quad\Longleftrightarrow\quad
 R_x=R_b=0.
\]
Together with the vector and adjoint equations this is
\eqref{eq:hol-involutive-equivalence}. The corresponding local
filtration calculation is given in Appendix~\ref{app:holomorphic-q}.
\end{proof}

\subsection{The derived lifting algebra}

The tangent partial connection is determined by the deformation of the
complex distribution and is therefore not an independent variable. For
\[
 v\in\Omega^{0,\bullet}(T^{1,0}X)
\]
define
\begin{equation}\label{eq:product-natural-graph}
 K_{\Gamma_0}(v)s
 =
 \nabla_vs-[v,s].
\end{equation}
We write
\[
 \kappa_e=K_{\Gamma_0}(\mu),
 \qquad
 \kappa=K_{\Gamma_0}(\mu^{[0]}).
\]
Thus \(\kappa_e\) is \(\mathbb D_{\mathbb C}\)-valued, whereas the
tangent argument in the anomaly pairing is reduced modulo \(\epsilon\).

The closed background volume $\Omega_0$ trivializes the canonical line.
In local holomorphic coordinates with $\Omega_0=dz^1\wedge dz^2\wedge dz^3$,
write $v=v^a\otimes\partial_a$ and define its action on
$s\in\Omega^{0,\bullet}(X)$ by
\begin{equation}\label{eq:canonical-line-action}
 \rho_\Omega(v)s
 =v^a\wedge\partial_a s
   +(\operatorname{div}_{\Omega_0}v)\wedge s,
 \qquad
 \operatorname{div}_{\Omega_0}v=\partial_a v^a.
\end{equation}
These formulas agree under changes of holomorphic volume coordinates.
They give the action on the Dolbeault complex of the canonical line.

\begin{lemma}[Compatibility of the induced tangent and canonical-line actions]
\label{lem:induced-action-compatibility}
Let
\[
 \widehat K(v)=(K_{\Gamma_0}(v),v)
\]
be the corresponding element of the tangent Atiyah complex. For
homogeneous \(v,w\),
\[
 d_{\rm At}\widehat K(v)=\widehat K(\bar\partial v),
 \qquad
 [\widehat K(v),\widehat K(w)]_{\rm At}
 =
 \widehat K([v,w]),
\]
and
\[
 [\rho_\Omega(v),\rho_\Omega(w)]_{\rm gr}
 =
 \rho_\Omega([v,w]),
 \qquad
 [\bar\partial,\rho_\Omega(v)]_{\rm gr}
 =
 \rho_\Omega(\bar\partial v).
\]
\end{lemma}

\begin{proof}
The Atiyah element \((\kappa,v)\) acts on \(T^{1,0}X\) by the
first-order operator
\[
 s\longmapsto\nabla_vs-\kappa(s).
\]
For \(\kappa=K_{\Gamma_0}(v)\), this operator is \(s\mapsto[v,s]\).
The Atiyah differential and bracket are characterized by the
commutator of these first-order operators. The first two identities
therefore follow from
\[
 [\bar\partial,[v,-]]_{\rm gr}=[\bar\partial v,-],
 \qquad
 [[v,-],[w,-]]_{\rm gr}=[[v,w],-].
\]
The representation is faithful, so the corresponding Atiyah elements
agree.

The remaining identities follow from the graded Leibniz rule applied to
\eqref{eq:canonical-line-action} in holomorphic volume coordinates.
\end{proof}

In particular, the tangent graph
\(\kappa_e=K_{\Gamma_0}(\mu)\) is preserved by the unary and binary
operations below, but the ternary operation has no tangent component.

For cochain degree \(p\), set
\begin{equation}\label{eq:hol-graded-carrier}
 Z_p=(x_p,\alpha_p,\mu_p,b_p,s_p),
 \qquad
 -1\leq p\leq4,
\end{equation}
where
\[
 \begin{aligned}
 x_p&\in\Omega^{0,p}(T^{*1,0}X),
 &
 \alpha_p&\in\Omega^{0,p}(\ad P_{G,\mathbb C}),\\
 \mu_p&\in\Omega^{0,p}(T^{1,0}X),
 &
 b_p&\in\Omega^{0,p+1}(X),
 &
 s_p&\in\Omega^{0,p-1}(X).
 \end{aligned}
\]
We set \(\Omega^{0,q}(X)=0\) outside \(0\leq q\leq3\).
The variables \(x_p,\mu_p,b_p,s_p\) are free over
\(\mathbb D_{\mathbb C}\), while \(\alpha_p\) has coefficients in
\((\mathbb D_0)_{\mathbb C}\).

Introduce coefficients \(t_p\) of ghost number \(1-p\) and write
\begin{equation}\label{eq:hol-total-packages}
 Y=\sum_p y_pt_p,
 \qquad
 y_p=(x_p,\alpha_p,\mu_p),
 \qquad
 \mathbf b=\sum_p(-1)^{1-p}b_pt_p,
 \qquad
 \mathbf u=1+\sum_p(-1)^{1-p}s_pt_p.
\end{equation}
In the Courant pairing and bracket, the tangent adjoint component is
always evaluated on the graph
\(\kappa=K_{\Gamma_0}(\mu^{[0]})\).

Define the degree-one derivation \(Q_N=-q\) by
\begin{equation}\label{eq:hol-q}
 \begin{aligned}
 qY&=
 \bar\partial_QY-\frac12[Y,Y]-\frac12\partial\mathbf b,\\
 q\mathbf b&=
 \bar\partial\mathbf b
 -\frac12\langle Y,\partial\mathbf b\rangle
 +\frac16\langle Y,[Y,Y]\rangle,\\
 q\mathbf u&=
 \bar\partial\mathbf u-\rho_\Omega(\mu)\mathbf u.
 \end{aligned}
\end{equation}
For fields of ghost number zero, the Maurer--Cartan equation is
\[
 R_y=0,
 \qquad
 R_b=0,
 \qquad
 q u=0.
\]
On the involutive locus \(R_y=R_b=0\), the last equation is precisely
the closure of the weighted complex volume determined by
\((\mu,u)\).

\begin{theorem}[Derived holomorphic lifting]
\label{thm:derived-holomorphic}
In the formal neighbourhood of the Calabi--Yau standard embedding,
\eqref{eq:hol-q} defines a homological vector field
\[
 Q_N^2=0
\]
on the graded space \eqref{eq:hol-graded-carrier}. Its Taylor operations
satisfy
\[
 \ell_n=0
 \qquad
 (n\geq4).
\]
Its Maurer--Cartan functor is locally identified with the transverse
isotropic-lifting problem of
Proposition~\ref{prop:intrinsic-lift-involutivity}, with tangent graph
\eqref{eq:product-natural-graph}, closed weighted canonical line and the
fixed relative horizontal string presentation.
\end{theorem}

\begin{proof}
Let \(R\) be a differential graded Artin algebra and let \(D\) denote
the sum of the exterior differential on \(X\) and the coefficient
differential. The fields determine an annihilator ideal
\(\mathcal I_\mu\), a weighted decomposable three-form
\(\boldsymbol\Omega\), the induced tangent connection
\[
 \mathcal N
 =
 \Gamma_0+K_{\Gamma_0}(\mu^{[0]}),
\]
and a relative string two-form \(\mathcal B\). Their intrinsic
equations are
\begin{equation}\label{eq:hol-intrinsic-main}
 \begin{gathered}
 D\mathcal I_\mu\subset\mathcal I_\mu,
 \qquad
 D\boldsymbol\Omega=0,
 \qquad
 F_D(A)=0\pmod{\mathcal I_\mu^{[0]}},\\
 D\mathcal B
 +\frac\epsilon4
 \left\{
 \operatorname{CS}_D(A)-\operatorname{CS}_D(\mathcal N)
 \right\}
 =
 0\pmod{\mathcal I_\mu^2}.
 \end{gathered}
\end{equation}
The first equation gives the \(\mu\)-component of
\eqref{eq:hol-q}; partial flatness gives the gauge component;
\(D\boldsymbol\Omega=0\) gives the canonical-line component; and the
two components of the final equation modulo \(\mathcal I_\mu^2\) give
the \(x\)- and \(b\)-components. Lemma~\ref{lem:induced-action-compatibility}
shows that the induced tangent graph is preserved. Thus
\eqref{eq:hol-intrinsic-main} is precisely the intrinsic form of
\eqref{eq:hol-q}.

To prove nilpotence, regard
\[
 \mathscr D=d_X+Q_N
\]
as an operator on the free graded jet algebra, without assuming
\(Q_N^2=0\). The four expressions in
\eqref{eq:hol-intrinsic-main} are respectively the curvature of the
deformed distribution, the canonical-line differential, the gauge
partial connection and the relative string curvature. Their Bianchi
identities imply that applying \(\mathscr D\) to these residuals gives
expressions linear in the same residuals. Extracting their components
in the coordinates
\((Y,\mathbf b,\mathbf u)\) gives
\[
 Q_N^2=0.
\]
The local coefficient calculation is recorded in
Appendix~\ref{app:holomorphic-q}.

Finally, the right-hand side of \eqref{eq:hol-q} is at most cubic.
Hence
\[
 \deg Q_N\leq3,
 \qquad
 \ell_n=0\quad(n\geq4).
\]
The inverse passage from these coordinates to the transverse lift and
its relative gerbe data, including dg-Artin and \v{C}ech descent, is
given in Appendix~\ref{app:tangent-density-gerbe}.
\end{proof}

\begin{corollary}[Holomorphic string algebroid]
\label{cor:hol-string-algebroid}
For the coefficient-restored theory at nonzero real anomaly coefficient
$z$, a Maurer--Cartan solution of the
holomorphic lifting algebra determines an involutive transverse isotropic
lift
\[
 \mathcal L\subset\mathbb E_{H,\mathbb C}.
\]
Therefore
\[
 \mathcal Q_{\mathcal L}:=\mathcal L^\perp/\mathcal L
\]
is a holomorphic string algebroid. The same Maurer--Cartan solution
carries the closed weighted complex volume, so
\[
 \dd\Omega=0.
\]
\end{corollary}

\begin{proof}
Theorem~\ref{thm:derived-holomorphic} identifies Maurer--Cartan solutions
with involutive transverse lifts carrying the closed weighted canonical
line. For $U\in\Gamma(\mathcal L)$ and a class
$[e]\in\Gamma(\mathcal Q_{\mathcal L})$, define
\[
 (\bar\partial_{\mathcal Q_{\mathcal L}}[e])(U)=[U\circ e],
\]
where $\circ$ is the Dorfman operation. Isotropy and involutivity of
$\mathcal L$ make this well-defined on $\mathcal L^\perp/\mathcal L$;
the Courant identities and involutivity give
\[
 \bar\partial_{\mathcal Q_{\mathcal L}}^2=0.
\]
Thus the quotient is the holomorphic string algebroid of the standard
quotient construction
\cite{StringAlgebroids,GaugeStringAlgebroids}.
\end{proof}

\subsection{Relation to the holomorphic heterotic deformation theory}
\label{subsec:unary-structure}

We now compare the lifting algebra with the holomorphic heterotic
\(L_3\) algebra of \cite{FiniteDeformations}. In that construction the
tangent partial connection is an independent variable. Here it is
restricted to the differential graph
\[
 \kappa_e=K_{\Gamma_0}(\mu),
\]
and the weighted canonical-line component is added. The scalar gerbe
hierarchy, including cochain degree \(-1\), is unchanged.

Before imposing the tangent graph, let
\(\ell_i^{\mathrm{FD}}\) denote the operations of
\cite[Appendix~A.3, equation~(A.29)]{FiniteDeformations}, with the
signed anomaly pairing used here. With the present normalization of the
Courant cotangent component,
\[
 \ell_1=-\ell_1^{\mathrm{FD}},\qquad
 \ell_2=\ell_2^{\mathrm{FD}},\qquad
 \ell_3=\ell_3^{\mathrm{FD}}.
\]
The sign in \(\ell_1\) converts the field-equation convention of
\cite[equation~(A.27)]{FiniteDeformations} to the present
Maurer--Cartan convention. The remaining normalization comes from the
cotangent component \(2x\) and
\[
 \operatorname{tr}_T=2\operatorname{tr}_3.
\]
Lemma~\ref{lem:induced-action-compatibility} shows that the unary and
binary operations preserve the tangent graph; the ternary operation has
no tangent component. Thus the operations above restrict to the
holomorphic lifting algebra constructed in the preceding subsection.

We next compare the corresponding tangent complexes. Let \(C_N\) be
the tangent complex of \(Q_N\) at
\[
 x=\alpha=\mu=b=0,\qquad u=1.
\]
Write
\[
 s=\delta u,\qquad K=K_{\Gamma_0},
\]
and locally decompose the background curvatures as
\[
 F_{A_0}=dz^a\wedge F_{A,a},
 \qquad
 R_{\Gamma_0}=dz^a\wedge R_a.
\]
In the coordinates
\(Z_p=(x_p,\alpha_p,\mu_p,b_p,s_p)\) of
\eqref{eq:hol-graded-carrier}, the differential is
\begin{equation}\label{eq:unary-endpoint}
\begin{split}
 d_N^pZ=-\bigl(&\bar\partial x_a
 -\tfrac14\jmath_\epsilon\operatorname{tr}_V(F_{A,a}\wedge\alpha)
 +\tfrac14\jmath_\epsilon\operatorname{tr}_T(R_a\wedge K\mu^{[0]})
 +\tfrac12(-1)^p\partial_ab,\\
 &\bar\partial_{A_0}\alpha+F_{A,a}\wedge(\mu^a)^{[0]},
 \quad
 \bar\partial\mu,\quad
 \bar\partial b,\quad
 \bar\partial s+(-1)^p\operatorname{div}_{\Omega_0}\mu
 \bigr),
 \qquad -1\leq p\leq4.
\end{split}
\end{equation}
Geometric arguments are reduced modulo \(\epsilon\) only in the gauge
component and inside \(\jmath_\epsilon\). The tangent component is
determined by
\begin{equation}\label{eq:unary-tangent-row}
 \bar\partial_{\Gamma_0}K\mu
 +R_a\wedge\mu^a
 =
 K\bar\partial\mu,
\end{equation}
which is the unary identity of
Lemma~\ref{lem:induced-action-compatibility}. Thus the curvature terms
in \eqref{eq:unary-endpoint} remain present at the standard embedding.

For field variations, \(p=1\), the canonical-line component is the
linearization of the weighted complex volume:
\[
 \delta\Omega
 =
 s\Omega_0+\mu^a\wedge\iota_{\partial_a}\Omega_0.
\]
Its exterior derivative has components
\[
 \bar\partial\mu,
 \qquad
 \bar\partial s-\operatorname{div}_{\Omega_0}\mu.
\]
At the other end of the complex, scalar gerbe reducibility gives
\[
 d_N^{-1}c
 =
 \left(\frac12\partial c,-\bar\partial c\right)
\]
in the cotangent and gerbe components of degree zero, while the
canonical-line component extends to
\[
 C_N^4=\Omega^{0,3}(X;\mathbb D_{\mathbb C}).
\]

\begin{proposition}[Identification of the holomorphic complexes]
\label{prop:unary-holomorphic-identification}
Let \(F_K\) be the holomorphic heterotic deformation complex of
\cite{FiniteDeformations}, restricted to the tangent graph
\(\kappa_e=K\mu\) and with the trace and coefficient conventions of
\eqref{eq:unary-endpoint}. Adjoin in degree \(p\) the canonical-line
component
\[
 \Omega^{0,p-1}(X)
\]
with differential
\[
 \bar\partial s+(-1)^p\operatorname{div}_{\Omega_0}\mu,
\]
and denote the resulting complex by \(\widehat F_K\). Its differential
has the opposite sign to \(d_N\). Then
\begin{equation}\label{eq:unary-fd-identification}
 T_p:C_N^p\longrightarrow\widehat F_K^p,
 \qquad
 T_pZ=(-1)^p(x,K\mu,\alpha,\mu,b,s)
\end{equation}
is a chain isomorphism. The maps \(T_p\) are differential operators on
sections rather than morphisms of holomorphic vector bundles.
\end{proposition}

\begin{proof}
The background Bianchi identities and cancellation of the tangent and
gauge characteristic forms imply that the independent-tangent
Dolbeault differential of \cite{FiniteDeformations} squares to zero.
The graph identity \eqref{eq:unary-tangent-row} restricts this
differential to \(\kappa_e=K\mu\). In the gerbe and canonical-line
components, the signs in adjacent degrees cancel because the coefficient
operators $\partial_a$ and $\operatorname{div}_{\Omega_0}$ commute with
$\bar\partial$ in holomorphic volume coordinates.
Multiplication by \((-1)^p\) therefore gives the chain identity for
\(T_p\).

The inverse deletes the dependent tangent component \(K\mu\) and
multiplies the remaining components by \((-1)^p\).
Appendix~\ref{app:unary-comparison} records the independent-tangent
differential.
\end{proof}

The preceding comparison does not collapse the \(L_3\) structure to a
dg Lie algebra. This is already visible at the flat background.

\begin{example}[Constant modes on the flat torus]
\label{ex:flat-constant-modes}
Let
\[
 X=
 \bigl(\mathbb C/(2\pi\mathbb Z+2\pi i\mathbb Z)\bigr)^3
\]
with its flat Calabi--Yau structure and trivial background connections.
Forms with constant coefficients form an \(L_3\) subalgebra of the
holomorphic lifting algebra. Its differential vanishes, its geometric
and canonical-line components are abelian, and on constant gauge
parameters of cochain degree zero,
\[
 \ell_2(A,B)=-[A,B],\qquad
 \ell_3(A,B,C)
 =
 -\frac{\epsilon}{4}\operatorname{tr}_7(A[B,C]).
\]
Let \(E_{ij}\) be the standard skew matrix in the \(ij\)-plane. Since
\[
 [E_{23},E_{31}]=-E_{12},
 \qquad
 \operatorname{tr}_7(E_{12}^2)=-2,
\]
we obtain
\begin{equation}\label{eq:flat-nonzero-ternary}
 \ell_3(E_{12},E_{23},E_{31})
 =
 -\frac{\epsilon}{2}\neq0.
\end{equation}
Thus the holomorphic lifting algebra has a nonzero ternary bracket even
at a flat background.
\end{example}

\subsection{Hamiltonian formulation}
\label{subsec:hol-hamiltonian}

Set $c_{\mathrm{an}}=\epsilon/4$ and define the variational one-form
\begin{equation}\label{eq:hol-hamiltonian-potential}
 \begin{aligned}
 \vartheta_N&=\int_X\mathbf u\,\lambda_N\wedge\Omega_0,\\
 \lambda_N&=-\delta\mathbf b+\mu^a\delta x_a+x_a\delta\mu^a
 +c_{\mathrm{an}}\operatorname{tr}_V(\alpha\delta\alpha)
 -c_{\mathrm{an}}\operatorname{tr}_T(\kappa\delta\kappa).
 \end{aligned}
\end{equation}
Define
\begin{equation}\label{eq:hol-hamiltonian-definition}
 \omega_N=\delta\vartheta_N,
 \qquad
 I_N=\iota_{Q_N}\vartheta_N.
\end{equation}

\begin{proposition}[Hamiltonian identities]
\label{prop:hol-hamiltonian}
The variational Hamiltonian structure satisfies
\[
 \iota_{Q_N}\omega_N=-\delta I_N,
 \qquad
 \delta\omega_N=0,
 \qquad
 \mathcal L_{Q_N}\omega_N=0.
\]
The contraction map of $\omega_N$ has a finite local inverse near the
reference background.
\end{proposition}

\begin{proof}
These identities and the finite local inverse are verified in
Appendix~\ref{app:hol-hamiltonian-identities}.
\end{proof}
\section{Hamiltonian reduction to the holomorphic theory}
\label{sec:hamiltonian-reduction}
\label{sec:product-correspondence}

The product variational theory and the holomorphic lifting theory are
written in different geometric variables. The former uses the Hull
connection and real representatives, while the latter uses the
Hermitian tangent connection and complex partial-connection data. We
compare them in three steps. First, we replace the Hull representative
by an equivalent Hermitian one and add contractible variables needed to
compare the two presentations. We then identify their symplectic forms
and Hamiltonians by local first-order transformations. Finally, we
remove the remaining representative directions by characteristic
reduction.

\subsection{Representative change and stabilization}
\label{sec:representatives}

Recall that, for a coisotropic inclusion
\[
 i:C\hookrightarrow(M,\omega),
\]
a projection \(\pi:C\to M_{\rm red}\) is a characteristic quotient when
\[
 \ker D\pi=\ker i^*\omega.
\]
The restriction to the product locus is of this form. Let
\(\mathcal M_\circ\) be the circle-invariant variational theory after
auxiliary elimination and set
\begin{equation}\label{eq:product-first-span}
 \begin{gathered}
 n_\times=(R-1,v,b_1,a_0;\xi^7,\Lambda_7),
 \qquad
 C_\times=\{n_\times=0\},\\
 \mathcal M_\circ
 \xleftarrow{i_\times}
 C_\times
 \xrightarrow{\pi_\times}
 \mathcal M_\times.
 \end{gathered}
\end{equation}
Here \(\xi^7\) and \(\Lambda_7\) are the circle components of the
symmetry parameters, and \(\pi_\times\) forgets the momenta conjugate to
\(n_\times\). The restricted Hamiltonian is independent of these normal
momenta, so the quotient is the product variational theory of
\eqref{eq:product-functional}. It does not impose the
Euler--Lagrange equations normal to the product locus; these are
recorded in Appendix~\ref{app:coisotropic}. The auxiliary elimination
is compatible with the circle action by
Appendix~\ref{app:parent-elimination}.

The product theory is expressed using the Hull connection
\(\Theta_H\), whereas the holomorphic lifting theory uses the Hermitian
connection \(\Gamma\) of Section~\ref{subsec:product-geometry}. Passing
from one to the other changes both the tangent representative and the
gerbe coordinate. Put
\[
 t=(\Theta_H-\Gamma)^{[0]},
 \qquad
 a_H=A-\Theta_H^{[0]},
 \qquad
 a_\Gamma=A-\Gamma^{[0]}.
\]
The corresponding connection and two-form change is
\begin{equation}\label{eq:representative-T0}
 T_0:\qquad
 a_\Gamma=a_H+t,
 \qquad
 B_\Gamma
 =
 B_H-\frac{\epsilon}{4}P(t,a_H).
\end{equation}
Let \(\mathcal T_{3,P}\) denote the transgression
\eqref{eq:ii-transgression-convention}. Its triangle identity gives
\begin{equation}\label{eq:representative-triangle}
 \mathcal T_{3,P}(\Gamma^{[0]},a_\Gamma)
 =
 \mathcal T_{3,P}(\Theta_H^{[0]},a_H)
 +\mathcal T_{3,P}(\Gamma^{[0]},t)
 -\dd P(t,a_H).
\end{equation}

Let \(W_H\) and \(W_\Gamma\) denote the product functionals in the two
representatives, with the metric and \(SU(3)\)-structure fixed. Then
\begin{equation}\label{eq:representative-defect}
 \begin{aligned}
 W_H\circ T_0^{-1}
 &=
 W_\Gamma+\epsilon\mathcal D_\Gamma,\\
 \mathcal D_\Gamma
 &=
 \frac{L}{16}\operatorname{Im}
 \int_X
 \mathcal T_{3,P}(\Gamma^{[0]},t)
 \wedge\Omega^{[0]}.
 \end{aligned}
\end{equation}
Thus the change of representative leaves a first-order defect
\(\mathcal D_\Gamma\). The next proposition removes it by a local
finite-jet change of the geometric variables.

\begin{proposition}[Representative equivalence]
\label{thm:representative-equivalence}
There is an equivariant vector field \(\mathcal V_\Gamma\), depending
on finitely many derivatives and vanishing at the reference background,
such that
\[
 T_1(q)
 =
 q+\jmath_\epsilon\mathcal V_\Gamma(q^{[0]}),
 \qquad
 T=T_1\circ T_0,
\]
satisfies
\begin{equation}\label{eq:representative-equivalence}
 T^*W_\Gamma=W_H,
 \qquad
 T_*f_H=f_\Gamma.
\end{equation}
The variational cotangent lift of \(T\) is therefore a local real
formal Hamiltonian \(Q\)-isomorphism to first order in \(\epsilon\).
\end{proposition}

\begin{proof}
Appendix~\ref{app:representative-change} proves that the defect in
\eqref{eq:representative-defect} factors through the leading
Euler--Lagrange covector:
\[
 \mathcal D_\Gamma
 =
 \int_X E_0(\mathcal V_\Gamma).
\]
Since \(T_1\) fixes leading coefficients and \(\epsilon^2=0\),
\[
 T_1^{-1}(q)
 =
 q-\jmath_\epsilon\mathcal V_\Gamma(q^{[0]}),
 \qquad
 T_1^*W_\Gamma
 =
 W_\Gamma+\epsilon\mathcal D_\Gamma.
\]
Together with \eqref{eq:representative-defect}, this gives
\(T^*W_\Gamma=W_H\).

The same appendix proves that \(\mathcal V_\Gamma\) is equivariant and
vanishes at the reference background. Hence \(T_*f_H=f_\Gamma\).
Finally, the cotangent lift
\[
 p_H=(DT)^\dagger p_\Gamma
\]
preserves the Liouville form, so the corresponding Hamiltonians are
identified.
\end{proof}

The product and holomorphic theories still treat the choice of Hermitian
frame and gerbe representative differently. To compare their Hamiltonian
structures on the same graded space, we first adjoin these representative
directions to the product theory as contractible cotangent pairs. Let
\[
 \mathcal U
\]
denote their sum and let \(\widetilde S\) denote the Hamiltonian on
\[
 \widetilde{\mathcal M}=\mathcal M_\times\oplus\mathcal U.
\]
The added coordinates consist of a Hermitian representative \(r_H\), a
real gerbe one-form \(m\), and its scalar parameter \(r_B\), together
with their contractible partners. Explicitly, this requires seven
cotangent pairs over \(\mathbb D\) and twenty-one over
\(\mathbb D_0\), whose differentials and contractions are given in
\eqref{eq:product-unit-pairs}.

On \(\widetilde{\mathcal M}\), the representative change of
Proposition~\ref{thm:representative-equivalence}, followed by the
common-frame and relative-gerbe coordinate changes and a cotangent
polarization, puts the product theory in the coordinates of the
holomorphic presentation. Appendix~\ref{app:canonical-maps} constructs
these transformations and their local inverses.

\subsection{Symplectic and Hamiltonian comparison}

In this comparison, $\mathcal M_N$ denotes the realification of the
holomorphic lifting space. We use $\vartheta_N,\omega_N,I_N$ for
$\frac L4\operatorname{Im}$ of the complex forms and functional in
Section~\ref{subsec:hol-hamiltonian}, respectively.

Let \(\mathcal R\) be the space of representative directions. Write
\[
 z=(r_H,m,r_B)\in\mathcal R
\]
for the representative coordinates and
\(\varpi\) for their conjugate momenta. Let
\(\omega_{\mathcal R}\) denote the corresponding canonical cotangent
form.

In the common coordinates, write
\[
 S_C=S_0+\jmath_\epsilon S_{1,C},
 \qquad
 I_N=I_0+\jmath_\epsilon I_1,
\]
where \(I_N\) is the holomorphic Hamiltonian of
\eqref{eq:hol-hamiltonian-definition}, with the real normalization of
\eqref{eq:product-dressed-primitive}. The leading differential \(Q_0\) preserves
\[
 C^{[0]}
 =
 \{\varpi^{[0]}=0\}.
\]

There are two first-order discrepancies between the stabilized product
theory and the holomorphic lifting theory. We first remove the
difference between their symplectic forms.

Let \(\widehat N\) be the total tangent connection determined by
\eqref{eq:product-natural-graph} in the common graded coordinates. Put
\[
 \mathcal C=\Gamma_C+j,
 \qquad
 b_C=\mathcal C-\widehat N,
 \qquad
 \widehat\Omega
 =
 e^{-\iota_{\widehat\xi}}\boldsymbol\Omega,
\]
where \(j\) is the tangent-frame ghost and \(\widehat\xi\) is the transformed
diffeomorphism ghost. Define
\begin{equation}\label{eq:comparison-beta}
 \beta
 =
 \frac{L}{16}\operatorname{Im}
 \int_X
 P\!\left(
 b_C,
 (\delta+\iota_{\delta\widehat\xi})(\mathcal C+\widehat N)
 \right)
 \widehat\Omega.
\end{equation}
The two symplectic forms differ by the exact first-order term
\[
 \omega_C
 =
 \omega_N+\omega_{\mathcal R}
 +\jmath_\epsilon\delta\beta.
\]

Let \(\omega_0\) be the leading symplectic form on the free geometric
variables and their duals. Since its contraction map is invertible,
there is a local vector field \(V\) satisfying
\[
 \iota_V\omega_0=-\beta.
\]
For
\[
 \Phi=\id+\jmath_\epsilon V,
\]
the first-order Moser calculation gives
\begin{equation}\label{eq:comparison-moser}
 \Phi^*\omega_C
 =
 \omega_N+\omega_{\mathcal R},
 \qquad
 VS_0=-\iota_{Q_0}\beta.
\end{equation}
The second identity is the induced first-order change in the
Hamiltonian. Lemma~\ref{lem:product-relative-moser} proves
\eqref{eq:comparison-moser} and the locality of \(V\).

After this symplectic correction, the remaining difference is a
relative tangent Chern--Simons term. Set
\[
 \widehat D
 =
 \dd+Q_0-\mathcal L_{\widehat\xi}.
\]
On \(C^{[0]}\),
Lemma~\ref{lem:product-source-total} gives
\[
 \widehat D^2=0,
 \qquad
 \widehat D\widehat\Omega=0.
\]
The relative transgression calculation then produces a local
functional \(\mathcal G\) of ghost number \(-1\) such that
\begin{equation}\label{eq:comparison-current-coboundary}
 \left.
 \bigl(S_1'-I_1\bigr)
 \right|_{C^{[0]}}
 =
 Q_0\mathcal G,
\end{equation}
where \(S_1'\) is defined by
\[
 \Phi^*S_C=S_0+\jmath_\epsilon S_1'.
\]
An explicit expression for \(\mathcal G\) is given
in \eqref{eq:product-current-generator}.

The proof of \eqref{eq:comparison-current-coboundary} is found in
Appendix~\ref{app:canonical-maps}. Its two relevant vanishings have
simple geometric origins: one transgression has horizontal degree
greater than six, while the other contains a factor in the annihilator
of \(\widehat\Omega\).

Extend \(\mathcal G\) independently of the representative momenta and
let
\[
 \Psi
 =
 \id+X_{\jmath_\epsilon \mathcal G}.
\]
For a functional \(G\) of ghost number \(-1\),
\[
 X_GS_0=-Q_0G,
\]
so the corresponding canonical transformation gives
\begin{equation}\label{eq:comparison-canonical}
 \begin{aligned}
 \widehat S
 &=
 \Psi^*\Phi^*S_C\\
 &=
 S_0+\jmath_\epsilon
 \bigl(S_1'-Q_0\mathcal G\bigr).
 \end{aligned}
\end{equation}
The leading Hamiltonians agree on $C=\{\varpi=0\}$ by the
comparison in Appendix~\ref{app:product-momenta}. Since the first-order
difference depends only on leading coefficients,
\eqref{eq:comparison-current-coboundary} gives
$\widehat S|_C=I_N$.

Both corrections act only in the ideal-valued free summands and are
the identity at leading order. They are therefore locally invertible by
reversing their first-order terms. Together with the coordinate changes
constructed in Appendix~\ref{app:canonical-maps}, they give a local
isomorphism
\[
 \mathscr T:\widetilde{\mathcal M}\xrightarrow{\sim}\widehat{\mathcal M}.
\]

\subsection{The characteristic quotient}

The remaining variables distinguish real representatives of the same
holomorphic data. In the final coordinates, let \(q\) denote the
holomorphic lifting variables and set
\begin{equation}\label{eq:product-final-span}
 C=\{\varpi=0\},
 \qquad
 i(q,z)=(q,z,0),
 \qquad
 \pi(q,z)=q.
\end{equation}
The constructions above fit into
\begin{equation}\label{eq:product-diagram}
 \begin{gathered}
 \mathcal M_\circ
 \xleftarrow{i_\times}
 C_\times
 \xrightarrow{\pi_\times}
 \mathcal M_\times,\\[2mm]
 \mathcal M_\times
 \xleftarrow{\operatorname{pr}}
 \widetilde{\mathcal M}=\mathcal M_\times\oplus\mathcal U
 \xrightarrow[\cong]{\mathscr T}
 \widehat{\mathcal M},\\[2mm]
 \widehat{\mathcal M}
 \xleftarrow{i}
 C
 \xrightarrow{\pi}
 \mathcal M_N.
 \end{gathered}
\end{equation}
The middle projection removes contractible factors. The first and last
spans are characteristic reductions.

\begin{theorem}[Hamiltonian realization]
\label{thm:hamiltonian-realization}
\label{prop:product-canonical}
\label{thm:main-d}
The map
\[
 \mathscr T:\widetilde{\mathcal M}\xrightarrow{\sim}\widehat{\mathcal M}
\]
is a local real formal Hamiltonian isomorphism, with
\[
 \mathscr T^*\widehat S=\widetilde S.
\]
On \(C\),
\begin{equation}\label{eq:final-coisotropic-identities}
 \begin{aligned}
 i^*\widehat\omega
 &=
 \pi^*\omega_N,\\
 i^*\widehat S
 &=
 \pi^*I_N,\\
 \ker i^*\widehat\omega
 &=
 \ker D\pi.
 \end{aligned}
\end{equation}
Hence the characteristic quotient carries the homological vector field
\(Q_N\) of the holomorphic lifting theory in every cochain degree.
\end{theorem}

\begin{proof}
Proposition~\ref{thm:representative-equivalence} identifies the Hull
and Hermitian product theories. The stabilization and coordinate
changes of Appendix~\ref{app:canonical-maps}, followed by
\eqref{eq:comparison-moser} and
\eqref{eq:comparison-canonical}, identify their symplectic forms and
Hamiltonians with those of \(\widehat{\mathcal M}\). This gives the local
Hamiltonian isomorphism \(\mathscr T\).

In the final coordinates,
\[
 \widehat\omega
 =
 \omega_N+\omega_{\mathcal R}.
\]
Restricting to \(\varpi=0\) gives
\[
 i^*\widehat\omega
 =
 \pi^*\omega_N,
 \qquad
 \ker i^*\widehat\omega
 =
 \ker D\pi
 =
 \langle\partial_z\rangle.
\]
Equation~\eqref{eq:comparison-current-coboundary} gives
\[
 i^*\widehat S
 =
 \pi^*I_N.
\]
In particular, the restricted Hamiltonian is independent of
\(z\). Its Hamiltonian vector field is therefore tangent to
\(C\) and descends through \(\pi\). On the quotient it satisfies
\[
 \iota_Q\omega_N=-\delta I_N.
\]
The Hamiltonian identities in Section~\ref{subsec:hol-hamiltonian}
identify this vector field with
\(Q_N\).
\end{proof}

\begin{corollary}[Linearized correspondence]
\label{prop:unary-source-span}
Let \(\mathcal C_\times\) be the tangent complex of the product theory
and \(\mathcal C_{N,\mathbb R}\) the realified holomorphic tangent complex.
Differentiating \eqref{eq:product-diagram} gives a correspondence of
these tangent complexes, with contractible stabilization and a surjective
characteristic projection in every degree. The characteristic
projection need not be a quasi-isomorphism. In degree \(-1\),
\begin{equation}\label{eq:unary-real-scalar-obstruction}
 \mathcal H^{-1}(\mathcal C_\times)
 =
 \underline{\mathbb D},
 \qquad
 \mathcal H^{-1}(\mathcal C_{N,\mathbb R})
 =
 \underline{\mathbb D}^{\,2}.
\end{equation}
\end{corollary}

\begin{proof}
Differentiate \eqref{eq:product-diagram} at the common stationary
background. The stabilization contributes only contractible summands,
and the final characteristic projection is surjective in each degree.

In degree \(-1\), the real product variational complex has the locally
constant real scalar gerbe parameters. The realification of the
holomorphic complex has the corresponding complex scalar parameters,
hence two real copies. Neither complex has a term in degree \(-2\).
This gives \eqref{eq:unary-real-scalar-obstruction}.
\end{proof}

The quotient therefore captures the holomorphic deformation theory. It
does not impose balance, gauge primitivity or leading-flux
admissibility. These are additional conditions on a real representative
and are not constant along the characteristic fibres by
Proposition~\ref{prop:nonbasic-bps}. Consequently the reduction does
not, by itself, produce a physical representative over every
holomorphic solution.

\section{Consequences and open problems}
\label{sec}

The finite \(L_3\) model constructed above belongs to the resolved
theory, not to the physical functional obtained after eliminating the
auxiliary fields. The resolved Hamiltonian is quartic and its
homological vector field has no Taylor operations of arity greater than
three, whereas the flat-torus physical functional has a nonzero quintic Taylor
coefficient in the local physical coordinates of
Section~\ref{subsec:quintic-coefficient}. Elimination of auxiliary fields can therefore increase
polynomial arity. The construction gives a cyclic \(L_3\) presentation
in the stabilization class considered here, but does not determine
whether a cyclic dg Lie presentation exists. Example~\ref{ex:flat-constant-modes} gives a nonzero ternary bracket
at a flat background.

Furthermore, the holomorphic heterotic theory arises as a characteristic quotient of
the product variational theory, rather than as an equivalent
deformation problem. The quotient removes directions associated with
the choice of real representative and recovers the holomorphic
equations together with their gauge and identity structure. Conformal
balance, gauge primitivity and leading-flux admissibility are not
constant along these characteristic fibres. The same difference is
visible after linearization: the degree-\(-1\) cohomology of the real
product complex does not agree with that of the realified holomorphic
complex.

This leaves a lifting problem that is not addressed by the reduction.
Given a Maurer--Cartan solution of the holomorphic theory, one may ask
whether its characteristic fibre contains a real representative
satisfying balance, primitivity and the leading-flux condition, and
how such representatives are parametrized when they exist. Existence
and uniqueness of such representatives require additional geometric
and analytic input beyond the characteristic reduction.

There is also a separate question about the algebraic form of the
resolved theory. The present construction gives a finite cyclic \(L_3\)
model, but it remains open whether the ternary operation can be removed
by a further finite stabilization compatible with the cyclic structure.
Beyond first order in \(\alpha'\), one can ask instead whether a
finite-arity presentation continues to exist once the higher-order
anomaly and connection corrections are included.
\section*{Acknowledgements}
Selected tensor, sign, and coefficient identities were checked by exact symbolic computation.
OpenAI ChatGPT/Codex assisted with checking derivations, developing verification scripts, and
drafting the manuscript. Responsibility for the arguments and mathematical claims rests with
the author.
\clearpage
\appendix
\section{Auxiliary elimination and Hamiltonian stabilization}
\label{app:parent-elimination}\label{app:auxiliary}\label{sec:equivalence}\label{app:momenta}
The auxiliary variables are eliminated at the Hamiltonian level. We first
show that the four constraints form local coordinates in the auxiliary
directions, then apply the standard generalized-auxiliary-field
construction \cite{BBH,GrigorievParent,GrigorievRudinsky} to split off
the resulting contractible factors. All adjoints are variational adjoints
on the closed manifold.

For a finite-rank graded bundle $F\to Y$, an integer $k$, and
$R\in\{\mathbb D,\mathbb D_0\}$, write
\[
 C(F,k;R)^p=
 \bigl(\Gamma(F^{p-k})\oplus\Gamma(F^{p-k-1})\bigr)\otimes R,
 \qquad d_C(a,b)=(0,a),
\]
Let $\mathscr C(F,k;R)$ be the formal graded base associated with the
suspension of this complex, and set
\[
 \mathcal K(F,k;R)=T^*[-1]\mathscr C(F,k;R).
\]
For suspended coordinates $u,v$ and shifted covectors $p_u,p_v$,
\begin{equation}\label{eq:elementary-factor}
 \vartheta_{\mathcal K}
 =\int_Y(\langle p_u,\delta u\rangle+\langle p_v,\delta v\rangle),
 \qquad
 \mathcal S_{\mathcal K}=\int_Y\langle p_v,u\rangle.
\end{equation}
For $R=\mathbb D_0$, the covectors take values in $\epsilon\mathbb D$.
The identity complex is contracted by $h(a,b)=(b,0)$, and the shifted
cotangent lift gives the corresponding contractible Hamiltonian factor.

The four constraints take values in
\[
 F_C=(\Lambda^7T^*Y\otimes\End E)
 \oplus(T^*Y\otimes\Delta)
 \oplus(\Lambda^2T^*Y\otimes E)\oplus\Lambda^3T^*Y.
\]

\subsection{Constraint coordinates}

For a fixed invertible coframe $e=(e^a)$, let $E_a$ be the dual frame and
set
\begin{equation}\label{eq:aux-pointwise-maps}
\begin{aligned}
 L_e(k)^a&=k^a{}_b\wedge e^b,&
 K_e(h)^a&=\tfrac12h^a{}_{bc}e^b\wedge e^c,\\
 E_3(h)&=\tfrac16h_{abc}e^a\wedge e^b\wedge e^c,&
 A_e(t)&=\tfrac13e_a\wedge t^a.
\end{aligned}
\end{equation}
Then $A_eK_e=E_3$, and both $L_e$ and $E_3$ are pointwise
isomorphisms. For example,
\[
 (L_e^{-1}t)_{iab}=\tfrac12(t_{aib}+t_{bai}-t_{iba}).
\]

\begin{lemma}[Constraint coordinate inverse]\label{lem:constraint-inverse}
For fixed $q=(e,\zeta,B,\nu,A)$, the map
\[
 y=(U,\Xi,\Theta,h)\longmapsto
 v=(C_U,C_\Xi,T,C_H)
\]
has a local formal inverse. At prescribed values
$v=(v_U,v_\Xi,v_T,v_H)$ it is
\begin{equation}\label{eq:aux-full-inverse}
\begin{aligned}
 h^u&=E_3^{-1}(\dd B+v_H+A_ev_T),\\
 \Theta^u&=L_e^{-1}(v_T-\dd e-K_eh^u),\\
 t^u&=\cT_3(\red{\Theta^u},A-\red{\Theta^u}),\\
 h&=h^u-\tfrac\eps4E_3^{-1}t^u,\\
 \Theta&=L_e^{-1}(v_T-\dd e-K_eh),\\
 U_b&=\sum_a\iota_{E_a}(\delta^a_b\nu+(v_U)^a{}_b),\\
 \Xi&=\dd\zeta+\sg(\Theta)\zeta+v_\Xi.
\end{aligned}
\end{equation}
The linearization $\mathcal J=D_yC$ and its adjoint have local inverses
of finite differential order.
\end{lemma}

\begin{proof}
The row combination $C_H+A_eT$ is
\[
 E_3h-\dd B+\frac\epsilon4\cT_3.
\]
This determines $h$, after which the torsion equation determines
$\Theta$. Since the transgression is multiplied by $\epsilon$, its
argument is needed only modulo $\epsilon$, which gives the first four
lines of \eqref{eq:aux-full-inverse}. The $U$ and $\Xi$ equations are
then triangular.

For the linearization, the $(\Theta,h)$ block becomes triangular after
the same row operation, with diagonal entries $L_e$ and $E_3$. The
remaining part is pointwise invertible. Writing
$\mathcal J=U+\epsilon V$, with $U$ this pointwise invertible operator,
\[
 \mathcal J^{-1}=U^{-1}-U^{-1}\epsilon VU^{-1},
 \qquad
 (\mathcal J^\dagger)^{-1}=(\mathcal J^{-1})^\dagger.
\]
\end{proof}

\subsection{Hamiltonian elimination}

Write
\[
 W_{\rm aux}=W_{\rm geom}(q,y)+\langle m,C(q,y)\rangle
\]
and include the surviving ghosts among $q$ and their covectors among
$\pi$. Write $K^\dagger\pi$ for the signed variational derivative in
$y$ of $\iota_f\vartheta_B$ on $\pi_y=\pi_m=0$, with the cotangent
signs of Appendix~\ref{app:cyclic-signs}. The auxiliary Hamilton
equations there are
\begin{equation}\label{eq:aux-euler-equations}
 C=0,
 \qquad
 W_{{\rm geom},y}+\mathcal J^\dagger m+K^\dagger\pi=0.
\end{equation}
Thus, with $y_{\rm gr}(q)$ denoting
\eqref{eq:aux-full-inverse} at $v=0$,
\[
 m_{\rm gr}
 =-(\mathcal J_{\rm gr}^\dagger)^{-1}
   (W_{{\rm geom},y}+K^\dagger\pi)
   \big|_{y=y_{\rm gr}(q)}.
\]
The resulting section
\[
 \iota_{\rm aux}(q,\pi)
 =(q,y_{\rm gr}(q),m_{\rm gr}(q,\pi),\pi,0,0)
\]
is invariant under $Q_B$, and
\begin{equation}\label{eq:aux-pullbacks}
 \iota_{\rm aux}^*\vartheta_B=\vartheta_q,
 \qquad
 \iota_{\rm aux}^*\cS_B
 =W_{\rm geom}(q,y_{\rm gr}(q))
  +\iota_{f_q|_{y=y_{\rm gr}(q)}}\vartheta_q.
\end{equation}
Indeed, \eqref{eq:aux-euler-equations} is regular by
Lemma~\ref{lem:constraint-inverse}, and the displayed formula is simply
its solution in $(y,m)$. The ideal generated by
$\pi_y,\pi_m,Q_B\pi_y,Q_B\pi_m$ and their jets is $Q_B$-stable because
$Q_B^2=0$, which proves invariance of the section.
Pullback of the Liouville form gives the first
identity in \eqref{eq:aux-pullbacks}; the second follows because the
constraints and auxiliary covectors vanish on the section.

The normal variables are now generalized auxiliary fields in the sense
of \cite{BBH,GrigorievParent,GrigorievRudinsky}. Taking
$v=C(q,y)$ as normal coordinates and shifting the multiplier by the
first variational remainder puts the functional in the form
$W_{\rm red}+\langle M,v\rangle$. Both changes are lifted to the
cotangent space. The Hamiltonian section may still have a
momentum-dependent $M$-coordinate. If that coordinate is $a(q,\pi)$,
the new-to-old canonical transformation generated by
$\langle\pi_M,a\rangle$ straightens the section. The generator has
formal degree at least three, since $a$ contains a surviving momentum
and a ghost. In the resulting coordinates,
\[
 \cS=\cS_{\rm red}+\langle M,v\rangle+R,
 \qquad R\in I_{\rm aux}^2\cap\mathfrak m^3,
\]
where $I_{\rm aux}$ is the ideal of the normal coordinates and their
jets, and $\mathfrak m$ is the ideal of the stationary background.
Invariance of the symplectic zero section excludes terms linear in the
normal coordinates. The elementary factor has the counting contraction
associated to \eqref{eq:elementary-factor}. Together with the master
equation, it removes each successive homogeneous term of $R$ by a
canonical transformation fixing that section. Consequently the auxiliary
normal directions split off the factor
\[
 \mathcal K(F_C,1;\mathbb D).
\]
This construction is local at each formal order because the only
inverse used before the contraction is the finite-order inverse of
Lemma~\ref{lem:constraint-inverse}.

\subsection{Scale and frame factors}

After auxiliary elimination, two further choices remain: the independent
normalization of $(\nu,\zeta)$ and the orthonormal coframe. Both give
free contractible factors.

For the scale symmetry set
\[
 \rho=\nu/\vol_e,
 \qquad t=\log\rho,
 \qquad \zeta_\ph=e^{t/2}\zeta,
 \qquad \eta_s=ft.
\]
Then $f\eta_s=0$ and the inverse ghost change is
$\sigma=\eta_s+\xi(t)$. The cotangent lift gives the elementary pair
\[
 u=-\eta_s,
 \qquad v=t,
 \qquad p_u=-p_{\eta_s},
 \qquad p_v=p_t,
\]
with Hamiltonian $\langle p_v,u\rangle$. Hence the density-scale
symmetry splits off
\[
 \mathcal K(\mathbb R,0;\mathbb D),
\]
and the quotient variable is the weighted spinor $\zeta_\ph$, with
$\varrho=|\zeta_\ph|^2\vol_e$.

For the frame symmetry use the right polar decomposition
\[
 E_{\rm fr}=R_{\rm fr}L,
 \qquad
 L=(E_{\rm fr}^{\mathsf T}E_{\rm fr})^{1/2},
\]
and put
\[
 q_{\rm fr}=\log R_{\rm fr},
 \qquad
 \eta_{\rm fr}=f(q_{\rm fr}).
\]
The frame action is locally free, so this is an invertible local ghost
change. Its cotangent lift again has the elementary form
\[
 u=-\eta_{\rm fr},
 \qquad v=q_{\rm fr},
 \qquad p_u=-p_{\eta_{\rm fr}},
 \qquad p_v=p_{q_{\rm fr}}.
\]
Transporting the gauge and gerbe coordinates through the same frame
change therefore splits off
\[
 \mathcal K(\mathfrak{so}(E),0;\mathbb D).
\]
The quotient is precisely the weighted-spinorial presentation used in
\cite{BrongersPaperI}.

\begin{theorem}[Hamiltonian stabilization]\label{thm:main-c}\label{thm:parent-equivalence}
Let $\mathcal M_{\rm red}$ be the physical-coordinate variational
Hamiltonian theory of \cite[Section~4]{BrongersPaperI}. Near the
stationary standard embedding,
\[
 \mathcal M_B\cong_Q\mathcal M_{\rm red}
 \times\mathcal K(F_C,1;\mathbb D)
 \times\mathcal K(\mathbb R,0;\mathbb D)
 \times\mathcal K(\mathfrak{so}(E),0;\mathbb D).
\]
The isomorphism and its inverse have finite differential order at each
formal degree and carry the distinguished physical inclusion to the
corresponding zero section.
\end{theorem}

\begin{proof}
The first factor is the Hamiltonian elimination above. The scale and
frame coordinate changes give the remaining two factors. On their
quotient, Proposition~\ref{prop:source-identity} gives $W_{\rm ext}$,
and the transported gauge hierarchy is the one of
\cite[Appendix~D]{BrongersPaperI}. The variational cotangent lifts
therefore identify the full Hamiltonians.
\end{proof}

All maps are equivariant under circle translation of the product
background, so the construction restricts to the circle-invariant theory
used in Section~\ref{sec:hamiltonian-reduction}.

\section{Auxiliary Hamiltonian signs and coefficients}
\label{app:cyclic-signs}\label{app:mixed}

We use the shifted-cotangent and suspension conventions of
\cite[Section~3 and Appendix~C]{BrongersPaperI}. Only the mixed
$\mathbb D_0/\epsilon\mathbb D$ duality and the two Taylor coefficients
used in the main text are recorded here.

For a base coordinate $q$ of the graded space $\mathscr B$, let $\pi_q$
be its shifted density covector. With covectors written first,
\begin{equation}\label{eq:mixed-potential-small}
 \vartheta_B=
 \int_Y\left(
   \sum_{q\;\mathbb D\text{-valued}}\langle\pi_q,\delta q\rangle
   +\epsilon\sum_{q\;\mathbb D_0\text{-valued}}
      P(\bar\pi_q,\delta q^{[0]})
 \right),
 \qquad \omega_B=\delta\vartheta_B.
\end{equation}
Here $\pi_q=\jmath_\epsilon(\bar\pi_q)$ on a residue summand. Thus
\[
 \iota_f\vartheta_B
 =\int_Y\sum_q(-1)^{|\pi_q|}\langle\pi_q,fq\rangle,
\]
and the cotangent component of the Hamiltonian vector field is the
signed variational adjoint of $D(fq)$. For a differential coordinate
change $q'=T(q)$, its cotangent lift is
\begin{equation}\label{eq:actual-cotangent-lift}
 \pi=(DT)^\dagger\pi',
 \qquad
 \int\langle\pi',\delta T(q)\rangle
 =\int\langle\pi,\delta q\rangle
 \quad(\bmod\ \dd_{\rm hor}).
\end{equation}
This convention is used throughout Appendices~\ref{app:parent-elimination}
and~\ref{app:canonical-maps}.

\subsection{Taylor coefficients}
\label{app:cyclic-conventions}

At the stationary background the desuspended pairing is
\begin{equation}\label{eq:cyclic-pairing-sign}
 \langle x,y\rangle_B=(-1)^{|x|}\omega_B(sx,sy),
 \qquad |sx|=|x|-1,
\end{equation}
with the polarization convention of
\cite[Appendix~C.4]{BrongersPaperI}. Applied to the cubic scalar term
\(
 \epsilon P(\lambda,[\lambda,\lambda])/24
\), this gives
\[
 \ell_3(\lambda_1,\lambda_2,\lambda_3)
 =-\frac\epsilon4P(\lambda_1,[\lambda_2,\lambda_3]),
\]
as stated in Section~\ref{sec:cyclic-l3}.

\paragraph{The quintic coefficient.}\label{app:interaction-coefficients}
Let $\chi_0$ be the parallel background spinor and put
\[
 u(t)=\chi_0+t\,c_0(v^\flat)\chi_0.
\]
The Clifford identities and the vanishing one-form bilinear give
\[
 \Bsp(u(t),u(t))=1+t^2|v|_0^2,
 \qquad
 \Bsp(u(t),\mathscr D_0u(t))=-t\operatorname{div}_0v+t^2Q(v),
\]
where
\[
 Q(v)=\Bsp(c_0(v^\flat)\chi_0,\mathscr D_0(c_0(v^\flat)\chi_0)).
\]
For the variation used in Section~\ref{sec:cyclic-l3},
\[
 v=\cos y^0\,E_0+(1+\sin y^0)E_1,
\]
$Q(v)=0$, while
\[
 \operatorname{div}_0v=-\sin y^0,
 \qquad
 |v|_0^2=2+2\sin y^0.
\]
Consequently
\[
 W(t)=w\int_Y
 \frac{-t\operatorname{div}_0v+t^2Q(v)}{1+t^2|v|_0^2}\,\vol_0,
\]
and expansion gives
\[
 [t^3]W(t)=-2\pi wV_6,
 \qquad
 [t^5]W(t)=8\pi wV_6\neq0.
\]
This is the coefficient used in \eqref{eq:quintic-example}.

\section{Courant coordinates for the gauge hierarchy}
\label{app:gauge-hierarchy}\label{app:gauge}\label{app:gauge-descent}

Section~\ref{sec:auxiliary-symmetry} gives the symmetry differential
used in the first-order presentation. We record only its relation to the
standard Courant--Green--Schwarz coordinates and the global frame ghost.

On a contractible principal-bundle chart, let $\gamma$ and $\beta_C$ be
the degree-one and degree-two coordinates of the Courant Lie-$2$
hierarchy of \cite[Appendix~D.1--D.2]{BrongersPaperI}. With independent
frame ghost $r$, set
\begin{equation}\label{eq:ii-parent-courant-dictionary}
\begin{aligned}
 u_T&=r,&u_G&=r^{[0]}+\lambda,\\
 \eta_C&=-\Lambda-\tfrac\epsilon2P(A,\lambda),&
 \beta_C&=-c+\tfrac\epsilon4P(r^{[0]},\lambda)
                -\tfrac12\iota_\xi\eta_C,\\
 \gamma&=\xi+u_T+u_G+\eta_C.
\end{aligned}
\end{equation}
For a tangent connection $\Theta$, the relative two-form coordinate is
\[
 B_a=B-\frac\epsilon4P(\Theta^{[0]},A-\Theta^{[0]}).
\]
This triangular change of variables converts the Courant differential of
\cite[(D.2), (D.10)--(D.11)]{BrongersPaperI} into the formulas of
Section~\ref{sec:auxiliary-symmetry}. In particular,
\begin{equation}\label{eq:b-lambda}
 b_\lambda=\tfrac12P(\lambda,F_A)-\tfrac14P(a,D_A\lambda),
 \qquad a=A-\Theta^{[0]}.
\end{equation}
The constraint and multiplier transformations are then the induced
representation and its dual. Inspection of the differential in Section~\ref{sec:auxiliary-symmetry}
and of \eqref{eq:global-frame-ghost} below gives the cubic bound used
in Section~\ref{sec:cyclic-l3}.

For globalization, the local frame ghost transforms by
\[
 r_b=g^{-1}r_ag-g^{-1}\xi(g).
\]
A fixed background connection therefore gives the global adjoint
coordinate
\[
 r_{\rm cov}=r+\iota_\xi\Theta_0,
\]
with
\begin{equation}\label{eq:global-frame-ghost}
\begin{aligned}
 fr_{\rm cov}
   &=r_{\rm cov}^2-\nabla^0_\xi r_{\rm cov}
     +\tfrac12\xi^i\xi^j(R_0)_{ij},\\
 f\lambda
   &=\lambda^2-\nabla^0_\xi\lambda
     +[\lambda,r_{\rm cov}^{[0]}].
\end{aligned}
\end{equation}
Thus the local formulas glue without choosing a global parallel frame.

\section{Transverse lifts, the homological differential and Hamiltonian identities}
\label{app:holomorphic-q}\label{app:holomorphic-endpoint}

\subsection{The complete Courant operations}

Before imposing the tangent graph, write \(y=(x,\rho,\mu)\) and
\(y'=(x',\rho',\mu')\) for homogeneous bundle-valued \((0,p)\)- and \((0,q)\)-forms. All products below have the displayed wedge order. Let \(c\) denote the signed anomaly pairing on the tangent and gauge adjoint factors. The operations are
\begin{equation}\label{eq:hol-ambient-operations}
 \begin{aligned}
 (\bar\partial_Qy)_a&=\bar\partial x_a-c(F_a,\rho),&
 (\bar\partial_Qy)_\rho&=\bar\partial_{C_0}\rho+F_a\mu^a,&
 (\bar\partial_Qy)^a&=\bar\partial\mu^a,\\
 [y,y']_a&=C_a(y,y')+\tfrac12c(\nabla_a\rho,\rho')
                         -\tfrac12c(\rho,\nabla_a\rho'),\\
 [y,y']_\rho&=-[\rho,\rho']_{\mathrm{gr}}
          +\mu^a\nabla_a\rho'-(\nabla_a\rho)\mu'^a,\\
 [y,y']^a&=\mu^d\partial_d\mu'^a-(\partial_d\mu^a)\mu'^d,\\
 \langle y,y'\rangle&=\mu^ax'_a+x_a\mu'^a+c(\rho,\rho'),
 \end{aligned}
\end{equation}
where the geometric cotangent expression is
\begin{equation}\label{eq:hol-geometric-bracket}
 \begin{split}
 C_a(y,y')={}&\mu^d\partial_dx'_a-(\partial_dx_a)\mu'^d
 -\tfrac12\mu^d\partial_ax'_d+\tfrac12(\partial_ax_d)\mu'^d\\
 &+\tfrac12(\partial_a\mu^d)x'_d-\tfrac12x_d\partial_a\mu'^d.
 \end{split}
\end{equation}
In the gauge component, geometric coefficients multiplying a gauge
variable are reduced modulo $\epsilon$. The tangent component uses
the full geometric coefficients before restriction to the graph. These are Dolbeault operations on the string extension.
Their cochain grading is distinct from the two-term grading of the
Courant symmetry algebra.

On the graph, write the background curvatures as
\(F_{A_0}=dz^a\wedge F_{A,a}\) and
\(R_{\Gamma_0}=dz^a\wedge R_a\). The unary operator becomes
\begin{equation}\label{eq:hol-induced-unary}
 \begin{aligned}
 (\bar\partial_Qy)_a={}&\bar\partial x_a
 -\tfrac\epsilon4\operatorname{tr}_V(F_{A,a}\alpha)
 +\tfrac\epsilon4\operatorname{tr}_T(R_aK\mu^{[0]}),\\
 (\bar\partial_Qy)_\alpha={}&\bar\partial_{A_0}\alpha
                                 +F_{A,a}(\mu^a)^{[0]},\qquad
 (\bar\partial_Qy)^a=\bar\partial\mu^a.
 \end{aligned}
\end{equation}

\subsection{The exterior expansion of the obstruction}

With $\theta^a=dz^a+\mu^a$, set
\[
 B_0=2\theta^a\wedge x_a+b-\mu^a\wedge x_a,
 \qquad d_\mu=\bar\partial-\mu^a\wedge\partial_a,
\]
a direct expansion of $dB_0+T_c(C_0,\rho)$ in powers of the moving
ideal $\mathcal I=(\theta^a)$ gives the residuals used in the main text. The coefficient linear in $\theta^a$ is
$-2R_{x,a}$, where
\begin{equation}\label{eq:hol-Rx-components}
 \begin{split}
 R_{x,a}={}&\bar\partial x_a-c(F_a,\rho)
 -\mu^b\wedge\partial_bx_a+\tfrac12\mu^b\wedge\partial_ax_b
 -\tfrac12\partial_a\mu^b\wedge x_b\\
 &+\tfrac12c(\rho,\nabla_a\rho)-\tfrac12\partial_ab,
 \end{split}
\end{equation}
and the constant coefficient, after subtracting
$\langle y,R_y\rangle$, is
\begin{equation}\label{eq:hol-Rb-components}
 \begin{split}
 R_b={}&\bar\partial b-\tfrac12\mu^a\wedge\partial_ab
 -\tfrac12\mu^a\wedge\mu^b\wedge\partial_ax_b
 +\tfrac12\mu^a\wedge\partial_a\mu^b\wedge x_b\\
 &-\tfrac12\mu^a\wedge c(\rho,\nabla_a\rho)
 -\tfrac16c(\rho,[\rho,\rho]).
 \end{split}
\end{equation}
Thus, for $G=dB_0+T_c(C_0,\rho)$,
\[
 G=-2\theta^a\wedge R_{x,a}+R_b+\langle y,R_y\rangle
 \pmod{\mathcal I^2}.
\]
The residuals above are the $x$- and scalar components of
\eqref{eq:hol-mc-residuals}. The remaining two components are the
integrability tensor of $T^{0,1}_\mu X$ and the curvature of the induced
partial connection. Consequently, integrability, partial flatness and vanishing of the
string curvature modulo $\mathcal I^2$ are together equivalent to
$R_y=R_b=0$.

\subsection{The homological differential}

Let \(R\) be an augmented nilpotent dg-Artin algebra over \(\mathbb C\),
and put \(R_n=R\otimes\Omega_{\mathrm{poly}}(\Delta^n)\). Tensor the fixed
complex mixed modules with the nilpotent external ideal over
\(\mathbb C\). Write \(D=d_X+d_{R_n}\). The total ideal generators are
\(\theta^a=dz^a+\mu^a\), now total-odd. The total volume is
\(\boldsymbol\Omega=\mathbf u\theta^1\theta^2\theta^3\). The tangent-connection graph gives \(\mathcal N=\Gamma_0+K(\mu^{[0]})\). The intrinsic
equations are
\begin{equation}\label{eq:hol-intrinsic-total}
 \begin{gathered}
 D\mathcal I\subset\mathcal I,\qquad D\boldsymbol\Omega=0,\qquad
 F_D(A)=0\pmod{\mathcal I^{[0]}},\\
 D\mathcal B+\tfrac\epsilon4
       \{\operatorname{CS}_D(A)-\operatorname{CS}_D(\mathcal N)\}
                                      =0\pmod{\mathcal I^2}.
 \end{gathered}
\end{equation}
Before graph restriction the full tangent partial-flatness equation is
also imposed. Here \(\operatorname{CS}_D\) uses the corresponding matrix
trace, with \(\operatorname{tr}_T\) in the tangent contribution.

Fix the background two-form presentation \(B_{\mathrm{bg}}\). The total
string coordinate is
\begin{equation}\label{eq:hol-intrinsic-twoform}
\begin{split}
 \mathcal B={}&B_{\mathrm{bg}}+B_0
 +\tfrac\epsilon4\{\operatorname{tr}_V(A_0\alpha)
                         -\operatorname{tr}_T(\Gamma_0\kappa)\},\\
 B_0={}&2\theta^ax_a+\mathbf b-\mu^ax_a.
 \end{split}
\end{equation}
The reference trace terms follow from
\(\operatorname{CS}_D(C_0+\rho)-\operatorname{CS}_D(C_0)
=T_D(C_0,\rho)-D\operatorname{tr}(C_0\rho)\).

To determine $Q_N$, use the completed free graded algebra of field
coordinates and their jets, with $D=d_X+Q_N$, without yet assuming
$Q_N^2=0$. The part of $D\theta$ independent of the generators of
$\mathcal I$ determines $Q_N\mu$. Partial flatness and volume closure
determine $Q_N\rho$ and $Q_N\mathbf u$. The constant and linear
$\mathcal I$-coefficients of the string equation determine
$Q_N\mathbf b$ and $Q_Nx$. In the total graded fields
\eqref{eq:hol-total-packages}, this gives \eqref{eq:hol-q} and

\begin{equation}\label{eq:hol-universal-identities}
 \begin{aligned}
 D\theta^i&=\theta^a\partial_a\mu^i,\\
 D_{C_0}\rho&=\theta^a\nabla_a\rho-F_a\mu^a-\rho^2,\\
 T_D(C_0,\rho)&=\theta^a\{2c(F_a,\rho)-c(\rho,\nabla_a\rho)\}
                -\mu^ac(F_a,\rho)-\tfrac13c(\rho,\rho^2),\\
 DB_0+T_D(C_0,\rho)&=2\theta^a\theta^b\partial_ax_b.
 \end{aligned}
\end{equation}
The cubic in the third line is
\(-c(\rho,[\rho,\rho])/6\). Its sign includes the crossing of the
odd \(\rho\) and the odd coframe. Thus the coefficient calculation also
fixes the all-degree scalar sign.

We prove \(Q_N^2=0\) before restricting to any solution locus. The
tensor-product sign gives \(D^2=Q_N^2\). Applying \(D\) to the first identity
of \eqref{eq:hol-universal-identities} yields
\begin{equation}\label{eq:hol-Q2-anchor}
 D^2\theta^i=
 \theta^b(\partial_b\mu^a)(\partial_a\mu^i)
 -\theta^a(\partial_a\mu^b)(\partial_b\mu^i)
 -\theta^a\theta^b\partial_a\partial_b\mu^i=0.
\end{equation}
The first two terms cancel after exchanging \(a,b\), and the last term
vanishes by antisymmetry. Hence \(Q_N^2\mu=0\).

For an arbitrary odd derivation the curved Bianchi identity is
\begin{equation}\label{eq:hol-curved-bianchi}
 D F_D(C)+[C,F_D(C)]=D^2C.
\end{equation}
Its use does not presuppose \(D^2=0\). Ideal stability and the universal
partial-curvature equation put the left-hand side in \(\mathcal I\), or in
\(\mathcal I^{[0]}\) for gauge coefficients. The right-hand side
\(Q_N^2\rho\) consists only of reference antiholomorphic forms with
external coefficients. Transversality makes the intersection of this
space with the moving ideal zero. Thus \(Q_N^2\rho=0\). After imposing
the tangent graph, its tangent component also follows from
\(Q_N^2\mu=0\).

Apply \(D^2\) to \(\mathbf u\theta^1\theta^2\theta^3\). Since $D^2\theta^a=0$, its top coefficient is
\(Q_N^2\mathbf u\), which vanishes. The connection components now satisfy $Q_N^2\rho=0$,
so the usual Chern--Simons derivative identity is valid. Differentiate
the last equation of \eqref{eq:hol-intrinsic-total}. The curvature lies
in \(\mathcal I\), and therefore its Pontryagin form lies in \(\mathcal I^2\). It follows
that \(Q_N^2\mathcal B\in \mathcal I^2\). The already-vanishing \(\mu,\rho\)
components and the even parity of \(Q_N^2\) give
\begin{equation}\label{eq:hol-Q2-twoform}
 Q_N^2\mathcal B=2\theta^aQ_N^2x_a
                  +Q_N^2\mathbf b-\mu^aQ_N^2x_a.
\end{equation}
The right-hand side has only linear and constant ideal degree. Their
independent coefficients force \(Q_N^2x=Q_N^2\mathbf b=0\).
The argument applies to all coefficients and jets, including ghosts and
upper cochain degrees. In the mixed category
\(\epsilon(\mathcal I^{[0]})^2\) maps into \(\mathcal I^2\), so the proof needs no division
by \(\epsilon\).

\subsection{The Taylor brackets}

For homogeneous \(Z=(y,b,s)\), \(Z'=(y',b',s')\),
\(Z''=(y'',b'',s'')\) of degrees \(p,q,r\), the brackets in the
plus-sign Maurer--Cartan convention are
\begin{equation}\label{eq:hol-all-brackets}
 \begin{aligned}
 \ell_1Z={}&-\bigl(
 \bar\partial_Qy+\tfrac12(-1)^p\partial b,
 \bar\partial b,\bar\partial s+(-1)^p\operatorname{div}_{\Omega_0}\mu\bigr),\\
 \ell_2(Z,Z')={}&\bigl([y,y'],
 \tfrac12\{\langle y,\partial b'\rangle
                   -(-1)^{pq}\langle y',\partial b\rangle\},
 \rho_\Omega(\mu)s'
             -(-1)^{pq}\rho_\Omega(\mu')s\bigr),\\
 \ell_3(Z,Z',Z'')={}&\tfrac13(-1)^{p+q+r}
 \bigl(0,\langle y,[y',y'']\rangle
       +(-1)^{p(q+r)}\langle y',[y'',y]\rangle\\
 &\hspace{44mm}+(-1)^{r(p+q)}\langle y'',[y,y']\rangle,0\bigr),\\
 \ell_n={}&0\qquad(n\geq4).
 \end{aligned}
\end{equation}
The operations in \eqref{eq:hol-ambient-operations} are pulled back along
the tangent graph. The gerbe and canonical-line complexes have internal degrees
$-1$ and $+1$, respectively.

The signs in \eqref{eq:hol-all-brackets} are those obtained from the suspension convention of \cite[Appendix~C]{BrongersPaperI}, with horizontal forms ordered before coefficient ghosts. In particular, the unary operator acquires the displayed minus sign, while the binary suspension and coefficient-crossing signs cancel.

For fields of cochain degree one,
\(\ell_1Z+\ell_2(Z,Z)/2+\ell_3(Z,Z,Z)/6\) is minus the three
residuals in \eqref{eq:hol-q}. The free-algebra square-zero proof supplies
every polarized identity, including arity four. At arity five the
composition also vanishes directly: \(\ell_3\) takes values in the gerbe component and vanishes
when an argument lies in that component. 

\subsection{The variational Hamiltonian identities}
\label{app:hol-hamiltonian-identities}

Set $c_{\mathrm{an}}=\epsilon/4$ and
\[
 \vartheta_N=\int_X\mathbf u\,\lambda_N\wedge\Omega_0,
 \qquad
 \lambda_N=-\delta\mathbf b+\mu^a\delta x_a+x_a\delta\mu^a
 +c_{\mathrm{an}}\operatorname{tr}_V(\alpha\delta\alpha)
 -c_{\mathrm{an}}\operatorname{tr}_T(\kappa\delta\kappa).
\]
With $Q_N=-q$, substitution of \eqref{eq:hol-q} gives
\[
 \iota_{Q_N}\vartheta_N
 =\int_X\mathbf u\{q\mathbf b+\langle Y,qY\rangle\}\Omega_0
 =I_N.
\]
The coefficient $-1/3$ of the cubic term in $I_N$ is the difference
$1/6-1/2$ between the scalar equation and the contraction of the
quadratic $Y$-term.

It remains to check invariance of the primitive. The Dolbeault terms
integrate to zero on the closed manifold. For the anchor terms one uses
\[
 (\mu^a\partial_a)^\dagger
 =-\mu^a\partial_a-\operatorname{div}_{\Omega_0}\mu,
\]
so the contribution from the volume coefficient is exactly cancelled by
$q\mathbf u$. In the gauge and tangent traces, invariance of the matrix
pairing and integration by parts combine the quadratic and cubic terms;
the background-curvature contributions cancel against the curvature
terms in the cotangent component of $qY$. The tangent calculation is the
same with the opposite trace sign and then restricts along the strict
differential graph. Thus
\[
 \mathcal L_{Q_N}\vartheta_N=0.
\]
Cartan's identity therefore gives
\[
 \iota_{Q_N}\omega_N=-\delta I_N,
 \qquad
 \delta\omega_N=0,
 \qquad
 \mathcal L_{Q_N}\omega_N=0.
\]

Finally, the contraction map of $\omega_N$ is triangular in the order
$\delta\mathbf b,\delta x,\delta\alpha,\delta\mu,\delta\mathbf u$.
Its diagonal entries are multiplication by the invertible coefficient
$\mathbf u$, together with the nondegenerate matrix trace and the fixed
module identification $\mathbb D_0\simeq\epsilon\mathbb D$ on the gauge
block. The differential graph contributes only finite-order off-diagonal
terms. Hence the contraction map has a finite local inverse near
$\mathbf u=1$.

\section{Gerbe descent and intrinsic identification}
\label{app:tangent-density-gerbe}

\subsection{The scalar Cartan coefficient}

We compare the scalar gerbe component with the string-Courant Lie-\(2\) convention
of \cite{BrongersPaperI}. Write \(r\) for the tangent-frame ghost,
\(t=r^{[0]}\) for its reduction, \(\lambda\) for the relative gauge
ghost and \(k=t+\lambda\) for the absolute gauge ghost. Let \(\xi\)
be the vector ghost, \(\eta\) the canonical covector ghost,
\(\beta\) the canonical scalar, and \(c\) the relative scalar. The
invariant polynomial is \(P=-\operatorname{tr}_7\).

Before restricting the tangent connection to its graph, the parameter
convention on holomorphic sections is
\begin{equation}\label{eq:hol-cartan-orientation}
 \mu_{(0)}=-\xi^{1,0},\quad
 \alpha_{(0)}=-k,\quad
 \kappa_{e,(0)}=-r,\quad
 \kappa_{(0)}=-t,\quad
 2x_{(0)}=-\eta^{1,0},\quad
 b_{(-1)}=-\beta.
\end{equation}
The subscript denotes the relevant cochain component.

The absolute scalar is
\(c_a=c-\epsilon P(t,\lambda)/4\), and
\(\beta=-c_a-\iota_\xi\eta/2\). Therefore
\begin{equation}\label{eq:hol-cartan-scalar-shift}
 b_{(-1)}=c-\tfrac\epsilon4P(t,\lambda)
                           +\tfrac12\iota_\xi\eta.
\end{equation}
The sign reversal already appears in the quadratic rules:
\(Q_N\alpha=-\alpha^2\), whereas the physical BRST derivation \(f\)
has \(fk=k^2\). At zero vector ghost, the pure-adjoint cubic in
\eqref{eq:hol-q} is
\[
 Q_Nb_{(-1)}
 =\frac{\epsilon}{24}
   \{\operatorname{tr}_V(\alpha[\alpha,\alpha])
                  -\operatorname{tr}_T(\kappa[\kappa,\kappa])\}.
\]
Applying \eqref{eq:hol-cartan-orientation} and \(P=-\operatorname{tr}\)
gives
\begin{equation}\label{eq:hol-cartan-coefficient}
 Q_Nb_{(-1)}
   =\frac{\epsilon}{24}\{P(k,[k,k])-P(t,[t,t])\}
   =fc_a\qquad(\xi=0).
\end{equation}
The tangent trace is the rank-seven trace restricted to the Hermitian
subalgebra. The standard Cartan identity for
\(P(a,[a,a])\), together with \eqref{eq:hol-cartan-scalar-shift},
recovers the mixed common-frame terms and the relative coefficient
\(\epsilon/24\). The remaining vector-ghost terms follow from the
transport part of the same Courant differential. Thus the scalar part of
\eqref{eq:hol-all-brackets} agrees with the relative gerbe hierarchy;
locally constant degree-\(-1\) scalars are central.

\subsection{Local intrinsic identification}

Fix the background bundle and relative gerbe presentation, including
its overlap data and topological sector. All deformations reduce to
that background under the external augmentation, and relative arrows
reduce to the identity. Nilpotent corrections to overlap data are
allowed. The two-form is the relative class modulo \(\mathcal I^2\) used in
\eqref{eq:hol-intrinsic-total}.

Let \(R\) be an augmented nilpotent dg-Artin algebra over
\(\mathbb C\), and set
\(R_n=R\otimes\Omega_{\mathrm{poly}}(\Delta^n)\).
First complexify each fixed real mixed module. Then tensor it with
\(\mathfrak m_R\otimes\Omega_{\mathrm{poly}}(\Delta^n)\) over
\(\mathbb C\), with differential \(D=d_X+d_{R_n}\). The geometric
and gauge coefficient types are unchanged. Transversality means
that the reference antiholomorphic form algebra maps isomorphically
onto the quotient by \(\mathcal I\). It gives unique normalized generators
\(\theta^a=dz^a+\mu^a\). These generators are total-odd, so their
exterior change of variables is invertible for graded coefficients.

Let \(\mathcal L_{\mathrm{hol}}\) be the graph-restricted graded
sheaf with operations \eqref{eq:hol-all-brackets}. For a chart
\(U\subset X\), define \(\mathfrak F_U(R)_n\) to consist of the
intrinsic total data satisfying \eqref{eq:hol-intrinsic-total},
with invertible volume coefficient and the chosen relative string
presentation.
Different full extensions of a partial connection are identified
using their transgression compensator.

\begin{proposition}\label{prop:hol-local-descent}
The coordinate formula \eqref{eq:hol-intrinsic-twoform} and normalized
coefficient extraction give a natural simplicial isomorphism
\[
 \operatorname{MC}_\bullet
 \bigl(\Gamma(U,\mathcal L_{\mathrm{hol}})
                     \otimes_{\mathbb C}\mathfrak m_R\bigr)
       \cong\mathfrak F_U(R).
\]
The isomorphism includes all external differentials and simplex degrees.
\end{proposition}
\begin{proof}
Ideal stability implies \(D(\mathcal I^2)\subset \mathcal I^2\), and partial flatness
puts curvature in \(\mathcal I\). Its Pontryagin form therefore lies in \(\mathcal I^2\).
Suppose a full extension changes from \(C'\) to \(C'+s\), with
\(s\in\mathcal I\). The transgression identity gives
\[
 \operatorname{CS}_D(C'+s)-\operatorname{CS}_D(C')
      =T_D(C',s)-D\operatorname{tr}(C's),\qquad
 T_D(C',s)\in \mathcal I^2.
\]
The induced two-form change is \(+c_{\mathrm{an}}
\operatorname{tr}(C's)\) in the gauge trace and its
negative in the tangent trace. Thus the string equation is
independent of the full extension, with its given one-form and
scalar parameters. Equivalently, in the combined signed notation,
\[
 B_0\longmapsto B_0+c(\rho,s)\pmod{\mathcal I^2},\qquad
 \mathcal B\longmapsto\mathcal B+c(C',s)\pmod{\mathcal I^2}.
\]
Only form directions have been extended. The gauge extension
remains \((\mathbb D_0)_{\mathbb C}\)-valued.

The forward formula is verified by
\eqref{eq:hol-universal-identities}. Its last identity is the
relative string equation after adding the reference trace
compensator. The inverse first normalizes the ideal, then extracts
\(\mathbf u\) from its top product and the partial connection from
its action on $T^{0,1}_\mu X$. The tangent connection is
computed from \(K\). Remove the reference compensator from $\mathcal B$ and write
\[
 B_0=2\theta^a\wedge x_a+B_c\pmod{\mathcal I^2}.
\]
The linear ideal coefficient, divided by two, gives $x$, and
\[
 \mathbf b=B_c+\mu^a\wedge x_a.
\]
These recover the two-form coordinates uniquely. The construction
commutes with dg-Artin maps, external differentials and polynomial
simplex pullbacks, so the degreewise bijections assemble to the stated
simplicial isomorphism.
\end{proof}

\subsection{Global descent in the fixed formal problem}

Choose a cover \(\{U_i\}\) trivializing the fixed background bundle
and gerbe data. The relative transgression cochains are a two-form \(j\), a one-form \(a_R\) and a scalar \(t_R\), of composition degrees one, two and three, respectively. Their total identities, with \(\delta_{\mathrm C}\) the
Čech differential and \(\Delta T\) the overlap difference of
transgressions, are
\begin{equation}\label{eq:hol-descent-cochains}
 Dj=\Delta T,\qquad Da_R=\delta_{\mathrm C}j,\qquad
 Dt_R=\delta_{\mathrm C}a_R,\qquad
 \delta_{\mathrm C}t_R=0.
\end{equation}
These are the Stokes identities for the total transgression on polynomial simplices and identify changes of presentation on both sides of Proposition~\ref{prop:hol-local-descent}.

\begin{proposition}\label{prop:hol-global-descent}\label{prop:product-global-formal-gerbe}
For the fixed background and gerbe presentation there are natural weak equivalences
\[
 \begin{split}
 \operatorname{MC}_\bullet
 \bigl(\Gamma(X,\mathcal L_{\mathrm{hol}})
                          \otimes_{\mathbb C}\mathfrak m_R\bigr)
 &\simeq
 \underset{[p]\in\Delta}{\operatorname{holim}}
 \prod_{i_0,\ldots,i_p}
 \operatorname{MC}_\bullet
 \bigl(\Gamma(U_{i_0\cdots i_p},\mathcal L_{\mathrm{hol}})
                          \otimes_{\mathbb C}\mathfrak m_R\bigr)\\
 &\simeq
 \underset{[p]\in\Delta}{\operatorname{holim}}
 \prod_{i_0,\ldots,i_p}\mathfrak F_{U_{i_0\cdots i_p}}(R).
 \end{split}
\]
\end{proposition}
\begin{proof}
The underlying graded sheaves are fine, so their augmented Čech
resolution is exact degreewise. Filter the Maurer--Cartan functors by
powers of the nilpotent external ideal. Each successive quotient is
abelian, and Proposition~\ref{prop:hol-local-descent} identifies the
corresponding Čech totalizations. Induction over the nilpotence length
gives the first weak equivalence, while
\eqref{eq:hol-descent-cochains} gives the second.
\end{proof}

\section{Unary comparison with the independent-tangent complex}
\label{app:unary-comparison}

We record the independent-tangent Dolbeault differential used in
Proposition~\ref{prop:unary-holomorphic-identification}. The background is
the fixed Calabi--Yau standard embedding, with constant dilaton and zero
flux.

Write $W=T^{1,0}X$ and $\mathfrak h=\mathfrak{gl}(W)$. The Hermitian
tangent algebra is embedded in
$\mathfrak g_T=\mathfrak{so}(7,\mathbb C)$ by
\[
 j(M)=\operatorname{diag}(M,-M^{\mathsf T},0),
\]
and the standard embedding gives a parallel map
$\iota:\mathfrak g_T\to\mathfrak g_G$. Hence
$F_{A,a}=\iota(jR_a)$ and
$\operatorname{tr}_T(jM\,jN)=2\operatorname{tr}_3(MN)$.

Before imposing the tangent graph, take degree-$p$ variables
\[
 (x,\kappa_e,\alpha,\mu,b),
\]
with the first four entries of Dolbeault degree $p$ and
$b\in\Omega^{0,p+1}(X)$. The tangent entry is
$\mathbb D_{\mathbb C}$-valued and the gauge entry is
$(\mathbb D_0)_{\mathbb C}$-valued. On
$y=(x,\kappa_e,\alpha,\mu)$ the background operator is
\begin{equation}\label{eq:unary-independent-operator}
\begin{aligned}
 (D_Qy)_{x,a}&=\bar\partial x_a
 -\tfrac14\jmath_\epsilon\operatorname{tr}_V(F_{A,a}\wedge\alpha)
 +\tfrac14\jmath_\epsilon
             \operatorname{tr}_T(jR_a\wedge\kappa_e^{[0]}),\\
 (D_Qy)_{\kappa_e}&=
          \bar\partial_{\Gamma_0}\kappa_e+jR_a\wedge\mu^a,\\
 (D_Qy)_\alpha&=\bar\partial_{A_0}\alpha
                                  +F_{A,a}\wedge(\mu^a)^{[0]},\\
 (D_Qy)_\mu&=\bar\partial\mu,
\end{aligned}
\end{equation}
and
\[
 d_F^p(y,b)=\bigl(D_Qy+\tfrac12(-1)^p\partial b,\bar\partial b\bigr).
\]
Adjoining the canonical-line entry
$s\in\Omega^{0,p-1}(X)$ adds the component
\[
 \bar\partial s+(-1)^p\operatorname{div}_{\Omega_0}\mu.
\]
The coefficient-restored operator is the unary operator of
\cite[Appendix~A]{FiniteDeformations}; the normalization used there is
$x_{\rm printed}=4x/z$, $b_{\rm printed}=4b/z$ at nonzero anomaly
coefficient $z$.

The induced tangent connection is the differential graph
\[
 \kappa_e=jK_{\Gamma_0}(\mu).
\]
It is preserved because
\begin{equation}\label{eq:unary-graph-app}
 \bar\partial_{\Gamma_0}\bigl(jK_{\Gamma_0}(\mu)\bigr)
 +jR_a\wedge\mu^a
 =jK_{\Gamma_0}(\bar\partial\mu),
\end{equation}
which is the unary identity of
Lemma~\ref{lem:induced-action-compatibility}. Restricting
\eqref{eq:unary-independent-operator} to this graph, adjoining the
canonical-line component, and multiplying degree $p$ by $(-1)^p$ gives
exactly the chain isomorphism of
Proposition~\ref{prop:unary-holomorphic-identification}.

\section{The representative-change calculation}
\label{app:representative-change}\label{app:product-coordinates}
This appendix proves the two technical facts used in
Section~\ref{sec:representatives}: the Hull--Hermitian functional defect
factors through the leading geometric residuals, and the corresponding
Euler map is pointwise invertible. Put $c_{\rm an}=\epsilon/4$ and write
$C$ for the Ehresmann--Libermann connection denoted by $\Gamma$ in the
main text.

\subsection{Residual factorization}
\label{app:product-maps}
Set
\begin{equation}\label{eq:product-connection-difference-type}
 t=\Theta_H^{[0]}-C^{[0]}=(\Theta_H-C)^{[0]}
 \in\Omega^1(\mathfrak{so}(TY))\otimes\mathbb D_0.
\end{equation}
The connection and two-form changes are
\eqref{eq:representative-T0}--\eqref{eq:representative-defect}.
In this appendix $D_C$ denotes the functional $\mathcal D_\Gamma$
with $C=\Gamma$.
All tensors in the first-order coefficient below are reduced modulo
$\epsilon$. We write $\operatorname{Tr}_P(A_1\cdots A_k)
=-\operatorname{tr}_7(A_1\cdots A_k)$, with exterior factors in the
displayed order.

Choose a type coframe $\theta^i$ and put
\[
 \mathfrak n^i=(\dd\theta^i)^{0,2},\qquad
 \mathcal K=\dd B+i\,\dd\omega.
\]
Write $n_{\bar k\bar l,\bar j}
=g_{i\bar j}\mathfrak n^i{}_{\bar k\bar l}$.
In the decomposition $T^{1,0}\oplus T^{0,1}$ write
\[
 t^{0,1}=\mathcal L+\mathcal E,
 \qquad
 \mathcal L=\begin{pmatrix}0&0\\L&0\end{pmatrix},
 \qquad
 \mathcal E=\begin{pmatrix}U&V\\0&W\end{pmatrix}.
\]
The contorsion formula gives
\begin{equation}\label{eq:product-hull-blocks}
 \begin{aligned}
 U^i{}_j&=\tfrac12g^{i\bar l}\mathcal K_{j\bar k\bar l}\bar\theta^{\bar k},&
 W^{\bar i}{}_{\bar j}
   &=-\tfrac12g^{l\bar i}\mathcal K_{l\bar k\bar j}\bar\theta^{\bar k},\\
 L^{\bar i}{}_j&=-\tfrac12g^{l\bar i}\mathcal K_{jl\bar k}\bar\theta^{\bar k},&
 V^i{}_{\bar j}
   &=g^{i\bar l}\left(\tfrac12\mathcal K_{\bar k\bar l\bar j}
        +n_{\bar k\bar l,\bar j}-n_{\bar k\bar j,\bar l}\right)
                                         \bar\theta^{\bar k}.
 \end{aligned}
\end{equation}
Let $\mathscr F_C=(R_{C^{[0]}})^{0,2}$ and
$Z=\sum_i t_i\mathfrak n^i$. Type preservation makes $\mathscr F_C$
block diagonal, and $(D_Ct)^{0,2}=\bar D_Ct^{0,1}+Z$. Expanding the
relative transgression yields
\begin{equation}\label{eq:product-residual-factor}
 \begin{aligned}
 \mathcal T_{3,P}(C,t)^{0,3}
 ={}&2P(\mathcal E,\mathscr F_C)
 +P(\mathcal E,\bar D_C\mathcal L)
 +P(\mathcal L,\bar D_C\mathcal E)
 +P(\mathcal E,\bar D_C\mathcal E)\\
 &+P(\mathcal L+\mathcal E,Z)
 +2\operatorname{Tr}_P(\mathcal L\mathcal E^2)
 +\tfrac23\operatorname{Tr}_P(\mathcal E^3).
 \end{aligned}
\end{equation}
Here $\mathcal L^2=0$,
$P(\mathcal L,\mathscr F_C)=P(\mathcal L,\bar D_C\mathcal L)=0$.
Every surviving term contains one of the residuals
$(\dd\Omega)^{2,2}$, $(\dd\Omega)^{3,1}$,
$\mathcal K^{1,2}$, $\mathcal K^{0,3}$, or a first covariant derivative
of one of them. Thus the density of $D_C$ vanishes when these residual
jets vanish.

\subsection{The Euler map and finite-jet correction}
Let $\mathcal Q_{\rm geo}$ be the real bundle of fields
$(\omega,\Psi,\Phi,B)$ satisfying the algebraic $SU(3)$ relations and set
\begin{equation}\label{eq:product-residual-bundle}
 \mathcal R=
 \big(\Lambda^{2,2}\oplus\Lambda^{3,1}
      \oplus\Lambda^{1,2}\oplus\Lambda^{0,3}\big)_{\mathbb R},
 \qquad
 F=((\dd\Omega)^{2,2},(\dd\Omega)^{3,1},
                         \mathcal K^{1,2},\mathcal K^{0,3}).
\end{equation}
The leading functional satisfies
\begin{equation}\label{eq:product-real-euler}
 \delta W_0=\frac L4\operatorname{Im}\int_X
 \{-\delta(B+i\omega)\wedge\dd\Omega
                   +\dd(B+i\omega)\wedge\delta\Omega\}
 =:E_0(\delta q),
\end{equation}
which defines a pointwise map
$\mathsf A_{\rm EL}:\mathcal R\to
T^*\mathcal Q_{\rm geo}\otimes\operatorname{Dens}_X$ by
$E_0=\mathsf A_{\rm EL}F$.

In the coordinates $(x,b,\mu,u)$ with $\Omega=u\Omega_\mu$, its diagonal
pairings are
\[
 \Lambda^{1,1}\times\Lambda^{2,2},\qquad
 \Lambda^{0,2}\times\Lambda^{3,1},\qquad
 \Lambda^{2,1}\times\Lambda^{1,2},\qquad
 \Lambda^{3,0}\times\Lambda^{0,3},
\]
with unit factors $2,u$ and the nonzero volume. The remaining terms are
lower triangular in this order. Passing to the moving $(2,0)$ quotient
only adds the shear
\[
 (\dd C_{20})^{0,3}=0,
 \qquad
 (\dd C_{20})^{1,2}=C_{ij}\theta^j\wedge\mathfrak n^i.
\]
Hence $\mathsf A_{\rm EL}$ is a pointwise real bundle isomorphism.

Let $F_I=\nabla_I F$ denote the finitely many residual jets occurring in
the density $\operatorname{Dens}_C$ of \eqref{eq:representative-defect}.
Use the explicit factorization \eqref{eq:product-residual-factor} to
regard this density as a polynomial in formal slots $F_I$, with its
other tensor factors fixed. Define
\begin{equation}\label{eq:product-euler-homotopy}
 \mathcal H_C=\sum_I(\nabla_I)^\dagger
 \int_0^1
 \frac{\partial\operatorname{Dens}_C}{\partial F_I}
       (z;sF_J)\,\dd s,
 \qquad
 \mathcal V_C=\mathsf A_{\rm EL}^{-\mathsf T}\mathcal H_C.
\end{equation}
Since the defect density vanishes when all residual jets vanish, the
fundamental theorem of calculus along \(sF_I\), followed by integration
by parts, gives
\begin{equation}\label{eq:residual-factorization}
 D_C=\int_X\langle\mathcal H_C,F\rangle
     =\int_XE_0(\mathcal V_C).
\end{equation}
Thus \(\mathcal V_C\) is local of finite jet order. At the Calabi--Yau reference the factors in
\eqref{eq:product-residual-factor} vanish as sections; after one
derivative with respect to a residual variable a vanishing factor remains. Consequently
$\mathcal V_C$ vanishes at the reference. The construction is tensorial and
depends on $B^{[0]}$ only through $\dd B^{[0]}$, so it commutes with the
order-zero diffeomorphism, frame and gerbe differential. These are the
properties used in Proposition~\ref{thm:representative-equivalence}.

\section{Coordinate maps and tangent-current transformations}
\label{app:canonical-maps}
We give the maps defined in Section~\ref{sec:hamiltonian-reduction},
their inverses, and their action on the full variational primitive.
All conjugations are those of the real formal comparison.
Put $c_{\rm an}=\epsilon/4$, interpreted as $\jmath_\epsilon/4$
on reduced arguments.

The construction below is the coordinate realization of the map
\(\mathscr T\) in Theorem~\ref{thm:hamiltonian-realization}. It is the
composite of the Hull--Hermitian representative change, common-frame
and gerbe maps, cotangent polarization, and the two first-order
Hamiltonian corrections described in Section~\ref{sec:hamiltonian-reduction}.

\subsection{Common frames and field coordinates}
\label{app:product-fields}

We express the tangent and gauge connections in a common frame, so that
their transgressions use the same trace convention. The tangent-frame
action is induced by the geometry, while the gauge action remains the
full orthogonal action. Both connections and their ghosts are transformed
together.
Let \(F_0\) be the fixed complex null frame, \(F_g\) the moving type
frame whose conjugate block is made metric-dual, and \(I_g\) the
positive isometry from the reference metric to \(g\).
Let $\mathsf Q$ be the matrix of the bilinear form in the null frame. Then
\begin{equation}\label{eq:product-common-frame}
 F=F_0^{-1}I_g^{-1}F_g,\qquad F^{\mathsf T}\mathsf QF=\mathsf Q,\qquad F|_0=1 .
\end{equation}
Extend by the fixed circle line. The determinant is one near the
reference, so \(F\) belongs to the same \(\SO(7,\C)\) bundle.
Geometric $\mathbb{D}$-valued connections transform by
$\Theta'=F^{-1}\Theta F+F^{-1}\dd F$.
For the gauge connection use the reduced frame $F_c=\red F$:
\[
 A'=F_c^{-1}AF_c+F_c^{-1}\dd F_c\quad\text{in }\mathbb{D}_0.
\]
In particular, the preliminary pullback to the symmetric metric frame contains
\((I_gF_0)^{-1}\dd(I_gF_0)\).

To display the ghost signs, write \(\xi_p\) for the physical
diffeomorphism ghost of the earlier sections and
\(\widehat\xi=-\xi_p\) for the Cartan-dressed ghost.
Let $j_g(\xi_p)$ be the infinitesimal compensating rotation of the
symmetric metric frame, as defined in \cite{BrongersPaperI}. Put
\begin{equation}\label{eq:product-frame-ghosts}
 \begin{gathered}
 j_{\rm pos}=\iota_{\xi_p}\Gamma_0-j_g(\xi_p),\\
 h_F=F^{-1}(f-\mathcal L_{\widehat\xi})F,\qquad
 j'=F^{-1}j_{\rm pos}F+h_F,\\
 \lambda'=F_c^{-1}\lambda F_c,\qquad g=\red{j'}-\lambda',\\
 \mathcal B'=B-c_{\rm an}P(\red{\Theta'},a')+i\omega,\qquad
 \Gamma'_{\rm gerbe}=-\Lambda-2c_{\rm an}P(\lambda',A'),\\
 c'_G=-c+c_{\rm an}P(\lambda',\red{j'}),\qquad a'=A'-\red{\Theta'}.
 \end{gathered}
\end{equation}
The primed entries are in the transformed common frame, while
\(B,\Lambda,c\) are the physical relative gerbe coordinates.
After this display we suppress primes in the gauge and transgression
formulas: $j$ denotes $\red{j'}$, and their tangent connection symbols
denote the reduced geometric connections. The full $\mathbb{D}$-valued geometric
frame action is the one above, before reduction. In particular the
reference $\Theta$ in Section~\ref{app:product-gerbe} and the connection
$\Gamma_C$ in Section~\ref{app:product-current} have type $\mathbb{D}_0$.

Write
\(\Omega=u\prod_a(\dd z^a+\mu^a)\) in local volume coordinates.
Let $\mathcal B_{\rm bg}$ be the fixed background two-form in this
complex chart. In the following splitting compensator, $\alpha=A-A_0$
and $\kappa=\Theta-\Gamma_0$ denote the full connection displacements.
With the joint connection compensators included in \(\mathcal B\), put
\[
 \mathcal B_{\rm spl}=\mathcal B-\mathcal B_{\rm bg}
 -c_{\rm an}\{\tr_V(A_0\alpha)-\tr_T(\Gamma_0\kappa)\}.
\]
This is the reference splitting of \eqref{eq:hol-intrinsic-twoform}.
The field map is
\begin{equation}\label{eq:product-field-chart}
 \mathcal B^\flat=\mathcal B_{\rm spl}
                -e^{\iota_\mu}(\mathcal B_{\rm spl})_0^{2,0},\qquad
 2x=(\mathcal B^\flat)^{1,1}_0,\qquad
 b=(\mathcal B^\flat)^{0,2}_0-\mu^ax_a.
\end{equation}
For a connection in this frame put $\Delta A=A-A_0$. After projection,
use $\alpha$ for its moving partial component:
\[
 \alpha=(\Delta A)^{0,1}_0-\iota_\mu(\Delta A)^{1,0}_0,\qquad
 a_+=e^{\iota_\mu}(\Delta A)^{1,0}_0,\qquad A^\flat=A_0+\alpha.
\]
The additional two-form shift is
\(-c_{\rm an}\tr_V(A^\flat a_+)\). The tangent term has the opposite
sign. These shifts follow from the same transgression triangle as
\eqref{eq:representative-triangle}.

For the inverse construction, let $V$ be the full class in the moving
two-form quotient, after restoring the background and connection
compensators subtracted above. The real two-form and Hermitian form are
\[
 B=\operatorname{Re}V^{1,1}+V^{0,2}+\overline{V^{0,2}},
 \qquad
 \omega=\operatorname{Im}V^{1,1},
\]
so that $\mathcal B_{\rm phys}=B+i\omega$.
The overline here is complex conjugation. The identity $\lVert\Omega\rVert^2=8e^{-4\Phi}$ recovers $\Phi$, and
\eqref{eq:product-forms} recovers the metric and normalized spinor
bilinears. The $\mathbb{D}_0$-valued gauge connection is recovered from its partial
connection and a compatible Hermitian reduction $r_H$.
The reconstruction equation \eqref{eq:product-compatible-connection}
is solved using the moving type projections.
All remaining \(\eps\)-shifts have the inverse obtained by reversing
their sign on the unchanged data modulo $\eps$.

\subsection{Gerbe representative maps and contractible factors}
\label{app:product-gerbe}

The Hermitian reduction and real gerbe one-form specify a real
representative of the holomorphic data. Their transformations use the relative
transgression cochains of the fixed gerbe presentation.
Let $j,a_R,t_R$ be the relative transgression cochains constructed in
Appendix~\ref{app:tangent-density-gerbe}. They satisfy
\[
 \dd j=\Delta T,\qquad \dd a_R=\delta j,\qquad
 \dd t_R=\delta a_R,\qquad \delta t_R=0,
\]
and are natural under the near-identity changes of presentation used below.

For these finite formulas, \(A\) is the complexified connection modulo $\eps$,
and \(\mathcal B_R=\beta+i\omega\) is the free complex two-form in the fixed
relative gerbe chart. A near-identity complex orthogonal gauge
transformation \(k\), together with a free complex one-form
\(\Lambda_R\), acts by
\[
 A^k=k^{-1}Ak+k^{-1}\dd k,\qquad
 \mathcal B'_R=\mathcal B_R+c_{\rm an}j(A;k)+\dd\Lambda_R.
\]
The positive Hermitian reduction modulo $\eps$ satisfies
\(r_H^\dagger=r_H\), \(r_H^{\mathsf T}=r_H^{-1}\), and transforms as
\[
 r_H^k=k^{-1}r_H(k^{-1})^\dagger .
\]
The dagger in these finite metric formulas is Hermitian conjugation.
The added metric \(m\) is a free real one-form, and \(\tau\) is its
free real scalar gauge parameter. Put
\[
 p=\sqrt{r_H},\qquad p'=\sqrt{r_H^k},\qquad k_K=p^{-1}kp'.
\]
Then \(k_K\) is compact, and the polar arrow has endpoints
\(\widehat A=A^p\) and \(\widehat A'=(A^k)^{p'}=\widehat A^{k_K}\).
The finite polar and gerbe maps are
\begin{equation}\label{eq:product-polar-gerbe}
 \begin{gathered}
 \widehat A=A^p,\qquad
 z_R=\Lambda_R+c_{\rm an}
       \{a_R(A;k,p')-a_R(A;p,k_K)\},\\
 \widehat{\mathcal B}_R=\mathcal B_R+c_{\rm an}j(A;p)-\tfrac i2\dd m,\qquad
 m'=m+2\operatorname{Im}z_R+\dd\tau .
 \end{gathered}
\end{equation}
The descent identities give the corresponding composition and scalar reducibility laws.

Let $\sigma_H$ be the logarithmic polar coordinate, $r_B$ the real
scalar metric-arrow coordinate, $\beta_\sigma$ the polar ghost,
$\widehat u_B$ the shifted imaginary one-form ghost and $c_I$ the
imaginary scalar ghost. In these adapted coordinates the added base
pairs have
\begin{equation}\label{eq:product-unit-pairs}
 \begin{array}{lll}
 f\sigma_H=\beta_\sigma,&f\beta_\sigma=0,&\mathsf h\beta_\sigma=\sigma_H,\\
 fm=2\widehat u_B,&f\widehat u_B=0,&\mathsf h\widehat u_B=m/2,\\
 fr_B=-2c_I,&fc_I=0,&\mathsf h c_I=-r_B/2.
 \end{array}
\end{equation}
Including the signed cotangent partners, these are elementary contractible Hamiltonian factors of Appendix~\ref{app:parent-elimination}.

\subsection{Ghost dressing and the cotangent chart}
\label{app:product-momenta}

Ghost dressing first reorganizes the total forms by a finite contraction
with the diffeomorphism ghost. Cotangent polarization is a separate step:
the variational one-form then identifies which transformed components
are momenta, including all dependence on the moving frame.
Use bold letters for totalized forms with coefficient ghost numbers.
Write $\mu_k$ for the horizontal $k$-form component of the total
Beltrami variable, so $\boldsymbol z^a=dz^a+\sum_{k=0}^3\mu_k^a$.

\paragraph{Ghost dressing.}
There is a unique real ghost \(\widehat\xi\), near the reference, satisfying
\[
 [e^{-\iota_{\widehat\xi}}\boldsymbol z^a]_{[0]}
 =\mu_0^a-\widehat\xi^a-\iota_{\widehat\xi}\mu_1^a
  +\tfrac12\iota_{\widehat\xi}^2\mu_2^a
  -\tfrac16\iota_{\widehat\xi}^3\mu_3^a=0 .
\]
Together with its conjugate, this equation has the invertible real
coframe matrix as its linearization. Applying that matrix inverse and
substituting recursively solves it in the formal completion.
Each nonlinear correction raises formal order. Set
\(\widehat\Omega=e^{-\iota_{\widehat\xi}}\boldsymbol\Omega\).
It has no components of horizontal degree below three.
For the gauge trace put
\[
 \widetilde{\mathcal B}=\mathcal B_{\rm spl}+c_{\rm an}\tr_V(A_0\alpha),\qquad
 e^{-\iota_{\widehat\xi}}(A_0+\alpha)=\mathsf A+P_2+P_3,\quad
 \mathsf A=g+a_{\rm part}.
\]
Here \(g\) has horizontal form degree zero and \(a_{\rm part}\) horizontal form degree one.
In particular,
\begin{equation}\label{eq:product-complete-ghost}
 g=\alpha_{[0]}-\iota_{\widehat\xi}(A_0+\alpha_{[1]})
       +\tfrac12\iota_{\widehat\xi}^2\alpha_{[2]}
       -\tfrac16\iota_{\widehat\xi}^3\alpha_{[3]}.
\end{equation}
Define
\[
 P_{\geq2}=P_2+P_3,\qquad
 V=e^{-\iota_{\widehat\xi}}\widetilde{\mathcal B}-c_{\rm an}\tr_V(\mathsf A P_{\geq2}).
\]
This last shift has the inverse obtained by addition.

\paragraph{Cotangent polarization.}
Let \(J_\delta=\iota_{\delta\widehat\xi}\). The dressed primitive is
\begin{equation}\label{eq:product-dressed-primitive}
 \begin{split}
 \vartheta=\frac L4\operatorname{Im}\int_X\big[
 &-\delta V+c_{\rm an}\tr_V(\mathsf A\delta\mathsf A
                                      +2P_{\geq2}\delta\mathsf A)-J_\delta V\\
 &+c_{\rm an}\tr_V(\mathsf A J_\delta\mathsf A
             +2P_{\geq2}J_\delta\mathsf A+P_{\geq2}J_\delta P_{\geq2})\big]\widehat\Omega .
 \end{split}
\end{equation}
Take the component of horizontal degree six. Replace
\(-\delta V_3\widehat\Omega_3\) by
\(+V_3\delta\widehat\Omega_3\) modulo an exact vertical differential.
The resulting primitive is semibasic over
\[
 (\widehat\Omega_3,V_2;\widehat\xi,V_1;V_0;a_{\rm part},g).
\]
The ordered coefficients of \eqref{eq:product-dressed-primitive} define the
new momenta. To compare them with the undressed chart, put
\[
 \vartheta_{\rm und}=\frac L4\operatorname{Im}\int_X
 [-\delta\widetilde{\mathcal B}
       +c_{\rm an}\tr_V(\mathbf A\delta\mathbf A)]\boldsymbol\Omega,
 \qquad \mathbf A=A_0+\alpha .
\]
Modulo the moving ideal square,
\[
 \mathcal B_{\rm spl}=2\boldsymbol z^ax_a+\mathbf b-\mu^ax_a,
 \qquad
 (\delta\mathcal B_{\rm spl})\boldsymbol\Omega
 = (\delta\mathbf b-\mu^a\delta x_a-x_a\delta\mu^a)\boldsymbol\Omega.
\]
The reference trace in $-\delta\widetilde{\mathcal B}$ cancels the
corresponding term in $c_{\rm an}\tr_V(\mathbf A\delta\mathbf A)$,
so the holomorphic top component is the potential \(\vartheta_N\) of Appendix~\ref{app:hol-hamiltonian-identities}. The momentum transformation is triangular:
the $V_0,V_1,V_2$ coefficients recover
$\widehat\Omega_6,\widehat\Omega_5,\widehat\Omega_4$; the
$\widehat\Omega_3$ coefficient gives $V_3$; the $a_{\rm part}$ and $g$
coefficients give $P_2$ and $P_3$ successively; and the $\widehat\xi$ coefficient
then determines $V_4$ through the real coframe inverse. Thus the
polarization map and its inverse are finite and local. Dependence of the
moving component bases only adds lower-triangular terms already contained
in \eqref{eq:product-dressed-primitive}.

To compute the Hamiltonian in these coordinates, let \(D_c=\dd-\mathcal L_{\widehat\xi}\) and \(b_c=[\widehat\xi,\widehat\xi]\). Its
positive trace contribution and reduction modulo $\eps$ are obtained from
\begin{equation}\label{eq:product-dressed-hamiltonian}
 \begin{split}
 \frac L4\operatorname{Im}\int_X\big\{
 &D_cV+c_{\rm an}\tr_V(\mathsf A D_c\mathsf A+\tfrac23\mathsf A^3)
       +2c_{\rm an}\tr_V(P_{\geq2}(D_c\mathsf A+\mathsf A^2))\\
 &-\tfrac12[\iota_{b_c}V+
 c_{\rm an}\tr_V(\mathsf A\iota_{b_c}\mathsf A
       +2P_{\geq2}\iota_{b_c}\mathsf A+P_{\geq2}\iota_{b_c}P_{\geq2})]
 -c_{\rm an}e^{-\iota_{\widehat\xi}}\CS(A_0)\big\}\widehat\Omega .
 \end{split}
\end{equation}
This follows from
\(e^{-\iota_{\widehat\xi}}\dd e^{\iota_{\widehat\xi}}
 =D_c-\iota_{b_c}/2\).
Terms with two \(P_{\geq2}\)'s vanish by horizontal form degree except the displayed
insertion term. Formula \eqref{eq:product-dressed-hamiltonian}
verifies the functional and all momentum coefficients together.
The tangent trace is treated with its opposite sign and its reference
subtraction in the same frame.

To recover the compatible $\mathbb{D}_0$-valued gauge connection, require
\begin{equation}\label{eq:product-compatible-connection}
 A=a_{\rm part}+\psi,\qquad
 \psi=\theta^ib_i,\qquad
 \dd r_H+Ar_H+r_HA^\dagger=0 .
\end{equation}
The partial value of \(A\) is already prescribed. Its conjugate value
is fixed by the last equation. Changing between the conjugate type
frames uses \(1-\mu\overline\mu\) and its conjugate. These are
pointwise units. This determines \(b_i\) from the first derivatives
of \(r_H\). The inverse takes the partial connection of a compatible
pair \((A,r_H)\).

The accompanying gerbe map is
\begin{equation}\label{eq:product-compatible-gerbe}
 \mathcal B=V_2+c_{\rm an}\tr_V(a_{\rm part}\psi),\qquad
 \Gamma_{\rm gerbe}=-V_1-c_{\rm an}\tr_V(ga_{\rm part}),\qquad c_G=-V_0.
\end{equation}
Its functional identity follows from
\(\CS(A)-\CS(a_{\rm part})
 =T(a_{\rm part},\psi)-\dd\tr(a_{\rm part}\psi)\).
The transgression times \(\widehat\Omega_3\) vanishes because every
term has a factor in the deformed holomorphic ideal. The full connection law and the
metric law preserve \eqref{eq:product-compatible-connection}.
Uniqueness therefore proves off-shell equivariance. Apply the
variational cotangent lift to \eqref{eq:product-compatible-connection}
and \eqref{eq:product-compatible-gerbe}, including the derivatives of
\(r_H\). Together with \eqref{eq:product-dressed-hamiltonian}, this
identifies the Hamiltonian with density and gerbe variables.

The connection, gerbe and momentum maps give the chart in which the
remaining functional difference is the tangent current. The following
calculation removes that difference on the representative coisotropic.

\subsection{The total tangent current}
\label{app:product-currents}

\label{app:product-current}
Let
\[
 C^{[0]}=\{\varpi^{[0]}=0\}.
\]
After the Moser correction, write
\(\Phi^*S_C=S_0+\jmath_\epsilon S_1'\) and
\(I_N=I_0+\jmath_\epsilon I_1\). We prove
\[
 \left.(S_1'-I_1)\right|_{C^{[0]}}=Q_0\mathcal G
\]
by constructing a relative tangent-transgression primitive. Here
\(\widehat N\) is the transformed tangent graph and \(\Gamma_C\) the
Hermitian connection in the common frame.

\begin{lemma}[The source total differential]\label{lem:product-source-total}
Let $Q_0$ be the order-zero source Hamiltonian derivation after the maps
of Section~\ref{app:product-momenta}, restricted to
$C^{[0]}=\{\varpi^{[0]}=0\}$. Put
\[
 \widehat{\boldsymbol z}=e^{-\iota_{\widehat\xi}}\boldsymbol z,
 \qquad \widehat{\mathcal I}=(\widehat{\boldsymbol z}^{\,1},
                  \widehat{\boldsymbol z}^{\,2},
                  \widehat{\boldsymbol z}^{\,3}),\qquad
 \widehat D=\dd+Q_0-\mathcal L_{\widehat\xi}.
\]
Then, in the total field, ghost and equation algebra,
\begin{equation}\label{eq:product-source-total-identities}
 Q_0\widehat\xi=\tfrac12[\widehat\xi,\widehat\xi],\qquad
 \widehat D^2=0,\qquad
 \widehat D\widehat\Omega=0,\qquad
 \widehat D\widehat{\mathcal I}\subseteq\widehat{\mathcal I}.
\end{equation}
\end{lemma}

\begin{proof}
With forms placed before coefficients, $Q_0$ anticommutes with $\dd$.
The ghost equation $Q_0\widehat\xi=[\widehat\xi,\widehat\xi]/2$ and the Cartan identities
therefore give $\widehat D^2=0$. The Hamilton equation obtained by varying
the dressed two-form in \eqref{eq:product-dressed-primitive}, at leading
coefficient, is $\widehat D\widehat\Omega=0$. Finally
$\widehat\Omega=\widehat u\,\widehat{\boldsymbol z}^{\,1}
\widehat{\boldsymbol z}^{\,2}\widehat{\boldsymbol z}^{\,3}$ with
$\widehat u$ invertible. Since the annihilator of multiplication by
$\widehat\Omega$ is the ideal $\widehat{\mathcal I}$, applying $\widehat D$ to
$\widehat{\boldsymbol z}^{\,a}\widehat\Omega=0$ yields
$\widehat D\widehat{\boldsymbol z}^{\,a}\in\widehat{\mathcal I}$.
\end{proof}

Thus the source Hamiltonian gives the nilpotent total differential,
volume closure and ideal stability before the following connection and
current comparison. Write the horizontal two-form component of
$\widehat{\boldsymbol z}$ as $\nu_2$, and put \(\widehat Q=Q_0-\mathcal L_{\widehat\xi}\). Let \(j\) be the
full geometric tangent-frame ghost. The moving type projections are
those of \(\widehat\Omega_3\).
\begin{equation}\label{eq:product-current-frame}
 \begin{gathered}
 u_j=P^{1,0}\partial_j,\qquad B_j=P^{0,1}\partial_j,\qquad
 T^a{}_j=\iota_{B_j}\nu_2^a,\qquad g_T=e^T,\\
 E_C=\Gamma_C-\widehat N_1-[T,j]+\widehat QT,\qquad
 b_i=\iota_{u_i}E_C,\qquad a_I=\widehat{\boldsymbol z}^{\,i}b_i,\\
 \mathcal C'=g_T(\widehat N+a_I)g_T^{-1}
                       -(\widehat Dg_T)g_T^{-1}.
 \end{gathered}
\end{equation}

\begin{lemma}[Connection matching in low form degree]
\label{lem:product-current-low}
The total connection \(\mathcal C'\) of
\eqref{eq:product-current-frame} satisfies
\((\mathcal C')_{[0]}=j\) and \((\mathcal C')_{[1]}=\Gamma_C\)
with every ghost and momentum component included. The construction
uses only the moving coframe inverse and finite differentiation.
\end{lemma}

\begin{proof}
Write $\theta^a=\widehat{\boldsymbol z}^{\,a}_{[1]}$ and
$\widehat\mu_1^a=\theta^a-dz^a$. In a local coordinate splitting, put
\[
 j^a{}_b=\theta^a(\partial_b\widehat\xi),\qquad
 n^a{}_b=\partial_b\widehat\mu_1^a,\qquad
 (R_{\widehat\xi}\nu_2)^a{}_b
 =\iota_{\partial_b\widehat\xi}\nu_2^a.
\]
The dressed graph connection has components
\[
 \widehat N_0=j,\qquad
 \widehat N_1=n+R_{\widehat\xi}\nu_2+\Gamma_{0,a}\theta^a.
\]
Set $M=\bar\partial\widehat\mu_1-\tfrac12[\widehat\mu_1,\widehat\mu_1]$.
The degree-two ideal equation of Lemma~\ref{lem:product-source-total}
gives, for $\widehat Q=Q_0-\mathcal L_{\widehat\xi}$,
\[
 \widehat Q\nu_2^a=-M^a
 +\theta^b\iota_{\partial_b\widehat\xi}\nu_2^a
 +\nu_2^b j^a{}_b.
\]
Define
\[
 \Delta^a{}_j=-\iota_{B_j}M^a,\qquad
 (\Lambda_b)^a{}_j
 =\iota_{B_j}\iota_{\partial_b\widehat\xi}\nu_2^a.
\]
Covariance of the type projections and the graded contraction rule yield
\[
 \widehat QT=-\Delta+R_{\widehat\xi}\nu_2+[T,j]+\theta^b\Lambda_b.
\]
Substitution in \eqref{eq:product-current-frame} therefore gives
\[
 E_C=\Gamma_C-n-\Delta-\Gamma_{0,a}\theta^a+\theta^b\Lambda_b.
\]
The vanishing mixed torsion of the Ehresmann--Libermann connection
identifies its moving antiholomorphic part with that of $n+\Delta$.
The other terms have a factor $\theta^a$, so $E_C=\theta^ib_i$.
Expanding the transformed connection through horizontal degree one now
shows
\[
 (\mathcal C')_{[0]}=j,\qquad
 (\mathcal C')_{[1]}
 =\widehat N_1+E_C+[T,j]-\widehat QT=\Gamma_C.
\]
Higher terms of the exponential have horizontal degree at least two.
\end{proof}

Put \(\mathcal C=\Gamma_C+j\) and \(a_C=\mathcal C'-\mathcal C\).
The reconstructed metric is invariant
under the gerbe action with metric representatives and is equivariant
under real diffeomorphisms. Hence
\(\delta_c\Gamma_C=D_{\Gamma_C}j\), \(\delta_cj=-j^2\), where
$\delta_c$ is the coefficient derivation of $\widehat Q$ in the
forms-first convention. Consequently
\[
 F_{\widehat D}(\mathcal C)=F_{\Gamma_C}.
\]

\begin{lemma}[Six-dimensional transgression vanishing]
\label{lem:product-current-degree}
On the coisotropic subspace modulo $\eps$ \(C^{[0]}\),
\[
 T_{3,\widehat D}(\mathcal C,a_C)\widehat\Omega=0,
 \qquad T_{3,\widehat D}(\widehat N,a_I)\widehat\Omega=0.
\]
\end{lemma}

\begin{proof}
Lemma~\ref{lem:product-current-low} gives horizontal form degree at least
two for \(a_C\), and \(\widehat D\) preserves or raises that horizontal form degree.
After multiplication by $\widehat\Omega$, the three terms of the
first transgression have horizontal degrees at least $7,7,9$,
respectively, and hence vanish on $X$.
The second transgression has a factor
in the \(\widehat D\)-stable ideal
\((\widehat{\boldsymbol z}^{\,1},\widehat{\boldsymbol z}^{\,2},
\widehat{\boldsymbol z}^{\,3})\). Each generator annihilates
\(\widehat\Omega\), so this transgression also vanishes. 
\end{proof}

For a formula compatible with the fixed global anomaly data set
\[
 b_C=\mathcal C-\widehat N,\qquad
 C_s=(\widehat N+a_I)^{e^{-sT}},\qquad \ell_s=-T.
\]
The primitive and its normalized generator are
\begin{equation}\label{eq:product-current-generator}
 \begin{aligned}
 \tau&=\int_0^1\left[
 2P(\ell_s,F_{\widehat D}(C_s))
       -P(C_s-\widehat N,D_{C_s}\ell_s)\right]\dd s,\\
 \mathcal G&=\frac L{16}\operatorname{Im}\int_X
                       [\tau+P(b_C,a_C)]\wedge\widehat\Omega .
 \end{aligned}
\end{equation}
\begin{lemma}[Relative current primitive]\label{lem:product-current-primitive}
With the relative \(B\) shift included, the unweighted tangent-current
difference is the \(Q_0\) differential of
\(\int_X[\tau+P(b_C,a_C)]\widehat\Omega\). Consequently its real
normalized coefficient is \(Q_0\mathcal G\).
\end{lemma}

\begin{proof}
For the convention \(C^h=h^{-1}Ch+h^{-1}\widehat Dh\), the path has
\(C_{s=0}=\widehat N+a_I\), \(C_{s=1}=\mathcal C'\), and
\(\partial_sC_s=D_{C_s}\ell_s\) with \(\ell_s=-T\).
The reference \(\widehat N\) is independent of \(s\), but need not
be \(Q_0\)-closed. Fixed-reference variation and the Bianchi identity give
\[
 \partial_sT_{3,\widehat D}(\widehat N,C_s-\widehat N)
 =\widehat D\{2P(\ell_s,F_{\widehat D}(C_s))
              -P(C_s-\widehat N,D_{C_s}\ell_s)\}.
\]
Integrating, using Lemma~\ref{lem:product-current-degree} at the first
endpoint, and then applying the relative triangle at the second yields
\begin{align*}
 T_{3,\widehat D}(\widehat N,b_C+a_C)\widehat\Omega
     &=\widehat D(\tau\widehat\Omega),\\
 [T_{3,\widehat D}(\widehat N,b_C+a_C)
       -T_{3,\widehat D}(\widehat N,b_C)]\widehat\Omega
     &=-\widehat D\{P(b_C,a_C)\widehat\Omega\}.
\end{align*}
Thus
\begin{equation}\label{eq:product-current-relative-identity}
 T_{3,\widehat D}(\widehat N,b_C)\widehat\Omega
 =\widehat D\{[\tau+P(b_C,a_C)]\widehat\Omega\}.
\end{equation}
The exterior derivative on $X$ and the Lie derivative of a top form integrate to zero, leaving \(Q_0\). The relative compensator makes the primitive compatible with the fixed gerbe presentation.
\end{proof}

The cotangent polarization of Section~\ref{app:product-momenta}
puts the independent gauge and tangent traces in the same symplectic
chart. Varying the relative transgression in that chart gives the
source-minus-target primitive
\begin{equation}\label{eq:product-relative-beta}
 \beta=\frac L{16}\operatorname{Im}\int_X
 P\bigl(b_C,(\delta+\iota_{\delta\widehat\xi})
          (\mathcal C+\widehat N)\bigr)\widehat\Omega,
 \qquad b_C=\mathcal C-\widehat N.
\end{equation}
This is the only part of the full momentum transformation needed below;
the omitted coefficients are obtained from the same variational
cotangent lift and are triangular in horizontal degree.

\begin{lemma}[Relative first-order symplectic coordinates]
\label{lem:product-relative-moser}
Let $\vartheta_{\mathcal R}$ be the cotangent potential for the
representative variables, so $\delta\vartheta_{\mathcal R}=\omega_{\mathcal R}$.
In the common mixed chart write
\[
 \vartheta_C=\vartheta_*+\jmath_\epsilon\beta_C,\qquad
 \vartheta_N+\vartheta_{\mathcal R}
 =\vartheta_*+\jmath_\epsilon\beta_N,\qquad
 \beta=\beta_C-\beta_N.
\]
Let \(\omega_0\) be the nondegenerate leading symplectic form on the
free block and define \(V\) by
\[
 \iota_V\omega_0=-\beta.
\]
Then
\[
 \Phi=\id+\jmath_\epsilon V,\qquad
 \Phi^{-1}=\id-\jmath_\epsilon V
\]
are local inverse maps and
\begin{equation}\label{eq:product-current-darboux} \Phi^*\omega_C=\omega_N+\omega_{\mathcal R},
 \qquad VS_0=-\iota_{Q_0}\beta.
\end{equation}
\end{lemma}
\begin{proof}
This is the first-order Moser calculation:
$\mathcal L_V\omega_0=-\delta\beta$, while
$\iota_{Q_0}\omega_0=-\delta S_0$ gives
$VS_0=-\iota_{Q_0}\beta$. Since the correction is
$\epsilon$-linear, $\id\pm\jmath_\epsilon V$ are inverse.
\end{proof}

We apply the lemma to \eqref{eq:product-relative-beta}.
Contracting with \(Q_0\) and using
\(Q_0\widehat\xi=[\widehat\xi,\widehat\xi]/2\) removes the insertion
term in the relative differential. Lemma~\ref{lem:product-current-primitive}
then gives, on the leading representative locus,
\begin{equation}\label{eq:product-current-coboundary-app}
 \left.(\Phi^*S_C-I_N)\right|_{C^{[0]}}
 =\jmath_\epsilon\,Q_0\mathcal G,
\end{equation}
where \(\mathcal G\) is the local tangent-current primitive of
\eqref{eq:product-current-generator}.

Extend \(\mathcal G\) independently of the representative momenta and
set
\[
 \Psi=\id+\epsilon X_{\mathcal G}.
\]
The first-order canonical transformation formula gives
\begin{equation}\label{eq:product-current-pullback}
 \widehat S
 =\Psi^*\Phi^*S_C
 =S_0+\epsilon(S_1'-Q_0\mathcal G).
\end{equation}
By \eqref{eq:product-current-coboundary-app},
\[
 \widehat S|_C=I_N,\qquad C=\{\varpi=0\}.
\]
The restricted Hamiltonian is independent of the representative
coordinate \(z\); hence its Hamiltonian vector field preserves \(C\)
and descends to the characteristic quotient. This is the final
canonical correction in Theorem~\ref{thm:hamiltonian-realization}.

\section{Normal equations on the product restriction}
\label{app:coisotropic}\label{app:product}

The first characteristic restriction fixes the product locus and removes
the momenta conjugate to its normal coordinates. This appendix records
the corresponding Euler--Lagrange equations.

\subsection{The invariant functional before product restriction}

Set $w=e^{-2\Phi}$ and decompose
\[
 H=h_3+\eta\wedge h_2,
 \qquad
 \mathcal T_3=\mathcal T_h+\eta\wedge\mathcal T_v.
\]
With
\[
 h_3=\dd b_2-\frac\epsilon4\mathcal T_h,
 \qquad
 h_2=-\dd b_1-\frac\epsilon4\mathcal T_v,
\]
the circle-invariant functional is
\begin{equation}\label{eq:product-invariant-functional}
 W_{\rm inv}
 =\frac L4\int_Xw\left[
 R(h_3-v\wedge h_2)\wedge\widehat\rho_X
 +R\,\dd\omega\wedge\rho_X
 +\frac12(h_2-R^2\dd v)\wedge\omega^2
 \right].
\end{equation}
Indeed, with
\[
 \rho_{\rm raw}=\rho_X+Rv\wedge\omega,
 \qquad
 \psi_4=\frac12\omega^2-Rv\wedge\widehat\rho_X,
 \qquad
 \psi_3=-R\widehat\rho_X,
\]
one has
\(
 \varphi=\rho_{\rm raw}+\eta\wedge R\omega
\)
and
\(
 \psi=\psi_4+\eta\wedge\psi_3
\).
Substitution in the seven-dimensional functional gives
\eqref{eq:product-invariant-functional} after using
$\omega\wedge\rho_X=0$.

For an invariant $k$-form $q=q_h+\eta\wedge q_v$, the cotangent term
reduces as
\begin{equation}\label{eq:product-momentum-sign}
 \int_Yp\wedge\delta q
 =L\int_X\left[
 p_v\wedge\delta q_h+(-1)^{7-k}p_h\wedge\delta q_v
 \right].
\end{equation}
This fixes the normalization of the normal momenta used by the
Hamiltonian product restriction.

\subsection{Normal variations}
\label{app:product-normal}

At the product locus vary the excluded coordinates by
\[
 r=\delta R,
 \qquad
 \nu=\delta v,
 \qquad
 \beta_1=\delta b_1,
 \qquad
 \tau=\delta a_0.
\]
Keeping the remaining product variables fixed gives
\[
 \delta g=2r\eta^2+\eta\otimes\nu+\nu\otimes\eta,
 \qquad
 \delta B=\eta\wedge\beta_1,
 \qquad
 \delta a=\eta\tau.
\]
Let $\Theta=(\Theta^u)^{[0]}$ and put
\[
 Z=(D\Theta^u)(\delta g,\delta B)^{[0]}.
\]
The relative transgression varies by
\begin{equation}\label{eq:normal-transgression-variation}
 \delta\mathcal T_3
 =2P(Z+\delta a,F_A)-2P(Z,R_\Theta)
   +\dd\{2P(Z,a)-P(a,\delta a)\},
\end{equation}
and hence
\[
 \delta H
 =\dd\delta B-\frac14\jmath_\epsilon(\delta\mathcal T_3)
 =\delta h_3+\eta\wedge\delta h_2.
\]
Differentiating \eqref{eq:product-invariant-functional} gives
\begin{equation}\label{eq:product-normal-variation}
 \left.\delta_nW\right|_{\rm product}
 =\frac L4\int_Xw\left[
 r\{H\wedge\widehat\rho+\dd\omega\wedge\rho\}
 +\delta h_3\wedge\widehat\rho
 +\frac12\delta h_2\wedge\omega^2
 -\frac12\dd\nu\wedge\omega^2
 \right].
\end{equation}
Substituting \eqref{eq:normal-transgression-variation} and integrating by
parts gives the Euler--Lagrange covector in each normal direction.

At leading order, the $b_1$ and $v$ variations are
\begin{equation}\label{eq:product-normal-balance}
 \left.\delta_{b_1}W\right|_{\rm product}
 =-\frac L8\int_X\delta b_1\wedge\dd(w\omega^2),
 \qquad
 \left.\delta_vW\right|_{\rm product}
 =-\frac L8\int_X\delta v\wedge\dd(w\omega^2).
\end{equation}
Thus both normal directions recover conformal balance. For the vertical
gauge variation $\delta a=\eta\tau$ one finds
\begin{equation}\label{eq:product-normal-gauge}
 \left.\delta_{a_0}W\right|_{\rm product}
 =\jmath_\epsilon\left(
 -\frac L{16}\int_Xw^{[0]}P(\tau,F_A)\wedge(\omega^2)^{[0]}
 +\frac L{32}\int_XP(a_1,\tau)\wedge\dd(w\omega^2)^{[0]}
 \right).
\end{equation}
On the balance locus, nondegeneracy of $P$ and injectivity of
$\jmath_\epsilon$ identify this equation with gauge primitivity.

The radius variation is controlled by
\begin{equation}\label{eq:normal-radius-identity}
 H\wedge\widehat\rho+\dd\omega\wedge\rho
 =T_h\wedge\widehat\rho-\Sigma_v.
\end{equation}
On a physical solution, Corollary~\ref{cor:leading-cy} makes the leading
geometry Calabi--Yau and the leading tangent and gauge curvatures
$SU(3)$-instantons. The algebraic curvature/transgression pairings with
$\widehat\rho^{[0]}$ and $(\omega^2)^{[0]}$ in the normal variation
therefore vanish by type and primitivity. The exact transgression terms
vanish after integration by parts because
\[
 \dd\Omega^{[0]}=0,
 \qquad
 \dd(w\omega^2)^{[0]}=0.
\]
Together with balance and
\eqref{eq:normal-radius-identity}, this makes every normal
Euler--Lagrange component vanish. Thus the product restriction is
compatible with the physical zero locus, although the normal equations
are not imposed by the characteristic quotient itself.

\section{Product BPS identities and representative dependence}
\label{app:product-bps}

This appendix collects the pointwise $SU(3)$ identities used in
Proposition~\ref{prop:product-bps-dictionary} and proves that the
remaining physical conditions are not constant along the fibres of the
final characteristic quotient.

\subsection{Product Hodge identities}

For a horizontal $k$-form $\alpha$,
\begin{equation}\label{eq:product-hodge}
 *_7\alpha=(-1)^k\eta\wedge *_6\alpha,
 \qquad
 *_7(\eta\wedge\alpha)=*_6\alpha.
\end{equation}
Since $\dd\eta=0$ and the fields are circle-invariant,
\[
 \dd(\alpha_h+\eta\wedge\alpha_v)
 =\dd_X\alpha_h-\eta\wedge\dd_X\alpha_v.
\]
With the conventions of Section~\ref{sec:product-reduction}, this gives
\[
 \dd\varphi=\dd\rho-\eta\wedge\dd\omega,
 \qquad
 \dd\psi=\omega\wedge\dd\omega+\eta\wedge\dd\widehat\rho.
\]
These formulas yield the horizontal and vertical residuals displayed in
\eqref{eq:product-bps-components}.

The determined Hull graph is horizontal on the product locus. Indeed,
\(
 \Theta^u=\Theta^{\rm LC}(g)+C_g(\dd_Xb_2)
\)
has no circle component, and the same is true of
$A=(\Theta^u)^{[0]}+a_1$. Hence
\begin{equation}\label{eq:product-horizontal-flux}
 H=\dd_Xb_2-\frac14\jmath_\epsilon
 \mathcal T_3((\Theta^u)^{[0]},a_1)
\end{equation}
is horizontal. The same argument applies to $\Theta_H=\Theta^{\rm LC}(g)+C_g(H)$.

Differentiating $\omega\wedge\rho=0$ gives
\[
 \Sigma_v=-2\dd\omega\wedge\rho
          =2\omega\wedge V,
\]
with $V$ as in the proof of Proposition~\ref{prop:product-bps-dictionary}.
Together with \eqref{eq:product-hodge}, this gives
\eqref{eq:product-volume-origin} and the relation between $\Sigma$ and
$T_v$.

For the horizontal torsion equation, decompose
\[
 \dd\omega=\beta+\lambda\wedge\omega,
 \qquad
 \omega\wedge\beta=0.
\]
On real three-forms the $SU(3)$-equivariant Hodge operator satisfies
\begin{equation}\label{eq:product-su3-hodge}
 *_6\beta=\mathcal J\beta,
 \qquad
 *_6(\lambda\wedge\omega)
 =-\mathcal J(\lambda\wedge\omega).
\end{equation}
These are the standard $SU(3)$-equivariant Hodge identities on the primitive and $\Lambda^1\wedge\omega$ summands.
Since
\[
 w^{-1}C_h=(\lambda-\dd\Phi)\wedge\omega^2,
\]
one has $b_\omega=\lambda-\dd\Phi$, and
\eqref{eq:product-su3-hodge} gives
\begin{equation}\label{eq:product-torsion-inverse-app}
 T_h=S_H+2\mathcal J(b_\omega\wedge\omega).
\end{equation}
On the integrable balanced locus, $T_h=0$ is therefore equivalent to
\[
 H=i(\partial-\bar\partial)\omega=-\dd^c\omega.
\]

For the gauge equation, a real two-form decomposes as
\[
 \Lambda^2_{\mathbb R}
 =\mathbb R\omega\oplus\Lambda^{1,1}_{0,\mathbb R}
  \oplus(\Lambda^{2,0}\oplus\Lambda^{0,2})_{\mathbb R}.
\]
The first two summands wedge trivially with $\widehat\rho$, while
\[
 F^{2,0}\wedge\widehat\rho
 =-\frac1{2i}F^{2,0}\wedge\overline\Psi,
 \qquad
 F^{0,2}\wedge\widehat\rho
 =\frac1{2i}F^{0,2}\wedge\Psi.
\]
The two volume-wedge maps are isomorphisms onto their targets. Thus, for
a real compact connection,
\[
 F_A\wedge\widehat\rho^{[0]}=0
 \quad\Longleftrightarrow\quad
 F_A^{0,2}=0,
\]
and the remaining component of the instanton equation is primitivity.
Finally, \eqref{eq:product-horizontal-flux} gives
$H^{[0]}=\dd_Xb_2^{[0]}$.

\subsection{Dependence on the representative}

In a flat complex chart $z^a=x^a+iy^a$, put
\[
 \omega_0=\dd x^1\wedge\dd y^1
          +\dd x^2\wedge\dd y^2
          +\dd x^3\wedge\dd y^3.
\]
The following variations fix the holomorphic data and change only the
real representatives.

\begin{proposition}\label{prop:nonbasic-bps}
The leading-flux, conformal-balance and Hermitian Yang--Mills metric
residuals are not constant along the fibres of the final characteristic
quotient.
\end{proposition}
\begin{proof}
For a real one-form $m$, the gerbe representative map changes the real
two-form by
\[
 \delta B_R=\frac i2\{(\dd m)^{2,0}-(\dd m)^{0,2}\}.
\]
Taking $m=(x^1)^2\dd x^2$ gives
\[
 \delta B_R=-\frac{x^1}{2}
 (\dd x^1\wedge\dd y^2+\dd y^1\wedge\dd x^2),
 \qquad
 \dd\delta B_R=-\frac12\dd x^1\wedge\dd y^1\wedge\dd x^2\neq0.
\]
Since $H^{[0]}=\dd B^{[0]}$, the leading-flux residual changes.

For balance, consider
\[
 \omega_t=(1+t f(x^1))\dd x^1\wedge\dd y^1
          +\dd x^2\wedge\dd y^2
          +\dd x^3\wedge\dd y^3.
\]
Choose a local primitive $F$ with $F'(x^1)=f(x^1)$ and take the real
gerbe representative one-form
\[
 m=-2F(x^1)\dd y^1.
\]
The transformation \eqref{eq:product-polar-gerbe} gives
\[
 \delta\omega=-\tfrac12(\dd m)^{1,1}
 =f(x^1)\dd x^1\wedge\dd y^1.
\]
This produces the displayed $\omega_t$ to first order while leaving
the holomorphic quotient data fixed.
At fixed $\Omega$, volume normalization gives
$w_t=(1+tf)^{-1/2}$, and hence
\[
 \dd(w_t\omega_t^2)
 =-t f'(x^1)(1+tf)^{-3/2}
   \dd x^1\wedge\dd x^2\wedge\dd y^2
   \wedge\dd x^3\wedge\dd y^3,
\]
which is nonzero when $f'$ is nonzero.

For the gauge equation, take $A^{0,1}=0$ and a nonzero
$Z\in\mathfrak{so}(7)$. With $h=iZ$, the Hermitian reduction
\[
 r_H=1+t(x^1)^2h+O(t^2)
\]
has compatible-connection variation
$\delta A^{1,0}=-x^1\dd z^1h$. Therefore
\[
 \delta F=Z\,\dd x^1\wedge\dd y^1,
 \qquad
 \delta F\wedge\omega_0^2=2Z\,\vol_6\neq0,
\]
while the partial connection and holomorphic bundle are unchanged.
These local variations may be multiplied by cutoff functions, so the
same tests apply on a closed background.
\end{proof}

\bibliographystyle{amsplain}
\bibliography{references}

@article{FiniteDeformations,
  author = {Ashmore, Anthony and {de la Ossa}, Xenia and Minasian, Ruben and Strickland-Constable, Charles and Svanes, Eirik Eik},
  title = {Finite deformations from a heterotic superpotential: holomorphic {Chern--Simons} and an {$L_\infty$} algebra},
  journal = {J. High Energy Phys.},
  volume = {10},
  year = {2018},
  pages = {179},
  doi = {10.1007/JHEP10(2018)179},
  eprint = {1806.08367},
  archivePrefix = {arXiv},
  note = {\href{https://doi.org/10.1007/JHEP10(2018)179}{doi:10.1007/JHEP10(2018)179}, \href{https://arxiv.org/abs/1806.08367}{arXiv:1806.08367}}
}

@article{DLSHeterotic,
  author = {{de la Ossa}, Xenia and Larfors, Magdalena and Svanes, Eirik Eik},
  title = {The infinitesimal moduli space of heterotic {$G_2$} systems},
  journal = {Commun. Math. Phys.},
  volume = {360},
  number = {2},
  year = {2018},
  pages = {727--775},
  doi = {10.1007/s00220-017-3013-8},
  eprint = {1704.08717},
  archivePrefix = {arXiv},
  note = {\href{https://doi.org/10.1007/s00220-017-3013-8}{doi:10.1007/s00220-017-3013-8},
          \href{https://arxiv.org/abs/1704.08717}{arXiv:1704.08717}}
}

@article{MSS,
  author = {McOrist, Jock and Sticka, Martin and Svanes, Eirik Eik},
  title = {The physical moduli of heterotic {$G_2$} string compactifications},
  journal = {J. High Energy Phys.},
  volume = {05},
  year = {2025},
  pages = {219},
  doi = {10.1007/JHEP05(2025)219},
  eprint = {2409.13080},
  archivePrefix = {arXiv},
  note = {\href{https://doi.org/10.1007/JHEP05(2025)219}{doi:10.1007/JHEP05(2025)219},
          \href{https://arxiv.org/abs/2409.13080}{arXiv:2409.13080}}
}

@article{BPS,
  author = {Kupka, Julian and Strickland-Constable, Charles and Svanes, Eirik Eik and Tennyson, David and Valach, Fridrich},
  title = {{BPS} complexes and {Chern--Simons} theories from {$G$}-structures in gauge theory and gravity},
  journal = {J. High Energy Phys.},
  volume = {10},
  year = {2025},
  pages = {192},
  doi = {10.1007/JHEP10(2025)192},
  eprint = {2406.03550},
  archivePrefix = {arXiv},
  note = {\href{https://doi.org/10.1007/JHEP10(2025)192}{doi:10.1007/JHEP10(2025)192},
          \href{https://arxiv.org/abs/2406.03550}{arXiv:2406.03550}}
}

@article{G2Superpotential,
  author = {{de la Ossa}, Xenia and Larfors, Magdalena and Magill, Matthew and Svanes, Eirik Eik},
  title = {Superpotential of three dimensional {$\mathcal N=1$} heterotic supergravity},
  journal = {J. High Energy Phys.},
  volume = {01},
  year = {2020},
  pages = {195},
  doi = {10.1007/JHEP01(2020)195},
  eprint = {1904.01027},
  archivePrefix = {arXiv},
  note = {\href{https://doi.org/10.1007/JHEP01(2020)195}{doi:10.1007/JHEP01(2020)195}, \href{https://arxiv.org/abs/1904.01027}{arXiv:1904.01027}}
}

@article{MSSMetric,
  author = {McOrist, Jock and Sticka, Martin and Svanes, Eirik Eik},
  title = {The heterotic {$G_2$} moduli space metric},
  journal = {J. High Energy Phys.},
  volume = {11},
  year = {2025},
  pages = {016},
  doi = {10.1007/JHEP11(2025)016},
  eprint = {2502.16093},
  archivePrefix = {arXiv},
  note = {\href{https://doi.org/10.1007/JHEP11(2025)016}{doi:10.1007/JHEP11(2025)016}, \href{https://arxiv.org/abs/2502.16093}{arXiv:2502.16093}}
}

@incollection{QuantumG2,
  author = {{de la Ossa}, Xenia and Larfors, Magdalena and Magill, Matthew and Svanes, Eirik Eik},
  title = {Quantum aspects of heterotic {$G_2$} systems},
  booktitle = {2024 MATRIX Annals, Part II},
  publisher = {Springer},
  series = {MATRIX Book Series},
  volume = {8},
  year = {2026},
  pages = {253--293},
  doi = {10.1007/978-3-032-16206-9_16},
  eprint = {2412.14715},
  archivePrefix = {arXiv},
  note = {\href{https://doi.org/10.1007/978-3-032-16206-9_16}{doi:10.1007/978-3-032-16206-9\_16}, \href{https://arxiv.org/abs/2412.14715}{arXiv:2412.14715}}
}

@article{GFDirac,
  author = {Garcia-Fernandez, Mario},
  title = {Ricci flow, {Killing} spinors, and {T}-duality in generalized geometry},
  journal = {Adv. Math.},
  volume = {350},
  year = {2019},
  doi = {10.1016/j.aim.2019.04.038},
  pages = {1059--1108},
  eprint = {1611.08926},
  archivePrefix = {arXiv},
  note = {\href{https://arxiv.org/abs/1611.08926}{arXiv:1611.08926}}
}

@article{StringAlgebroids,
  author = {Garcia-Fernandez, Mario and Rubio, Roberto and Tipler, Carl},
  title = {Holomorphic string algebroids},
  journal = {Trans. Amer. Math. Soc.},
  volume = {373},
  year = {2020},
  number = {10},
  doi = {10.1090/tran/8149},
  pages = {7347--7382},
  eprint = {1807.10329},
  archivePrefix = {arXiv},
  note = {\href{https://arxiv.org/abs/1807.10329}{arXiv:1807.10329}}
}

@article{CourantLie2,
  author = {Roytenberg, Dmitry and Weinstein, Alan},
  title = {Courant algebroids and strongly homotopy {Lie} algebras},
  doi = {10.1023/A:1007452512084},
  journal = {Lett. Math. Phys.},
  volume = {46},
  year = {1998},
  pages = {81--93},
  eprint = {math/9802118},
  archivePrefix = {arXiv},
  note = {\href{https://arxiv.org/abs/math/9802118}{arXiv:math/9802118}}
}

@article{CanonicalHermitian,
  author = {Tosatti, Valentino and Weinkove, Ben and Yau, Shing-Tung},
  title = {Taming symplectic forms and the {Calabi--Yau} equation},
  journal = {Proc. London Math. Soc. (3)},
  volume = {97},
  number = {2},
  year = {2008},
  eprint = {math/0703773},
  archivePrefix = {arXiv},
  pages = {401--424},
  doi = {10.1112/plms/pdn008},
  note = {\href{https://doi.org/10.1112/plms/pdn008}{doi:10.1112/plms/pdn008}; \href{https://arxiv.org/abs/math/0703773}{arXiv:math/0703773}}
}

@article{BBH,
  author = {Barnich, G. and Brandt, F. and Henneaux, M.},
  title = {Local {BRST} cohomology in the antifield formalism: {I}. {General} theorems},
  journal = {Commun. Math. Phys.},
  volume = {174},
  year = {1995},
  pages = {57--92},
  doi = {10.1007/BF02099464},
  eprint = {hep-th/9405109},
  archivePrefix = {arXiv},
  note = {\href{https://doi.org/10.1007/BF02099464}{doi:10.1007/BF02099464}; \href{https://arxiv.org/abs/hep-th/9405109}{arXiv:hep-th/9405109}}
}

@article{GrigorievParent,
  author = {Grigoriev, Maxim},
  title = {Parent formulations, frame-like {Lagrangians}, and generalized auxiliary fields},
  journal = {J. High Energy Phys.},
  volume = {12},
  year = {2012},
  pages = {048},
  doi = {10.1007/JHEP12(2012)048},
  eprint = {1204.1793},
  archivePrefix = {arXiv},
  note = {\href{https://doi.org/10.1007/JHEP12(2012)048}{doi:10.1007/JHEP12(2012)048}; \href{https://arxiv.org/abs/1204.1793}{arXiv:1204.1793}}
}

@article{GrigorievRudinsky,
  author = {Grigoriev, Maxim and Rudinsky, Dmitry},
  title = {Notes on the {$L_\infty$}-approach to local gauge field theories},
  journal = {J. Geom. Phys.},
  volume = {190},
  year = {2023},
  pages = {104863},
  doi = {10.1016/j.geomphys.2023.104863},
  eprint = {2303.08990},
  archivePrefix = {arXiv},
  note = {\href{https://doi.org/10.1016/j.geomphys.2023.104863}{doi:10.1016/j.geomphys.2023.104863}; \href{https://arxiv.org/abs/2303.08990}{arXiv:2303.08990}}
}

@misc{BrongersPaperI,
 author = {Brongers, Bram},
 title = {Derived deformation theory of heterotic {$G_2$} systems near the standard embedding},
 year = {2026}, eprint = {2609.35071}, archivePrefix = {arXiv},
 note = {\href{https://arxiv.org/abs/2609.35071}{arXiv:2609.35071}}
}

@misc{DLSRestrictions,
 author = {{de la Ossa}, Xenia and Larfors, Magdalena and Svanes, Eirik Eik},
 title = {Restrictions of Heterotic {$G_2$} Structures and Instanton Connections},
 year = {2017}, eprint = {1709.06974}, archivePrefix = {arXiv},
 note = {\href{https://arxiv.org/abs/1709.06974}{arXiv:1709.06974}}
}

@article{GaugeStringAlgebroids,
 author = {Garcia-Fernandez, Mario and Rubio, Roberto and Tipler, Carl},
 title = {Gauge theory for string algebroids},
 journal = {J. Differential Geom.},
 volume = {128}, number = {1}, year = {2024}, pages = {77--152},
 doi = {10.4310/jdg/1721075260}, eprint = {2004.11399}, archivePrefix = {arXiv},
 note = {\href{https://doi.org/10.4310/jdg/1721075260}{doi:10.4310/jdg/1721075260}, \href{https://arxiv.org/abs/2004.11399}{arXiv:2004.11399}}
}

@article{BCGReduction,
 author = {Bursztyn, Henrique and Cavalcanti, Gil R. and Gualtieri, Marco},
 title = {Reduction of {Courant} algebroids and generalized complex structures},
 journal = {Adv. Math.}, volume = {211}, number = {2}, year = {2007}, pages = {726--765},
 doi = {10.1016/j.aim.2006.09.008}, eprint = {math/0509640}, archivePrefix = {arXiv},
 note = {\href{https://doi.org/10.1016/j.aim.2006.09.008}{doi:10.1016/j.aim.2006.09.008}, \href{https://arxiv.org/abs/math/0509640}{arXiv:math/0509640}}
}
\end{document}